\documentclass[10pt]{amsart}
\usepackage{amsfonts}
\usepackage{amsmath,amssymb,graphicx}
\usepackage{amsthm,stackrel,bm,accents,longtable}
\usepackage{enumerate}
\usepackage{caption}
\usepackage{subcaption}
\usepackage{setspace}
\usepackage[alphabetic,initials]{amsrefs}
\usepackage[pdfborderstyle={/S/U/W 1}]{hyperref}

\newtheorem{theorem}{Theorem}
\newtheorem{lemma}[theorem]{Lemma}
\newtheorem{proposition}[theorem]{Proposition}
\newtheorem{corollary}[theorem]{Corollary}

\newtheorem{property}{Partition Property}

\theoremstyle{definition}
\newtheorem{definition}[theorem]{Definition}
\newtheorem{remark}[theorem]{Remark}

\numberwithin{theorem}{section}
\numberwithin{equation}{section}
\numberwithin{figure}{section}
\DeclareMathOperator{\Diff}{Diff}
\DeclareMathOperator{\sgn}{sgn}
\DeclareMathOperator{\Tr}{Tr}
\DeclareMathOperator{\tr}{tr}

\newcommand{\abs}[1]{{\left\lvert{#1}\right\rvert}}

\newcommand{\smallang}[1]{{\langle{#1}\rangle}}

\newcommand{\F}{\mathcal{F}}

\newcommand{\bo}{\mathrm{b}}

\newcommand{\Diffb}[1][]{\operatorname{Diff}_{\bo}^{#1}}
\newcommand{\Psib}[1][]{\operatorname{\Psi}_{\bo}^{#1}}

\newcommand{\Hb}{H_{\bo}}

\DeclareMathOperator{\Dom}{Dom}

\newcommand{\CI}{\mathcal{C}^\infty}

\newcommand{\pa}{{\partial}}
\newcommand{\del}{\ensuremath{\partial}}
\newcommand{\ep}{\ensuremath{\varepsilon}}

\newcommand{\RR}{\mathbb{R}}

\newcommand{\Lap}{\Delta}
\newcommand{\cutoff}{\rho}

\DeclareMathOperator{\WF}{WF}

\newcommand{\dcal}{\mathcal{D}}

\newcommand{\loc}{\text{loc}}
\renewcommand{\mod}[1]{\ensuremath{ \left( \mathrm{mod} \ #1 \right)}}

\newcommand{\bbR}{\ensuremath{\mathbb{R}}} % blackboard bold R
\newcommand{\bbZ}{\ensuremath{\mathbb{Z}}} % blackboard bold Z

\newcommand{\calC}{\ensuremath{\mathcal{C}}} % calligraphic C
\newcommand{\calD}{\ensuremath{\mathcal{D}}} % calligraphic D
\newcommand{\calE}{\ensuremath{\mathcal{E}}} % calligraphic E
\newcommand{\calK}{\ensuremath{\mathcal{K}}} % calligraphic K
\newcommand{\calU}{\ensuremath{\mathcal{U}}} % calligraphic U
\newcommand{\calV}{\ensuremath{\mathcal{V}}} % calligraphic V
\newcommand{\calW}{\ensuremath{\mathcal{W}}} % calligraphic W

\renewcommand{\vec}[1]{\ensuremath{\mathbf{#1}}} % vector

\newcommand{\barZ}{\ensuremath{\overline{\mathbf{Z}}}} % overline vec Z

\newcommand{\bsfD}{\ensuremath{{\bm{\mathsf{D}}}}} % bold sans serif D
\newcommand{\bsfG}{\ensuremath{{\bm{\mathsf{G}}}}} % bold sans serif G
\newcommand{\bsfJ}{\ensuremath{{\bm{\mathsf{J}}}}} % bold sans serif J
\newcommand{\bsfK}{\ensuremath{{\bm{\mathsf{K}}}}} % bold sans serif K
\newcommand{\bsfL}{\ensuremath{{\bm{\mathsf{L}}}}} % bold sans serif L
\newcommand{\bsfQ}{\ensuremath{{\bm{\mathsf{Q}}}}} % bold sans serif Q
\newcommand{\bsfV}{\ensuremath{{\bm{\mathsf{V}}}}} % bold sans serif V
\newcommand{\bsfJdot}{\ensuremath{\skew{3.45}{\dot}{{\bm{\mathsf{J}}}}}} % bold
\newcommand{\bmw}{\ensuremath{{\bm{w}}}} % bold math w
\newcommand{\bmD}{\ensuremath{{\bm{D}}}} % bold math D
\newcommand{\bmE}{\ensuremath{{\bm{E}}}} % bold math E
\newcommand{\bmU}{\ensuremath{{\bm{U}}}} % bold math U
\DeclareMathOperator{\LSp}{LSp} % length spectrum
\DeclareMathOperator{\dist}{dist} % distance function
\DeclareMathOperator{\Id}{Id} % identity operator
\DeclareMathOperator{\pr}{pr} % projection operator
\DeclareMathOperator{\inj}{inj} % injectivity radius
\DeclareMathOperator{\length}{length} % length of a path
\DeclareMathOperator{\Hess}{Hess} % Hessian
\DeclareMathOperator{\grad}{grad} % gradient
\DeclareMathOperator{\ind}{ind} % index
\DeclareMathOperator{\End}{End} % endomorphisms

\let\det=\relax
\DeclareMathOperator{\det}{det}

\newcommand{\defeq}{\ensuremath{\stackrel{\mathrm{def}}{=}}} % definition =
\renewcommand{\del}{\partial} % partial derivative
\newcommand{\upc}{\ensuremath{\mathrm{c}}} % upright c
\newcommand{\upb}{\ensuremath{\mathrm{b}}} % upright b
\newcommand{\upt}{\ensuremath{\mathrm{t}}} % upright t
\newcommand{\To}{\ensuremath{\longrightarrow}} % mapping arrow

\newcommand*{\underdot}[1][]{\ensuremath{%
  \skew{14}{\underaccent{\dot}}{#1}}} % underdot command

\newcommand{\Tbstar}{{}^{\upb} T^*}
\newcommand{\Sbstar}{{}^{\upb} S^*}
\newcommand{\Tdot}{\ensuremath{\dot{T}}}

\newcommand{\coneT}{\ensuremath{{}^{\text{cone}} T}} % cone bundle

\newcommand{\dtilde}[1][t]{\ensuremath{\stackbin[{\bsfD}]{#1}{\sim}}}

\newcommand{\gtilde}[1][t]{\ensuremath{\stackbin[{\bsfG}]{#1}{\sim}}}
\newcommand{\regtilde}[1][t]{\ensuremath{\stackbin{#1}{\sim}}}

\newcommand{\rtLap}{\ensuremath{\sqrt{ \smash[b]{ \Delta_g }}}} % square root of
\newcommand{\IH}{\ensuremath{\mathit{IH}}} % Sobolev-based Lag'n distributions
\newcommand{\IB}{\ensuremath{\mathit{IB}}} % Besov-based Lag'n distributions

\newcommand{\HD}{\ensuremath{|\Omega|^{\frac{1}{2}}}} % half-densities
\newcommand{\scal}{\ensuremath{\mathrm{scal}}} % scalar operator

\theoremstyle{plain}
\newtheorem*{maintheorem}{Main Theorem}

\title{The diffractive wave trace on manifolds with conic singularities}

\author[G.A.~Ford]{G.~Austin Ford}
\address{Stanford University}
\email{austin.ford@math.stanford.edu}

\author[J.~Wunsch]{Jared Wunsch}
\address{Northwestern University}
\email{jwunsch@math.northwestern.edu}

\begin{document}

\begin{abstract}
  Let $(X,g)$ be a compact manifold with conic singularities.  Taking $\Delta_g$
  to be the Friedrichs extension of the Laplace-Beltrami operator, we examine
  the singularities of the trace of the half-wave group
  $e^{- i t \sqrt{ \smash[b]{\Delta_g}}}$ arising from strictly diffractive
  closed geodesics.  Under a generic nonconjugacy assumption, we compute the
  principal amplitude of these singularities in terms of invariants associated
  to the geodesic and data from the cone point.  This generalizes the classical
  theorem of Duistermaat--Guillemin on smooth manifolds and a theorem of
  Hillairet on flat surfaces with cone points.
\end{abstract}

\maketitle
% spacing during writing process \onehalfspace

\setcounter{section}{-1}

% INTRODUCTION

% fw-conetrace-0.tex Section 0:  Introduction

\section{Introduction}
\label{sec:introduction}

% \af{Big singularity at $t = 0$?}

In this paper, we consider the trace of the half-wave group
$\calU(t) \defeq e^{-it\sqrt{ \smash[b]{\Delta_g} }}$ on a compact manifold with
conic singularities $(X,g)$.  Our main result is a description of the
singularities of this trace at the lengths of closed geodesics undergoing
diffractive interaction with the cone points.  Under the generic assumption that
\begin{equation*}
  \label{eq:nonconjugacy-assumption-cone-points}
  \text{the cone points of $X$ are pairwise nonconjugate along the
    geodesic flow},
\end{equation*}
the resulting singularity at such a length $t = L$ has the oscillatory integral
representation
\begin{equation*}
  \int_{\bbR_\xi} e^{-i (t-L) \cdot \xi} \, a(t,\xi) \, d\xi ,
\end{equation*}
where the amplitude $a$ is to leading order
\begin{multline*}
  a(t,\xi) \sim L_0 \cdot (2\pi)^{\frac{kn}{2}} \, e^{\frac{i k \pi (n-3)}{4}}
  \cdot \chi(\xi) \, \xi^{-\frac{k(n-1)}{2}} \\
  \mbox{} \times \left[ \prod_{j = 1}^k i^{-m_{\gamma_j}} \cdot
    \bmD_{\alpha_j}(q_j,q_j') \cdot
    \dist_g^{\gamma_{j}}(Y_{\alpha_{j+1}},Y_{\alpha_{j}})^{-\frac{n-1}{2}} \cdot
    \Theta^{-\frac{1}{2}}(Y_{\alpha_{j}} \to Y_{\alpha_{j+1}}) \right]
\end{multline*}
as $|\xi| \to \infty$ and the index $j$ is cyclic in $\{1,\ldots,k\}$.  Here,
$n$ is the dimension of $X$ and $k$ the number of diffractions along the
geodesic, and $\chi$ is a smooth function supported in $[1,\infty)$ and equal to
$1$ on $[2,\infty)$.  $L_0$ is the primitive length, in case the
geodesic is an iterate of a shorter closed geodesic.  The product is over the diffractions undergone by the
geodesic, with $\bmD_{\alpha_j}$ a quantity determined by the functional
calculus of the Laplacian on the link of the $j$-th cone point $Y_{\alpha_j}$,
the factor $\Theta^{-\frac{1}{2}}(Y_{\alpha_{j}} \to Y_{\alpha_{j+1}})$ is
(at least on a formal level)
the
determinant of the differential of the flow between the $j$-th
and $(j+1)$-st cone points, and $m_{\gamma_j}$ is the Morse index of the
geodesic segment $\gamma_j$ from the $j$-th to $(j+1)$-st cone points.  All of
these factors are described in more detail below.

To give this result some context, we recall the known results for the
Laplace-Beltrami operator $\Lap_g = d_g^* \circ d$ on a \emph{smooth}
($\calC^\infty$) compact Riemannian manifold $(X,g)$.  In this setting, there is
a countable orthonormal basis for $L^2(X)$ comprised of eigenfunctions
$\varphi_j$ of $\Lap_g$ with eigenvalues
$\left\{\lambda_j^2\right\}_{j=0}^\infty$ of finite multiplicity and tending to
infinity.  The celebrated trace formula of Duistermaat and Guillemin
\cite{Duistermaat-Guillemin}, a generalization of the Poisson summation formula
to this setting, establishes that the quantity
\begin{equation*}
  \sum_{j=0}^\infty e^{-it\lambda_j}
\end{equation*}
is a well-defined distribution on $\bbR_t$.  Moreover, it satisfies the
``Poisson relation'':  it is singular only on the \emph{length spectrum} of
$(X,g)$,
\begin{equation*}
  \LSp(X,g) \defeq \left\{ 0 \right\} \cup \left\{ \pm L \in \bbR : \text{$L$ is
      the length of a closed geodesic in $(X,g)$} \right\} .
\end{equation*}
(This was shown independently by Chazarain \cite{Chazarain}; see also
\cite{Colin}.)  Subject to a nondegeneracy hypothesis, the singularity at the
length $t = \pm L$ of a closed geodesic has a specific leading form encoding the
geometry of that geodesic---the formula involves the linearized Poincar\'e map
and the Morse index of the geodesic.  The proofs of these statements center
around the identification
\begin{equation*}
  \sum_{j=0}^\infty e^{-it\lambda_j} = \Tr \calU(t)
\end{equation*}
where $\calU(t) = e^{-it\sqrt{\smash[b]{\Delta_g}}}$ is (half of) the propagator
for solutions to the wave equation on $X$; in particular, $\calU(t)$ is a
Fourier integral operator.

In this paper, we prove an analogue of the Duistermaat--Guillemin trace
formula on compact \emph{manifolds with conic singularities} (or ``conic
manifolds''), generalizing results of Hillairet
\cite{Hillairet:Contribution} from the case of flat surfaces with conic
singularities.  We again consider the trace $\Tr \calU(t)$, a spectral
invariant equal to $\sum_{j=0}^\infty e^{-it \lambda_j}$.  The Poisson
relation is complicated here by the fact that closed geodesics may have two
different meanings on a manifold with conic singularities.  On the one
hand, we may regard a geodesic striking a cone point as being legitimately
continued by any other geodesic emanating from a cone point.  On the other
hand, we may only take those geodesics which are (locally) uniformly
approximable by families of geodesics that miss the cone point.  It turns
out that singularities of solutions to the wave equation can propagate
along all geodesics in the former, broader interpretation, and this is the
phenomenon of ``diffraction.''  It was extensively (and explicitly) studied
for cones admitting a product structure by Cheeger and Taylor
\cites{Cheeger-Taylor1,Cheeger-Taylor2}.  We refer to geodesics of this
broader type as \emph{diffractive geodesics,} and call them \emph{strictly
  diffractive} if they are \emph{not} (locally) approximable by ordinary
geodesics.  We refer to the (locally) approximable geodesics as
\emph{geometric geodesics}.  In \cite{Melrose-Wunsch1}, Melrose and the
second author showed that singularities of solutions to the wave equation
propagate along diffractive geodesics, although the singularities at
strictly diffractive geodesics are generically weaker than the
singularities at the geometric geodesics.  In \cite{Wunsch:Poisson} the
second author used this fact to prove that the singularities of the wave
trace on a conic manifold are a subset of the length spectrum $\LSp(X,g)$,
consisting again of zero and the positive and negative lengths of closed
geodesics.  A new wrinkle in this case is the fact that singularities at
closed strictly diffractive geodesics are weaker than the singularities at
ordinary or geometric closed geodesics, reflecting the analogous phenomenon
for the propagation of singularities.

Expanding upon these previous works, we describe explicitly in this paper the
leading order behavior of the singularities at lengths of closed diffractive
geodesics.  We must assume that these geodesics are isolated and appropriately
nondegenerate, essentially in the sense that no pair of cone points are mutually
conjugate along the geodesic flow of the manifold.  Note that these
hypotheses are generically satisfied and moreover, the diffractive closed
geodesics are generically the only closed geodesics apart from interior
ones, where the contribution to the trace is already known \cite{Duistermaat-Guillemin}.

To describe these leading
order asymptotics, we need to briefly set up some of the framework.  Let us
suppose that $\gamma$ is a diffractive geodesic undergoing diffraction with $k$
cone points and repeating with period $T$ (so $T$ is the primitive length of
$\gamma$). We let $\gamma_j$ denote the segment of geodesic from the $j$-th to
the $(j+1)$-st cone point, and we denote by $m_{\gamma_j}$ the Morse index of
each of these segments.  As the link of the $j$-th cone point is a Riemannian
manifold $Y_j$, we may consider the operator
\begin{equation*}
  \nu_j \defeq \sqrt{ \Delta_{Y_j} + \left( \frac{2-n}{2} \right)^2 }
\end{equation*}
on the link as well as the operators in its functional calculus.  The
ordinary propagation of singularities implies that the kernel of the
half-Klein-Gordon propagator $e^{-it\nu_j}$ is smooth away from points
distance $\abs{t}$ apart.  In particular, if $q_j$ and $q_j'$ are the
initial and terminal points on $Y_j$ of the geodesic segments $\gamma_j$
and $\gamma_{j-1}$ respectively along a diffractive geodesic, then the
Schwartz kernel of the half-Klein-Gordon propagator $e^{-i \pi\nu_j}$ is
smooth near $(q_j,q_j')$.  We write $\bmD_{\alpha_j}(q_j,q_j')$ for this
value $\calK \! \left[ e^{-i \pi\nu_j} \right] \! (q_j,q_j')$.  Finally,
for each segment $\gamma_j$ we define an invariant $\Theta(Y_{\alpha_j} \to
Y_{\alpha_{j+1}})$ in Section~\ref{section:Jacobi} below, letting $j$ range
cyclically over $\{1,\ldots,k\}$.  This invariant can be viewed in more
than one way: it looks formally like the determinant of the differential of the
flow from one cone point to the next, but owing to the singular of nature
of this flow, we employ a definition in terms
of (singular) Jacobi fields; alternatively, it measures the tangency of the
intersection along $\gamma_j$ of the geodesic ``spheres'' centered at the
successive cone points, via an interpretation in terms of Wronskians of
Jacobi fields vanishing at successive cone points (see
Section~\ref{sec:proof-theorem-int-ampl} for the latter interpretation).

\begin{maintheorem}
  % This fixes the link.
  \phantomsection
  \label{theorem:main}
  For $t$ sufficiently close to $L$, the wave trace $\Tr \calU(t)$ is a conormal
  distribution with respect to $t=L$ of the form
  \begin{equation*}
    \int_{\bbR_\xi} e^{-i (t-L) \cdot \xi} \, a(t,\xi) \, d\xi .
  \end{equation*}
  Its amplitude $a \in S^{-\frac{k(n-1)}{2}}(\bbR_t \times \bbR_\xi)$ has the
  leading order asymptotics
  \begin{multline*}
    a(t,\xi) \equiv L \cdot (2\pi)^{\frac{kn}{2}} \, e^{\frac{i k \pi (n-3)}{4}}
    \cdot \chi(\xi) \, \xi^{-\frac{k(n-1)}{2}} \\
    \mbox{} \times \left[ \prod_{j = 1}^k i^{-m_{\gamma_j}} \cdot
      \bmD_{\alpha_j}(q_j,q_j') \cdot
      \dist_g^{\gamma_{j}}(Y_{\alpha_{j+1}},Y_{\alpha_{j}})^{-\frac{n-1}{2}}
      \cdot \Theta^{-\frac{1}{2}}(Y_{\alpha_{j}} \to Y_{\alpha_{j+1}}) \right]
  \end{multline*}
  modulo elements of $S^{-\frac{k(n-1)}{2} - \frac{1}{2} + \varepsilon}$ for any
  $\varepsilon > 0$, where $\chi(\xi) \in \mathcal{C}^\infty(\bbR)$ is supported
  in $[1,\infty)$ and equal to $1$ for $\xi > 2$.
\end{maintheorem}

The rest of the paper is organized as follows.  In
Section~\ref{sec:conic-geometry} we review the geometry of manifolds with conic
singularities, in particular the geometry of geodesics and Jacobi fields.
Section~\ref{section:singleproduct} contains the calculation of the principal
amplitude of the diffractive part of the half-wave propagator near the cone
point of a metric cone, and Section~\ref{section:single} generalizes this
calculation to the wider class of conic manifolds.  In
Section~\ref{sec:amplitude-mult-diff}, we use the previous work to calculate the
principal amplitude of a multiply-diffracted wave on a manifold with (perhaps
multiple) cone points, and the proofs of the required results make up
Sections~\ref{sec:proof-theorem-int-ampl} and \ref{sec:proof-theorem-mult-diff}.
Finally, using a microlocal partition of unity developed in
Section~\ref{section:partition}, we compute the trace of the half-wave group
along the diffractive closed geodesics in Section~\ref{sec:diff-wave-trace}.  At
the end, we include Appendix~\ref{sec:lagrangian-dists-amplitudes} as a brief
review of the theory of Lagrangian distributions and their amplitudes.

\subsection*{Notation}
\label{sec:notation}

We use the following pieces of notation throughout this work.
\begin{enumerate}[\hspace*{.25em}$\bullet$]
\item If $\left\{ \mathcal{Y}^s : s \in \bbR \right\}$ is an $\bbR$-filtered
  collection of vector spaces with inclusions
  $\mathcal{Y}^s \subseteq \mathcal{Y}^t$ if $s < t$, then we write
  \begin{equation*}
    \mathcal{Y}^{t - 0} \defeq \bigcup_{s < t} \mathcal{Y}^s .
  \end{equation*}
  Similarly, if the inclusions are of the form
  $\mathcal{Y}^s \supseteq \mathcal{Y}^t$ if $s < t$, then we write
  \begin{equation*}
    \mathcal{Y}^{t + 0} \defeq \bigcap_{s > t} \mathcal{Y}^s .
  \end{equation*}
\item If $\pr : E \To Z$ is a vector bundle over a manifold $Z$ and
  $V \subseteq E$ is a subset of this bundle, then we write
  $V^\flat \defeq \pr[V]$ for the projection of this subset to the base
  manifold.
\item If $Z$ is a smooth manifold, then we let
  $\dot{T}^* Z \defeq T^*Z \setminus \vec{0}$ be its punctured cotangent bundle,
  where $\vec{0} \subseteq T^*Z$ is the zero section.
\item We write
  $\mathcal{F}[u](\xi) \defeq (2\pi)^{-\frac{n}{2}} \int_{\bbR^n} e^{- i x \cdot
    \xi} \, u(x) \, dx$ for the unitary normalization of the Fourier transform.
\end{enumerate}

\subsection*{Acknowledgements}
The authors are pleased to thank Andrew Hassell, Luc Hillairet, Richard Melrose,
Andr\'as Vasy, and Steve Zelditch for helpful conversations.  GAF gratefully
acknowledges support from NSF Postdoctoral Fellowship \#1204304, and JW
gratefully acknowledges support from NSF Grant DMS-1265568.

% SECTION 1

% fw-conetrace-1.tex Section 1:  Conic geometry

\section{Conic geometry}
\label{sec:conic-geometry}

In this section we recall the basic notions of the geometry of manifolds with
conic singularities and the analysis of distributions on them.  A conic manifold
of dimension $n$ is a Riemannian manifold with boundary $(X,g)$ whose metric is
nondegenerate over the interior $X^\circ$ but singular at the boundary $\del X$.
Near the boundary it is assumed to take the form
\begin{equation*}
  g = dx^2 + x^2 \, h(x,dx;y,dy)
\end{equation*}
for some boundary defining function\footnote{$x$ is a defining function for the
  boundary if $\del X = \{x = 0\}$ and $dx \big\vert_{\del X} \neq 0$.} $x$,
where $h$ is a smooth symmetric tensor of rank $2$ restricting to be a metric on
$\del X$.  Writing $Y$ for the boundary $\del X$, this restriction has the
effect of reducing each of the components $Y_\alpha$ of $Y$ to a ``cone point.''
Thus, we refer to these components $Y_\alpha$ as the ``cone points'' of our
manifold throughout the rest of this work.

\begin{figure}
  \centering

  \begin{minipage}[b]{.5\linewidth}
    \centering
    \includegraphics{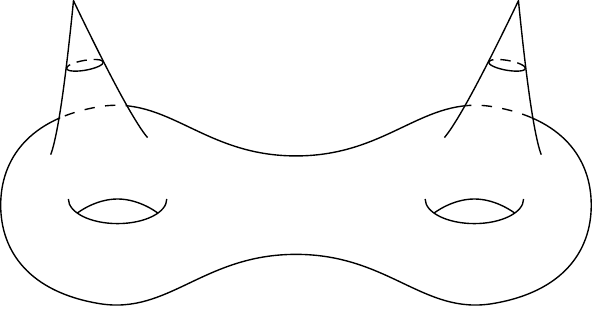}
    \subcaption{The metric picture}
    \label{fig:mfld-conic-sings-metric}
  \end{minipage}%
  \begin{minipage}[b]{.5\linewidth}
    \centering
    \includegraphics{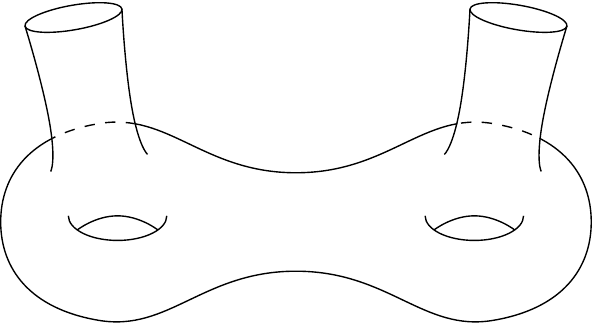}
    \subcaption{The topological (or ``blown-up'') picture}
    \label{fig:mfld-conic-sings-top}
  \end{minipage}

  \caption{A manifold with conic singularities}
  \label{fig:mfld-conic-sings}
\end{figure}

In Section~1 of \cite{Melrose-Wunsch1} the authors show the existence of a
defining function $x$ decomposing a collar neighborhood of the boundary into a
product $[0,\varepsilon)_x \times Y$ such that in the associated local product
coordinates $(x,y)$ the metric decomposes as a product
\begin{equation}
  \label{semiproduct}
  g = dx^2 + x^2 \, h(x,y,dy) .
\end{equation}
In particular, when written in these coordinates $h$ is a smooth family in $x$
of metrics on $Y$.
% In Section~1 of \cite{Melrose-Wunsch1} the authors show the existence of a
% product decomposition of a collar neighborhood of $Y$ into
% \begin{equation*}
%   [0,\ep)_x \times Y
% \end{equation*}
% with associated coordinates $(x,y)$ so that locally
% \begin{equation}\label{semiproduct}
%   g =dx^2+x^2 h(x,y,dy)
% \end{equation}
% i.e., $h$ is now a smooth family in $x$ of metrics on $Y.$
Equivalently, the family of curves $[0,\ep) \times \{y_0\}$ parametrized by
$y_0 \in Y_\alpha$ are geodesics reaching the boundary component $Y_\alpha$ and
foliating the associated collar neighborhood.  These are moreover the
\emph{only} geodesics reaching $Y_\alpha$.  \emph{Henceforth, we will always
  assume that the metric has been reduced to this normal form.}  We write
\begin{equation*}
  h_0 \defeq h\big\vert_{x=0} \quad \text{and} \quad h_\alpha \defeq h
  \big\vert_{Y_\alpha}
\end{equation*}
for the induced metric on the boundary as a whole and an individual boundary
component $Y_\alpha$ respectively.  Additionally, we define $x_* \defeq \sup x$
to be the supremum of the value of the designer boundary defining function,
i.e., an upper bound on the distance from $\del X$ for which $x$ is defined.  We
also write $x_\alpha$ for the restriction of $x$ to the corresponding connected
neighborhood of the cone point $Y_\alpha$; it is thus a designer boundary
defining function for this cone point.

\subsection{Operators and domains}
\label{sec:operators-domains}

Let $(z_1,\ldots,z_n)$ be a local system of coordinates on the interior
$X^\circ$.  We consider the Friedrichs extension of the Laplace operator
\begin{equation*}
  \Lap_g \defeq \frac{1}{\sqrt{\smash[b]{\det g(z)}}} \sum_{j,k=1}^n D_{z_j} \,
  g^{jk}(z) \sqrt{ \det \smash[b]{g(z)}} \, D_{z_k} 
\end{equation*}
from the domain $\mathcal{C}_\upc^\infty(X^\circ)$.  Here,
$D_{z_j} = \frac{1}{i} \, \del_{z_j}$ is the Fourier-normalization of the
$z_j$-derivative.  Using our preferred product coordinates $(x,y)$ near the
boundary, we compute
\begin{equation}
  \label{conelaplacian}
  \Delta_g = D_x^2 -\frac{i \left[(n-1) + x \, e(x) \right]}{x} \, D_x +
  \frac{1}{x^2} \, \Delta_{h(x)},
\end{equation}
where $e(x)$ is the function
\begin{equation*}
  e(x) = \frac{1}{2} \frac{\pa\log\det h(x)}{\pa x}= \frac{1}{2} \tr
  \left[h^{-1}(x) \, \frac{\pa h(x)}{\pa x} \right] .
\end{equation*}
This is the Laplace operator considered as an operator on scalars.  In the
sequel, however, we will primarily work with the version of the operator
$\Lap_g$ acting on half-densities.  To define this, we trivialize the
half-density bundle using the convention that the metric half-density
$\omega_g$, which in the local $z$-coordinates is
\begin{equation*}
  \omega_g = \left[ \det g(z) \right]^{\frac{1}{4}} \left| dz_1
    \wedge \cdots \wedge dz_n \right|^\frac{1}{2} ,
\end{equation*}
is annihilated by $\Lap_g$.  In other words, the action of $\Lap_g$ on a
general half-density $f(z) \, \omega_g$ is
\begin{equation*}
  \Lap_g \left[f(z) \, \omega_g \right] = \left[ \frac{1}{\sqrt{\smash[b]{\det
          g(z)}}} \sum_{j,k=1}^n D_{z_j} \, g(z)^{jk} \sqrt{ \smash[b]{ \det
        g(z)}} \, D_{z_k} f(z) \right] \omega_g .
\end{equation*}
Note that near the boundary we have
\begin{equation}
  \label{metrichalfdensity}
  \omega_g = x^{\frac{n-1}{2}} \left[ \det h(x) \right]^{\frac{1}{4}} \left| dx
    \wedge dy_1 \wedge \cdots \wedge dy_{n-1} \right|^\frac{1}{2}, 
\end{equation}
with $h(x)$ the family of induced metrics on the boundary.

From $\Delta_g$ we construct its complex powers $\Delta_g^z$ via the functional
calculus.  We use the domains of the real powers,
\begin{equation*}
  \calD_s \defeq \Dom \! \left[ \Delta_g^{s/2} : L^2\! \left(X ; \HD(X)\right)
    \To L^2 \! \left( X ; \HD(X) \right) \right] ,  
\end{equation*}
as the principal regularity spaces in this work.  Here,
$L^2 \! \left( X ; |\Omega|^\frac{1}{2}(X) \right)$ are the $L^2$-half-densities
on $X$, so each domain is a space of distributional
half-densities.  % We also write
% \begin{equation*}
%   \calD_\infty \defeq \bigcap_{s \in \bbR} \calD_s \quad \text{and} \quad
%   \calD_{-\infty} \defeq \bigcup_{s \in \bbR} \calD_s
% \end{equation*}
% for the collection of residual distributions and the filtered collection of
% all distributions respectively.

An alternate characterization of these domains comes in terms of b-Sobolev
spaces, which we now recall.  First, set $\calV_\upb(X)$ to be the Lie algebra
of all vector fields on $X$ tangent to $\del X$, and let $\Diff_\upb^*(X)$
denote the filtered algebra of differential operators generated by these vector
fields over $\CI(X)$.  For $m \in \bbZ_{\geqslant 0}$ we define the b-Sobolev
space
\begin{equation*}
  \Hb^m(X) = \left\{ u \in L^2 \! \left( X ;
      \HD(X) \right) : A u \in L^2 \! \left( X ;
      \HD(X) \right)
    \text{ for all } A \in \Diffb[m](X) \right\} .
\end{equation*}
More generally, we define the b-Sobolev spaces $H^s_\upb(X)$ for all real orders
$s$ by either interpolation and duality or by substituting Melrose's b-calculus
of pseudodifferential operators $\Psib[s](X)$ for the differential
b-operators---see \cite{Melrose:APS} for further details on the latter method.

\begin{proposition}[\cite{Melrose-Wunsch1}\footnote{We use a different
    convention for the density with respect to which $L^2$ is defined than was
    used in \cite{Melrose-Wunsch1}.  The b-weight $\frac{dx}{x} \, dy$ was used
    in that work rather than the metric weight $\omega_g^2$.}]
  \label{thm:domains-b-Sobolev}
  For $\abs{s} < \frac{n}{2}$ there is an identification
  \begin{equation*}
    \dcal_s = x^{-s} \, \Hb^s(X).
  \end{equation*}
\end{proposition}

It further follows from the analysis in \cite{Melrose-Wunsch1} that for every
$s \in \bbR$ we have%\jw{Check this part!}
\begin{equation*}
  \rtLap : x^{-1} \Hb^s(X) \To \Hb^{s-1}(X) \ \text{if $n > 2$},
\end{equation*}
while in the case $n=2$ (which is nearly always a borderline case in such
computations),
\begin{equation*}
  \rtLap : x^{-1+\ep} \, \Hb^s(X) \To \Hb^{s-1}(X) \text{ for all $\varepsilon >
    0$ (when $n = 2$)}.
\end{equation*}
Moreover, it follows from the more detailed description of these complex powers
$\Delta_g^z$ in \cite{Loya} that $\rtLap$ is \emph{microlocal} over the interior
manifold $X^\circ$:  its Schwartz kernel is that of a pseudodifferential
operator over compact sets in $X^\circ \times X^\circ$.  This is a fact which
will be implicitly used in our analysis below.

From the Laplace-Beltrami operator $\Delta_g$ we construct the d'Alembertian (or
wave operator) acting on the spacetime $\bbR \times X$,
\begin{equation*}
  \Box_g \defeq D_t^2 - \Lap_g,
\end{equation*}
and we consider this operator acting on half-densities on $X$ lifted to the
product spacetime.  We define the \emph{half-wave group} $\calU(t)$ as
\begin{equation*}
  \calU(t) \defeq  e^{-it\rtLap},
\end{equation*}
again considered as acting on half-densities.  As usual, note that
$\Box_g \!  \left[ \calU(t) \mu \right] = 0$ for all
$\mu \in L^2 \! \left( X ; \HD(X) \right)$ (or more general distributional
half-densities $\mu$).  We remark that conjugating $\calU(t)$ to get between
scalars and half-densities has no effect on the overall trace of the group, and
hence the introduction of half-densities is merely for computational convenience
and clarity.

\subsection{Diffractive and geometric geodesics}
\label{sec:diff-geom-geodesics}

Two different notions of geodesic exist on a conic manifold $X$, one more
restrictive and one less restrictive.  We now recall these, following the
exposition from \cite{Baskin-Wunsch:Resolvent}.
\begin{definition}
  Suppose $\gamma$ is a \emph{broken geodesic}, i.e., a union of a finite number
  of closed, oriented geodesic segments $\gamma_1, \ldots, \gamma_N$ in $X$.
  Let $\gamma_j$ be parametrized by the interval $[T_j,T_{j+1}]$.
  \begin{enumerate}[(AA)]
  \item[({$\bm{\mathsf{D}}$})] The curve $\gamma$ is a \emph{diffractive
      geodesic} in $X$ if
    \begin{itemize}
    \item[(i)] all end points except possibly the initial point $\gamma_1(T_1)$
      and the final point $\gamma_N(T_{N+1})$ of $\gamma$ lie in the boundary
      $Y \defeq \del X$, and
    \item[(ii)] the intermediate terminal points $\gamma_j(T_{j+1})$ lie in the
      same boundary component as the initial points $\gamma_{j+1}(T_{j+1})$ for
      each $j = 1,\ldots,N-1$.
    \end{itemize}
  \item[({$\bm{\mathsf{G}}$})] The curve $\gamma$ is a \emph{(partially)
      geometric geodesic} if it is a diffractive geodesic such that for some
    $j = 1,\ldots,N-1$ the intermediate terminal point $\gamma_j(T_{j+1})$ and
    the initial point $\gamma_{j+1}(T_{j+1})$ are connected by a geodesic of
    length $\pi$ in the boundary component $Y_\alpha$ in which they lie (with
    respect to the boundary metric $h_0 \defeq h \big\vert_Y$).  If this is true
    for all $j = 1,\ldots,N-1$, then we call $\gamma$ a \emph{strictly geometric
      geodesic}; if it is never true for $j = 1,\ldots,N-1$, then $\gamma$ is a
    \emph{strictly diffractive geodesic}.
  \end{enumerate}
\end{definition}
\begin{figure}
  \centering
  \begin{minipage}[b]{.5\linewidth}
    \centering
    \includegraphics{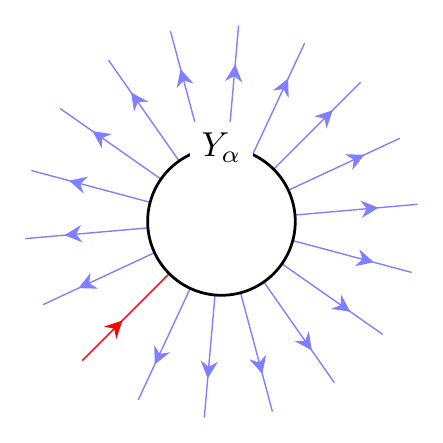}
    \subcaption{Diffractive geodesics}
    \label{fig:diffractive-geodesics}
  \end{minipage}%
  \begin{minipage}[b]{.5\linewidth}
    \centering
    \includegraphics{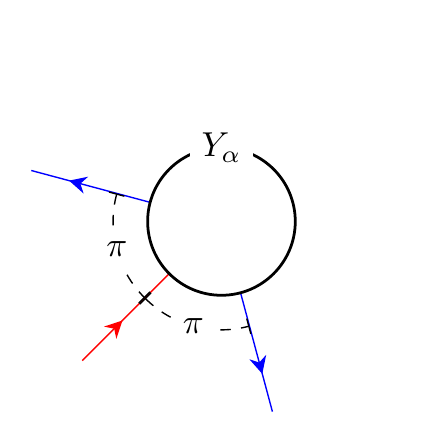}
    \subcaption{Geometric geodesics}
    \label{fig:geometric-geodesic}
  \end{minipage}
  \caption{Diffractive and geometric geodesics}
  \label{fig:geom-diff-geodesics}
\end{figure}
As described in \cite{Melrose-Wunsch1}, the geometric geodesics are those that
are locally realizable as limits of families of ordinary geodesics in the
interior $X^\circ$; see also \cite{Baskin-Wunsch:Resolvent}*{Section 2.1} for a
discussion of the question of local versus global approximability.

We shall also use the geodesic flow at the level of the cotangent bundle, and it
is helpful to see how this behaves uniformly up to $\del X$ (although in this
paper, microlocal considerations will only arise over $X^\circ$---a considerable
simplification).  To describe this flow, let $\Tbstar X$ denote the b-cotangent
bundle of $X$, i.e., the dual of the bundle ${}^\upb TX$ whose sections are
$\calV_\upb(X)$, the smooth vector fields tangent to $\pa X$.  We write
$\Sbstar X$ for the corresponding sphere bundle.  The b-cotangent bundle comes
equipped with a canonical 1-form ${}^\upb \alpha$, which in our product
coordinates near the boundary is
\begin{equation*}
  {}^\upb \alpha = \xi \, \frac{dx}{x} + \eta \cdot dy ,
\end{equation*}
and we write ${}^\upb \sigma \defeq d \! \left( {}^\upb \alpha \right)$ for the
associated symplectic form.  This defines the natural coordinate system
$(x,\xi;y,\eta)$ on $\Tbstar X$ near $Y$.  (We refer the reader to Chapter 2 of
\cite{Melrose:APS} for a detailed explanation of ``b-geometry,'' of which we
only use the rudiments here.)

Let $\vec{H}_g$ be the Hamilton vector field (with respect to ${}^\upb \sigma$)
of the metric function $\frac{1}{2} \, g$, the symbol of $\frac{1}{2} \Delta_g$
on $\Tbstar X$; in product coordinates near $Y$, this is
\begin{equation*}
  \frac{g}{2} = \frac{1}{2} \, \frac{\xi^2 + h(x,y,\eta)}{x^2},
\end{equation*} 
With this normalization, $\vec{H}_g$ is the geodesic spray in $\Tbstar X$ with
velocity $\sqrt{g}$.  It is convenient to rescale this vector field so that it
is both tangent to the boundary of $X$ and homogeneous of degree zero in the
fibers.  Near a boundary component $Y_\alpha$, we have (see
\cite{Melrose-Wunsch1})
\begin{equation}
  \vec{H}_g = x^{-2} \left\{ \vec{H}_{Y_\alpha}(x) + \left[ \xi^2 + h(x,y,\eta)
      + \frac{x}{2} \, \frac{\del h}{\del x} \right] \del_\xi + \xi x \, \pa _x
  \right\}, 
  \label{spray}
\end{equation}
where $\vec{H}_{Y_\alpha}(x)$ is the geodesic spray in $Y_\alpha$ with respect
to the family of metrics $h_\alpha(x, \cdot)$.  Hence, our desired rescaling is
\begin{equation*}
  {\vec{Z}} \defeq \frac{x}{\sqrt g} \, \vec{H}_g .
\end{equation*}
(Note that $g$ here refers to the metric function on the b-cotangent bundle and
not the determinant of the metric tensor.)  By the homogeneity of $\vec{Z}$, if
we radially compactify the fibers of the cotangent bundle and identify
$\Sbstar X$ with the ``sphere at infinity'', then $\vec{Z}$ is tangent to
$\Sbstar X$ and may be restricted to it.  Henceforth, we let $\barZ$ denote this
restriction of $\frac{x}{\sqrt{g}} \, \vec{H}_g$ to the \emph{compact} manifold
$\Sbstar X$.  On the sphere bundle $\Sbstar X$, we replace the b-dual
coordinates $(x,\xi;y,\eta)$ in a neighborhood of the boundary by the
(redundant) coordinate system
\begin{equation*}
  \left(x, \bar\xi; y, \bar\eta \right) \defeq \left( x, \frac{\xi}{\sqrt{\xi^2
        + h(x,y;\eta)}} ; y,  \frac{\eta}{\sqrt{\xi^2 + h(x,y;\eta)}} \right) .
\end{equation*}
Using these coordinates, it is easy to see that $\barZ$ vanishes only at the
critical manifold $\{x = 0, \bar\eta=0\}$ over $\pa X$, and thus the closures of
maximally extended integral curves of this vector field can only begin and end
over $\pa X$.  Since $\barZ$ is tangent to the boundary, such integral curves
either lie entirely over $\pa X$ or lie over $\pa X$ only at their limit points.
Hence, the interior and boundary integral curves giving rise to our broken
geodesics can meet only at their limit points in this critical submanifold
$\{x=0, \bar\eta=0\} \subseteq \Sbstar X.$

We now introduce a way of measuring the lengths of the integral curves of
$\barZ$.  Suppose $\gamma$ is such an integral curve over the interior
$X^\circ$.  Let $k$ be a Riemannian metric on $\Sbstar X^\circ$ such that
$\barZ$ has unit length, i.e., $k(\barZ,\barZ)=1$, and let
\begin{equation}
  \lambda = x \, k(\cdot, \barZ) \in \Omega^1 \! \left(\Sbstar X\right).
  \label{omega}
\end{equation}
Then
\begin{equation*}
  \int_\gamma \lambda= \int_\gamma x \,  k\!\left (\frac{d\gamma}{ds},
    \barZ\right) ds = \int_\gamma \frac{x}{\sqrt g} \,  k\!\left( \vec{H}_g,
    \barZ \right) ds = \int_\gamma ds = \length(\gamma) ,
\end{equation*}
where $s$ parametrizes $\gamma$ as an integral curve of
$\frac{1}{\sqrt{g}} \, \vec{H}_g$, the unit speed geodesic spray.  Given this
setup, we may define two symmetric relations between points in $\Sbstar X$:  a
``geometric'' relation and a ``diffractive'' relation.  These correspond to the
two different possibilities for linking these points via geodesic flow through
the boundary.

\begin{definition}\label{defn:relations}
  Let $p$ and $p'$ be points of the b-cosphere bundle $\Sbstar X$.
  \begin{enumerate}[(a)]
  \item We write $p \dtilde p'$ if there exists a piecewise smooth but not
    necessarily continuous curve $\gamma:  [0,1] \To \Sbstar X$ with
    $\gamma(0)=p$ and $\gamma(1)=p'$ and such that $[0,1]$ can be decomposed
    into a finite union of closed subintervals $I_j$, intersecting at their
    endpoints, where
    \begin{enumerate}[(i)]
    \item on each $I_j^\circ$, $\gamma$ is a (reparametrized) positively
      oriented integral curve of $\barZ$ in $\Sbstar X^\circ$;
    \item the final point of $\gamma$ on $I_j$ and the initial point of $\gamma$
      on $I_{j+1}$ lie over the same component $Y_\alpha$ of $\pa X$; and
    \item $\int_\gamma \lambda = t$ for
      $\lambda \in \Omega^1\!\left( \Sbstar X \right)$ as in \eqref{omega}.
    \end{enumerate}
  \item We write $p \gtilde p'$ if there exists a \emph{continuous} and
    piecewise smooth curve $\gamma:  [0,1] \To \Sbstar X$ with $\gamma(0)=p$ and
    $\gamma(1)=p'$ such that $[0,1]$ can be decomposed into a finite union of
    closed subintervals $I_j$, intersecting at their endpoints, where
    \begin{enumerate}[(i)]
    \item on each $I_j^\circ$, $\gamma$ is a (reparametrized) positively
      oriented integral curve of $\barZ$ in $\Sbstar X$;
    \item on successive intervals $I_j$ and $I_{j+1}$, interior and boundary
      curves alternate; and
    \item $\int_\gamma \lambda = t$.
    \end{enumerate}
  \end{enumerate}
\end{definition}

We know from the preceding discusion that the integral curves of $\barZ$ over
$X^\circ$ are lifts of geodesics in $X^\circ$.  It follows from the formula
\eqref{spray} for $\vec{H}_g$ near the boundary that the maximally extended
integral curves of $\barZ$ in the restriction $\Sbstar_{\pa X} X$ of the
b-cosphere bundle to the boundary are the lifts of geodesics of length $\pi$ in
$\pa X$ (see \cite{Melrose-Wunsch1} for details).  Hence, we may conclude the
following proposition.

\begin{proposition}
  Suppose that $p$ and $p'$ are points of $\Sbstar X$.  Then
  \begin{enumerate}[(a)]
  \item $p\gtilde p'$ if and only if $p$ and $p'$ are connected by a (lifted)
    geometric geodesic of length $t$, and
  \item $p\dtilde p'$ if and only if $p$ and $p'$ are connected by a (lifted)
    diffractive geodesic of length $t$.
  \end{enumerate}
\end{proposition}

An important feature of these equivalence relations, proved in
\cite{Wunsch:Poisson}, is the following:
\begin{proposition}[\cite{Wunsch:Poisson}*{Prop.~4}]
  \label{proposition:closed}
  The sets
  \begin{equation*}
    \left\{(p,p',t) : p\gtilde p' \right\} \quad \text{and} \quad \left\{
      (p,p',t) : p\dtilde p' \right\}
  \end{equation*}
  are closed subsets of $\Sbstar X \times \Sbstar X \times \RR_+$.
  \label{prop:closed}
\end{proposition}

We remark that based on pure dimensional considerations, closed geodesics that
involve geometric interactions should not generically exist, just as closed
geodesics on a smooth manifold passing through a fixed, finite set of marked
points are non-generic.\footnote{A simple way to demonstrate that a residual set
  of conic metrics admits no closed geodesics with geometric interactions at
  cone points is as follows.  Suppose we are given a conic manifold with no pair
  of cone points which are conjugate.  (Genericity of this situation should also
  be easy to demonstrate by perturbing in $X^\circ$.)  Now consider all
  geodesics of length less than $A$ both starting and ending at cone points; the
  endpoints of these at the boundary components then form a discrete (and hence
  finite) set $S_A$, by nonconjugacy.  Now we scale the boundary metric $h_0$
  while leaving the metric $g$ unchanged except in a small neighborhood of the
  boundary: for $\beta$ close to $1$ we set
  \begin{equation*}
    g_\beta = dx^2 + x^2 \left[\psi(x) \, \beta + (1-\psi(x)) \right] h(x)
  \end{equation*}
  where $\psi$ is supported in $[0, \ep)$ and equal to $1$ in
  $[0,\frac{\ep}{2}].$ Then for all but finitely many values of $\beta$ no pair
  of points in $S_A$ are exactly distance $\pi$ apart, hence no concatenation of
  any two is a geometric geodesic.  If $\ep$ is chosen small, we have not
  introduced any new geodesics of length less than $A$ that connect cone points,
  implying that we have killed off all geometric geodesics of length less than
  $A$ by this simple perturbation.} For this reason we focus our attention in
this paper on the contributions to the wave trace of closed diffractive
geodesics; they, together with closed geodesics in $X^\circ$, should typically
give the only singularities of the wave trace.

\subsection{Jacobi fields and Fermi normal coordinates}
\label{section:Jacobi}

% \af{This section needs to be fleshed out further.}

In the course of our analysis, we will encounter two classes of Jacobi fields
along the geodesics $\gamma$ of $X$.  The first of these, the \emph{cone Jacobi
  fields}, are the solutions to the Jacobi equation with respect to the metric
$g$ which are sections of the \emph{cone tangent bundle}
\begin{equation}
  \label{eq:cone-tangent-bundle}
  \coneT X \defeq x^{-1} \cdot {}^\upb T X
\end{equation}
along $\gamma$.  The sections of this bundle, which we denote by
$\calV_{\text{cone}}(X)$, are spanned over $\calC^\infty(X)$ by the vector
fields $\del_x$ and $x^{-1} \cdot \del_{y_j}$ for $j = 1, \ldots, n-1$ near $Y$
and restrict to be smooth vector fields in the interior.  While these vector
fields $\del_x$ and $x^{-1} \, \del_{y_j}$ are singular at the boundary, we note
that they are the vector fields of unit length with respect to $g$ and thus,
from a Riemannian point of view, naturally extend the smooth vector fields from
the interior.  Indeed, the standard Riemannian objects such as the metric $g$,
the volume half-density $\omega_g$, and the curvature tensor $\text{Riem}$ are
smooth sections of bundles constructed from $\coneT X$ and its dual bundle,
$\coneT^* X = x \cdot \Tbstar X$.

The second class of Jacobi fields we encounter are the \emph{b-Jacobi fields},
introduced by Baskin and the second author in \cite{Baskin-Wunsch:Resolvent}.
They are the sections of the b-tangent bundle ${}^\upb T X$ along $\gamma$ which
satisfy the Jacobi equation with respect to $g$, i.e., they are the smooth
Jacobi fields which are tangent to $Y$.  In particular, if $\gamma$ connects two
different cone points, this definition entails tangency to both the starting and
ending components of the boundary.  We note that nonzero b-Jacobi fields are
necessarily normal to a geodesic:  in the boundary coordinates $(x,y)$, the
geodesics reaching the boundary are exactly the curves $\{y = y_0\}$.  The
smooth Jacobi fields along these geodesics are spanned by the coordinate vector
fields $\del_x$ and $\del_{y_j}$, and the constant linear combinations of these
which are tangent to the boundary are precisely the span of the $\del_{y_j}$'s.

The b-Jacobi fields with respect to a cone point should be regarded as the
analogue on a conic manifold of the ordinary Jacobi fields on a smooth manifold
that vanish at a given point---they are precisely the cone Jacobi fields on $X$
which vanish simply \emph{as a cone vector field} at the cone point.  Since the
Jacobi equation is second order, these cone vector fields are specified by their
derivatives at this initial point, and thus the corresponding b-Jacobi fields
are specified by the point in the boundary component $Y_\alpha$ from which they
emanate, owing to the uniqueness of the geodesic striking that point in the
boundary.  This measures an ``angle of approach'' to the cone point when viewed
metrically, i.e., in the blown-down picture shown in Figure
\ref{fig:mfld-conic-sings-metric}.  As in \cite{Baskin-Wunsch:Resolvent}, we use
these Jacobi fields to define when the endpoints of geodesics emanating from a
cone point are conjugate.

\begin{definition}[\cite{Baskin-Wunsch:Resolvent}*{Section 2.4 and Appendix A}]
  \mbox{}
  \begin{enumerate}[(a)]
  \item We say that a point $p \in X^\circ$ is \emph{conjugate} to a cone point
    $Y_\alpha$ along a geodesic $\gamma$ if there exists a nonzero cone Jacobi
    field $\vec{J} \in \calV_{\text{cone}}(\gamma)$ along $\gamma$ which
    vanishes at both $p$ and $Y_\alpha$. Equivalently, $p$ is conjugate to
    $Y_\alpha$ along $\gamma$ if there exists a nonzero b-Jacobi field
    ${}^\upb \vec{J}$ along $\gamma$ vanishing at $p$.
  \item We say two cone points $Y_\alpha$ and $Y_\beta$ are \emph{conjugate}
    along a geodesic $\gamma$ if there exists a nonzero cone Jacobi field
    $\vec{J} \in \calV_{\text{cone}}(\gamma)$ along $\gamma$ vanishing at both
    $Y_\alpha$ and $Y_\beta$.  Equivalently, they are conjugate along $\gamma$
    if there exists a nonzero b-Jacobi field ${}^\upb \vec{J}$ along $\gamma$.
  \end{enumerate}
\end{definition}

Out of these Jacobi fields we build the corresponding classes of \emph{Jacobi
  endomorphisms} (cf.~\cite{KowVan}).  These are the smooth sections $\bsfJ$ of
the bundles
\begin{equation*}
  \End \! \left( \coneT X \right) \quad \text{and} \quad \End \! \left( {}^\upb
    TX \right)
\end{equation*}
along a geodesic $\gamma : [0,T]_t \To X$ satisfying the analogous Jacobi
equation
\begin{equation*}
  \skew{3}{\ddot}{\bsfJ}(t) + \bm{\mathsf{Riem}}(t) \circ \bsfJ(t) = \vec{0} ,
\end{equation*}
where $\bm{\mathsf{Riem}} \in \End \! \left( \coneT X \right) \big\vert_\gamma$
is the endomorphism defined by applying the Riemann curvature tensor to the
tangent vector $\dot{\gamma}(t)$.  (In other words, Jacobi endomorphisms are
simply matrices of Jacobi fields.)  We only use a few facts about these
endomorphisms in our work.  The first, which is a consequence of a simple
calculation using the connection, is that the Wronskian
\begin{equation}
  \label{eq:Jacobi-endo-Wronskian}
  \calW(\bm{\mathsf{X}}, \bm{\mathsf{Y}}) \defeq
  \dot{\bm{\mathsf{X}}} {}^\upt (t) \cdot \bm{\mathsf{Y}}(t) -
  \bm{\mathsf{X}}{}^\upt(t) \cdot \dot{\bm{\mathsf{Y}}}(t)
\end{equation}
of any two solutions to this equation is constant, where
$\bm{\mathsf{X}}^\upt(t)$ is the adjoint endomorphism determined by the metric.
The second is the standard orthogonal decomposition into the tangential and
normal blocks of the endomorphism coming from the Gauss lemma:
\begin{equation*}
  \bsfJ(t) = \bsfJ^\parallel(t) \oplus \bsfJ^\perp(t).
\end{equation*}
Here, $\bsfJ^\parallel(t) : T_{\gamma(t)}\gamma \To T_{\gamma(t)}\gamma$ and
$\bsfJ^\perp(t) : N_{\gamma(t)}\gamma \To N_{\gamma(t)}\gamma$.  We note that
the constancy of the Wronskian for the whole Jacobi endomorphisms descends to
constancy of the Wronskians of the tangential and normal endomorphisms as well.

We will often use \emph{Fermi normal coordinates} $(\nu,\ell)$ along the
geodesics $\gamma$ of $X$.  These are defined by choosing a basepoint
$z_0 = \gamma(t_0)$ and a $g$-orthonormal basis
$\left( \del_\nu , \del_\ell \right)$ of $\coneT_{z_0} X$ with
$\del_\ell = \dot\gamma(t_0)$ the tangent vector to the geodesic at the
basepoint.  Extending this basis to all of $\gamma$ by parallel transport and
exponentiating then yields the Fermi normal coordinates $(\nu, \ell)$ on a small
tube around the geodesic.  The orthonormal basis $(\del_\nu, \del_\ell)$ at the
basepoint generally will be the initial data for a set of Jacobi fields
$\left( \vec{J}_1, \ldots, \vec{J}_n \right)$ or, equivalently, a Jacobi
endomorphism $\bsfJ$.  Indeed, we will primarily deal with the Jacobi
endomorphism $\bsfJ(t)$ along a geodesic $\gamma(t)$ whose initial data is
\begin{equation*}
  \bsfJ(0) = \vec{0} \quad \text{and} \quad \skew{3}{\dot}{\bsfJ}(0) = \Id .
\end{equation*}
Using Fermi normal coordinates along $\gamma$, we may show this endomorphism is
computable by the derivative of the exponential map at $\gamma(0)$:
\begin{equation}
  \label{eq:Jacobi-endo-exp}
  \bsfJ(t) \cdot \left( a \cdot \del_\nu + b \, \del_\ell \right) = \left. D
    \exp_{\gamma(0)}(-) \right|_{t \dot\gamma(0)} \cdot t \left( a \cdot
    \del_\nu + b \, \del_\ell \right) .
\end{equation}
In particular, we have the identification
\begin{equation}
  \label{eq:Jacobi-endo-exp-2}
  t^{-1} \cdot \bsfJ(t) \cdot \left( a \cdot \del_\nu + b \, \del_\ell \right) =
  \left. D \exp_{\gamma(0)}(-) \right|_{t \dot\gamma(0)} \cdot \left( a \cdot
    \del_\nu + b \, \del_\ell \right) .
\end{equation}

Note that the map $\exp_p(v),$ defined on $(x,v) \in TX^\circ,$ does \emph{not}
extend smoothly to the boundary, since the only geodesics that reach
the boundary have tangents that are multiples of $\pa_x$ (and the same
issue arises even if we view the bicharacteristic flow in the b-cotangent
bundle: bicharacteristics only limit to $x=0$ from the interior at the
radial points of the flow, which are at $\RR dx/x + 0
\cdot dy$---see \cite{Melrose-Wunsch1} for details).  However, we can
still define a flowout map from the boundary: employing a variant on the notation
of \cite{Melrose-Vasy-Wunsch1}, we denote the ``flowout map''
\begin{align*}
\F: [0,x_*) \times Y&\To X\\
\F: (\ell, y) &\longmapsto (x=\ell,y).
\end{align*}
We view this as time-$\ell$ flow along a geodesic (described here in
our product coordinates), however we can just as easily view it as the
``identity map'' that identifies a piece of a \emph{model cone}
$$
X_0 \equiv [0, x_*) \times Y
$$
with a neighborhood of $\pa
X$ in $X.$

Now we may take the b-differential of $\F$ as in
\cite{Melrose-Conormal} to obtain
$$
{}^{\upb}\F_*: {}^{\upb} T X_0 \To {}^{\upb} TX.
$$
Since multiplication by $x^{-1}$ identifies b- and cone-tangent
spaces, this also gives a map
$$
{}^{\text{cone}}\F_*: \coneT X_0 \To \coneT X.
$$
Likewise, we have dual maps
$$
{}^{\upb}\F^*: {}^{\upb} T^* X \To {}^{\upb} T^*X_0,
$$
$$
{}^{\text{cone}}\F^*: {}^{\text{cone}} T^* X \To {}^{\text{cone}} T^*X_0.
$$

Now equip $X_0$ with the ``product metric''
$$
g_0 = dx^2 + x^2 \, h(0,y,dy).
$$
As remarked above
$\omega_g(x,y) = \left| \det h(y) \right|^{\frac{1}{4}} \left| dx \wedge x \,
  dy_1 \wedge \cdots \wedge x \, dy_{n-1} \right|^{\frac{1}{2}}$
is naturally a smooth section of the vector bundle
$\left| \bigwedge^n \! \left[ \coneT^* X \right] \right|^{\frac{1}{2}}$, the
\emph{cone half-densities} on $X$; likewise $X_0$ has its own
metric half-density $\omega_{g_0}$ in its cone half-density bundle.  The map
${}^{\text{cone}}\F^*$ naturally induces isomorphisms of the density and
half-density bundles as well.

% For any $y_0 \in Y,$ we thus have an isometry, for which we employ the
% notation\footnote{Here,
%   $\det_g D \F\big|_{(x_0,y_0)}$ is the determinant of the
%   matrix representing 
% $\det_g D \F\big|_{(x_0,y_0)}$
% in
%   $g$-orthonormal bases of $\coneT_{(0,y_0)}X$ and
%   $\coneT_{(x_0,y_0)}X$; note that we have now stopped using pullback
%   and pushforward notation explicitly.}
% \begin{equation*}
%   \left|\det_g D \F\big|_{(x_0,y_0)} \right|^{-\frac{1}{2}} : \left| \textstyle\bigwedge^n \! \left[ 
%       \coneT^*_{(0,y_0)} X \right] \right|^{\frac{1}{2}} \stackrel{\cong}{\To}
%   \left| \textstyle\bigwedge^n \! \left[ \coneT^*_{(x_0,y_0)} X \right]
%   \right|^{\frac{1}{2}}  
% \end{equation*}
% which maps $\omega_g(0,y_0) (= \omega_{g_0}(0,y_0)=\omega_{g_0}(x_0,y_0))$  to $\omega_g(x_0,y_0)$.

We now set
\begin{equation}
  \label{eq:cone-Theta-iso}
  \begin{aligned}
    \Theta(Y_\alpha \to p_0) &\defeq  \left|
\det_g {}^{\text{cone}}\F_* \big|_{(x_0,y_0)}\right| \\
    \Theta(p_0 \to Y_\alpha) &\defeq
 \left|
\det_g ({}^{\text{cone}}\F_*)^{-1}\big|_{(x_0,y_0)}\right|;
  \end{aligned}
\end{equation}
the metric determinant is taken with respect to an orthonormal basis for each
space with respect to the metrics on $X_0$ and $X.$

Since the dual mapping ${}^{\text{cone}}\F^*$ is related to
${}^{\text{cone}}\F_*$ by taking the adjoint matrix in the dual basis,
we easily see that
$$
{}^{\text{cone}}\F^* \omega_g^2(p_0)=\Theta(Y_\alpha\to p_0)\,  \omega_{g_0}^2(p_0),
$$
i.e., if we abuse notation by identifying $X_0$ and $X$ via the map
${}^{\text{cone} }\F,$ we may simply write:
\begin{equation}\label{nonproductmetric}
\omega_{g_0}(p_0) = \Theta^{-1/2}(Y_\alpha\to p_0)\,  \omega_g(p_0).
\end{equation}

We now work somewhat more globally.  To begin, note
that for most values of $y_0,$ \eqref{eq:cone-Theta-iso} makes equally good sense for any $x_0
\in \RR,$ where we are interpreting $\F$ as the geodesic flow along
the geodesic $\gamma$ emanating from $(0,y_0) \in Y_\alpha,$ a given component
of $Y;$ hence the definition extends to map
$$
[0,\infty)_x \times Y \To X
$$
which, however, fails to be defined beyond some maximal value of $x$ along
the geodesics that hit other cone points.
  The fact that
$$
{}^\upb\F_* \pa_{y_j} =\pa_{y_j}
$$
in our special coordinate system near the cone point easily
generalizes to show that ${}^{\upb}\F_*$ maps the tangent space of $Y_\alpha$
to b-Jacobi fields.  Thus, as above we may let
$\bsfJ(t)$
denote the basis of b-Jacobi fields along a flowout geodesic $\gamma$ with $\bsfJ(0)$ equal to an
orthonormal basis of $TY_\alpha.$  Then if $p_0=\F(x_0,y_0),$ we simply have
$$
\Theta(Y_\alpha\to p_0)=x_0^{-(n-1)} \lvert \det_g \bsfJ(x_0) \rvert,
$$
with the factor of $x_0$ arising in the transition from b- to
cone-tangent bundles; note that $x_0$ is of course the distance
traveled along the geodesic from $Y_\alpha$ to $p_0.$

Returning to our discussion of interior points above, we may
reinterpret this definition as follows: In the special case where
$Y_\alpha$ is a ``fictional'' cone point obtained by blowing up a
point $z_0$ on a smooth manifold, these b-Jacobi fields correspond to
Jacobi fields on the original manifold that vanish at $z,$ and it is
easily seen that we may reinterpret $\Theta$ as the determinant of the
differential of the exponential map $\exp_z(-)$ (cf.\ \cite{Ber}, p.271).

Now suppose that $(0,y_0) \in Y_\alpha,$ a given boundary component, and that there 
exists $x_0$ such that
$$
\F(x_0,y_0) \in Y_\beta,
$$
a different boundary component.  We may define\footnote{As before, we
  omit from the notation the choice of which geodesic we are using to
  connect $Y_\alpha$ and $Y_\beta,$ although this certainly matters.
  Note that the choice is from a discrete set, owing to our
  non-conjugacy hypotheses.}
\begin{equation}\label{ThetaYY}
\Theta(Y_\alpha \to Y_\beta) \defeq  x_0^{-(n-1)}
\left\lvert
\det_g \bsfJ(x_0) \right\rvert
\end{equation}
just as at interior points, where we simply note that while the Jacobi
fields in $\bsfJ$ are generically \emph{singular} at $Y_\beta$ owing
to the singularity of the Jacobi equation, they are precisely sections
of $\coneT X$ there, hence the metric is well-defined and the
resulting quantity is finite.

The quantity $\Theta(Y_\alpha \to Y_\beta) $ is a measure of the
convexity of the flowout manifold $\F(-,Y_\alpha)$---analogous to a geodesic
sphere in the smooth manifold case---as it reaches $Y_\beta.$ The
manipulations involved in the proof of
Theorems~\ref{thm:interior-ampl} and \ref{thm:amplitude-mult-diff}
will allow another interpretation, namely as the Hessian of the
difference of shape operators of the flowouts of $Y_\alpha$ and
$Y_\beta$ respectively at a point $p_0$ along a geodesic $\gamma$
connecting them, rescaled by certain half-density factors; this
product arises via a Wronskian of the two sets of b-Jacobi fields, one
set coming from each of the two cone points, and can be seen to be a
measure of the degree of tangency of the two flowouts at $p_0$; it is
intriguing that the result is independent of the choice of $p_0$ along
$\gamma,$ but we will not dwell on this issue here.

\section{Single diffraction on a product cone}
\label{section:singleproduct}

In this section we review the results of Cheeger and Taylor on the symbol of the
diffracted wavefront on a product cone \cites{Cheeger-Taylor1, Cheeger-Taylor2}.
Note that in our work the space dimension is denoted $n$, while this dimension
is denoted $m+1$ in these references.

Suppose $Y$ is a connected, closed manifold of dimension $n-1$ (such as a single
component $Y_\alpha$ of the boundary appearing in the previous
discussion).  As above, let
$X_0$ denote the ``product cone'' over $Y$, the noncompact cylinder
$[0, \infty) \times Y$ equipped with the scale-invariant metric
\begin{equation*}
  g_0 = dx^2+ x^2 \, h_0(y,dy)
\end{equation*}
(we have slightly changed notation so that $X_0$ has an infinite
end, but our considerations are all local in any case).
Here, $h_0$ is a Riemannian metric on $Y$ (e.g.,
$h_\alpha \defeq h \big\vert_{Y_\alpha}$ from the previous section).  Let
$\Lap_0$ denote the Laplace-Beltrami operator acting on half-densities on $X_0$,
and let $\Box_0$ denote the d'Alembertian $D_t^2-\Lap_0$ on the half-densities
of the associated spacetime $\bbR \times X_0$.  Following Cheeger and Taylor, we
define a shifted square-root of the Laplacian
\begin{equation*}
  \nu \defeq \sqrt{ \Delta_{h_0} + \left( \frac{ 2 - n}{2} \right)^2 } ,
\end{equation*}
determined in the functional calculus of $\Delta_{h_0}$ on $Y$.  For a function
$f \in L^\infty(\bbR)$, we let $\calK[f(\nu)](y,y')$ (or sometimes simply
$f(\nu)$) denote the Schwartz kernel of the corresponding element of the
functional calculus.

Having set up the framework, we now state a mild reinterpretation of the results
of Cheeger and Taylor calculating the asymptotics of the sine propagator on
$\bbR \times X_0$,
\begin{equation*}
  \vec{W}_0(t) \defeq \frac{\sin \! \big(t \sqrt{ \smash[b]{\Delta_0}} \big) }{
    \sqrt{ \smash[b]{\Delta_0}} }.
\end{equation*}
In what follows, we let $u \defeq (x+x')-t$ denote the defining function for the
diffracted wavefront, and we write $N^*\{x + x' = t\}$ for its conormal bundle.

\begin{proposition}
  \label{proposition:sinproduct}
  Suppose $p = (x,y)$ and $p' = (x',y')$ are strictly diffractively
  related\footnote{Note that it is only possible for a geodesic to undergo one
    diffraction on $X_0$.} points in $X_0^\circ$, i.e.,
  \begin{equation*}
    p \dtilde p' \quad \text{and} \quad p \stackbin[{\bsfG}]{t}{\not\sim} p'
    .
  \end{equation*}
  Then near $(t,p,p') \in \bbR \times X^\circ_0 \times X^\circ_0$, the Schwartz
  kernel $\bmE_0$ of the sine propagator $\vec{W}_0(t)$ lies locally in the
  space of Lagrangian distributions\footnote{See
    Appendix~\ref{sec:lagrangian-dists-amplitudes} for a definition of
    Lagrangian distributions.}
  \begin{equation*}
    I^{-\frac{5}{4} - \frac{n-1}{2}} \! \left( \bbR \times X_0^\circ \times
      X_0^\circ, N^*\{ x + x' = t\} ; |\Omega|^\frac{1}{2}(X^\circ_0 \times
      X^\circ_0) \right) . 
  \end{equation*}
  In particular, $\bmE_0$ has an oscillatory integral representation
  \begin{equation}
    \bmE_0(t,x,y;x',y') = \int_{\bbR_\xi} e^{i (x + x' - t) \cdot
      \xi} \, e(t,x,y;x',y';\xi) \, d\xi ,
  \end{equation}
  where the amplitude has the leading order behavior
  \begin{equation*}
    e(t,x,y;x',y';\xi) \equiv \calE(x,y;x',y';\xi) \left( \mathrm{mod} \
      S^{-\frac{3}{2}} \! \left( \bbR \times X_0^\circ \times X_0^\circ ;
        \HD(X_0^\circ \times X_0^\circ) \right) \right)
  \end{equation*}
  for
  \begin{multline}
    \label{sinsymbolofonediffraction}
    \calE(x,y;x',y';\xi) \defeq \frac{(xx')^{-\frac{n-1}{2}}}{2\pi} \,
    \frac{\chi(\xi)}{2|\xi|} \\
    \mbox{} \times \left[ H(\xi) \, \calK[e^{-i\pi\nu}](y,y') + H(-\xi) \,
      \calK[e^{i \pi \nu}](y,y') \right] \omega_{g_0}(x,y) \,
    \omega_{g_0}(x',y').
  \end{multline}
  Here, $\cutoff \in \calC^\infty(\bbR_\xi)$ is a smooth function satisfying
  $\cutoff \equiv 1$ for $|\xi| > 2$ and $\cutoff \equiv 0$ for $|\xi| < 1$, and $H$
  is the Heaviside function.
\end{proposition}

\begin{remark}
  We call attention to a few aspects of this proposition.
  \begin{enumerate}[(a)]
  \item The sine propagator has order $-\frac{5}{4} - \frac{n-1}{2}$ as a
    Lagrangian distribution, showing the diffractive improvement of
    $\frac{n-1}{2}$ derivatives.
  \item The principal part of the amplitude, $\calE(x,y';x',y';\xi)$, is only
    given here modulo symbolic half-densities of $\frac{1}{2}$-order lower
    rather than those a full order lower.
    % \item[(iii)] $\calE(y,y',\xi)$ is independent of $x$ and $x'$, a
    %   consequence
    %   of it satisfying a transport equation along the flowout of the cone
    %   point.
  \item The operators $e^{i \pi \nu}$ and $e^{-i \pi \nu}$ are half-Klein-Gordon
    propagators for time-$\pi$, and hence their Schwartz kernels have singular
    support only at distance $\pi$ from the diagonal.
  \end{enumerate}
\end{remark}

\begin{proof}[Proof of Proposition \ref{proposition:sinproduct}]
  Varying slightly upon Cheeger and Taylor, let $\delta$ be the defining
  function
  \begin{equation*}
    \delta \defeq \sgn\!\left(t^2-(x+x')^2\right) \cdot \left| \frac{
        t^2-(x+x')^2}{xx'} \right|^{\frac{1}{2}}. 
  \end{equation*}
  From Theorems 5.1 and 5.3 of \cite{Cheeger-Taylor2}, we know that uniformly
  away from the set $\left\{ \dist_{h_0}(y,y') = \pi \right\}$ there is a
  complete asymptotic expansion of $\bmE_0$ of the form
  \begin{equation*}
    \bmE_0(t,x,y;x',y') \sim \sum_{j = 0}^\infty a_j(t,x,y;x',y')
    \, \delta^j + \sum_{k = 0}^\infty b_k(t,x,y;x',y') \, \delta^{2k}
    \log|\delta|
  \end{equation*}
  as $|\delta| \to 0$, where the leading terms are\footnote{In fact the
    expansions given here differ slightly from those stated in Theorems~5.1 and
    5.3 of \cite{Cheeger-Taylor2}.  In equation (5.20) of \cite{Cheeger-Taylor2}
    the factor of $1/(4-u^2)^{1/2}$ at the end of the first line should in fact
    read $1/(1-u^2/4)^{1/2},$ leading to the replacement of the factor
    $\log\sqrt{2/\delta}$ in the statement of these theorems with
    $\log (2\sqrt 2/\delta).$}
  \begin{equation*}
    \begin{aligned}
      a_0 &= \frac{1}{\pi (xx')^{\frac{n-1}{2}}} \left[ \log(2 \sqrt{2})
        \cos(\pi\nu) \phantom{ + \int_{s = 0}^\pi \frac{\cos(s\nu) -
            \cos(\pi\nu)}{2 \cos
            \! \left( \frac{s}{2} \right)} \, ds} \right. \\
      &\hspace*{1.3in} \left. \mbox{}+ \int_{s = 0}^\pi \frac{\cos(s\nu) -
          \cos(\pi\nu)}{2 \cos \! \left( \frac{s}{2} \right)} \, ds - H(-\delta)
        \, \frac{\pi}{2}
        \sin(\pi\nu) \right] \\
      b_0 &= - \frac{1}{\pi (xx')^{\frac{n-1}{2}}} \cos(\pi\nu) .
    \end{aligned}
  \end{equation*}
  (This may be thought of as a $\frac{1}{2}$-step quasipolyhomogeneous expansion
  of the kernel of $\vec{W}_0(t)$ in distributions associated to
  $\{ \delta = 0 \}$.)  The existence of this expansion establishes the
  Lagrangian structure of $\bmE_0$ in this region, and moreover the leading
  order singularity of $\bmE_0$ at the diffractive front is
  \begin{equation}
    \label{eq:sine-prop-leading-order-singularity}
    \frac{1}{\pi (xx')^{\frac{n-1}{2}}} \left[ - \frac{\pi}{2}
      \sin(\pi\nu) H(-\delta) - \cos(\pi\nu) \log|\delta| \right] .
  \end{equation}
  Note that the Schwartz kernels of the propagators $\cos(\pi\nu)$ and
  $\sin(\pi\nu)$ are in fact smooth since we are localized away from the
  submanifold $\left\{ \dist_{h_0}(y,y') = \pi \right\}$.

  We now convert this expansion in $\delta$ to one in our defining function
  $u \defeq x + x' - t$ using the relation
  \begin{equation*}
    \delta \sim \sgn(u) \left| \frac{2t}{x x'} \right|^\frac{1}{2}
    |u|^\frac{1}{2} \text{ as $|\delta| \to 0$} .
  \end{equation*}
  Thus \eqref{eq:sine-prop-leading-order-singularity} becomes
  \begin{equation}
    \label{eq:sine-prop-leading-order-singularity-redux}
    \frac{1}{\pi (xx')^{\frac{n-1}{2}}} \left[ - \frac{\pi}{2}
      \sin(\pi\nu) H(-u) - \frac{1}{2} \cos(\pi\nu) \log|u| \right]
  \end{equation}
  modulo a smooth function.  Using the oscillatory integral representations
  \begin{equation*}
    H(-u) = \int_{\bbR_\xi} e^{i u \xi} \, \frac{i}{2\pi}  \frac{1}{\xi + i 0}
    \, d\xi \quad \text{and} \quad \log|u| = \int_{\bbR_\xi} e^{i u \xi} \left[
      - \frac{1}{2 |\xi|} - \gamma \, \delta(\xi) \right] d\xi
  \end{equation*}
  where $\gamma$ is the Euler-Mascheroni constant, we may represent the leading
  order singularity of \eqref{eq:sine-prop-leading-order-singularity-redux} as
  the oscillatory integral
  \begin{equation*}
    \frac{1}{2\pi \left(x x'\right)^{\frac{n-1}{2}}} \int_{\bbR_\xi} e^{i u \xi}
    \left[ \frac{1}{2|\xi|} \cos(\pi \nu) + \gamma \, \delta(\xi) \, \cos(\pi
      \nu) - \frac{i}{2} \frac{1}{\xi + i 0} \sin(\pi \nu) \right] d\xi .
  \end{equation*}
  The singularities at $\xi = 0$ in this expression are superfluous since we are
  only interested in the large-$|\xi|$ behavior of this function.  Thus, this
  distribution is equivalent (up to introducing a smooth error) to
  \begin{equation*}
    \frac{1}{2\pi \left(x x'\right)^{\frac{n-1}{2}}} \int_{\bbR_\xi} e^{i u \xi}
    \, \cutoff(\xi) 
    \left[ \frac{1}{2|\xi|} \cos(\pi \nu) - \frac{i}{2\xi} \sin(\pi \nu) \right]
    d\xi .  
  \end{equation*}
  where $\cutoff \in \calC^\infty(\bbR_\xi)$ is a smooth function satisfying
  $\cutoff \equiv 1$ for $|\xi| > 2$, say, and $\cutoff \equiv 0$ for $|\xi| < 1$.  To
  make this expression more intuitive, we replace the sine and cosine
  Klein-Gordon propagators on the link with their half-wave counterparts,
  yielding
  \begin{equation*}
    \frac{1}{2\pi \left(x x'\right)^{\frac{n-1}{2}}} \int_{\bbR_\xi} e^{i u \xi}
    \, \cutoff(\xi) \, \frac{H(\xi) \, e^{- i \pi \nu} + H(-\xi) \, e^{i \pi
        \nu}}{2 |\xi|} d\xi ,
  \end{equation*}
  and therefore the principal amplitude of this distribution is
  \begin{equation}
    \label{eq:symbol-sine-prop-scalars}
    \frac{\left( x x' \right)^{-\frac{n-1}{2}}}{2 \pi} \cdot \cutoff(\xi) \,
    \frac{H(\xi) \cdot e^{- i \pi \nu} + H(-\xi) \, e^{i \pi \nu}}{2 |\xi|} .
  \end{equation}

  As we have calculated it, this is the principal amplitude of the propagator
  acting on scalars, and thus it is missing a right density factor, i.e.,
  $\bmE_0 \, \omega_g^2(x',y')$ is the (right-density) kernel of the operator on
  scalars. Letting ${}^\frac{1}{2} \bmE_0$ denote the kernel acting on
  half-densities, we calculate:
  \begin{equation*}
    {}^\frac{1}{2} \bmE_0 = \left[ \bmE_0 \,
      \omega_g^2(x',y') \right] \omega_g(x,y) \, \omega_g^{-1}(x',y') = \bmE_0
    \, \omega_g(x,y) \, \omega_g(x',y'),
  \end{equation*}
  yielding the expression \eqref{sinsymbolofonediffraction}.
\end{proof}

We now use the Lagrangian structure of $\bmE_0$ near the diffractive front to
conclude the analogous structure for the Schwartz kernel of the (forward)
half-wave propagator $\calU_0(t) \defeq e^{-it \sqrt{ \smash[b]{\Delta_0} } }$
on $X_0$.

\begin{corollary}
  \label{cor:onediffraction}
  Suppose $p = (x,y)$ and $p'=(x',y')$ are strictly diffractively related points
  in $X_0^\circ$ as above.  Then near
  $(t,p,p') \in \bbR \times X_0^\circ \times X_0^\circ$, the Schwartz kernel
  $\bmU_0$ of the half-wave propagator lies locally in the Lagrangian
  distributions:
  \begin{equation*}
    \bmU_0 \in I^{-\frac{1}{4} - \frac{n-1}{2}}\! \left(\bbR \times X_0^\circ \times
      X_0^\circ, N^* \{ t = x + x'\} ; |\Omega|^\frac{1}{2} (X_0^\circ \times
      X_0^\circ) \right) . 
  \end{equation*}
  Using the phase function $\phi(t,x,x',\xi) = (x + x' - t) \cdot \xi$, its
  principal amplitude is
  \begin{equation}
    \label{symbolofonediffraction}
    \calD(x,y';x',y';\xi) \defeq \frac{(xx')^{-\frac{n-1}{2}}}{2\pi i} \,
    \cutoff(\xi) \, 
    \calK[e^{-i \pi \nu}](y,y') \, \omega_{g_0}(x,y) \, \omega_{g_0}(x',y') ,
  \end{equation}
  modulo elements of $S^{-\frac{1}{2}}$.  Here,
  $\cutoff \in \calC^\infty(\bbR_\xi)$ is a smooth function satisfying
  $\cutoff \equiv 1$ for $\xi > 2$ and $\cutoff \equiv 0$ for $\xi < 1$.
\end{corollary}

\begin{proof}
  We compute the kernel of the half-wave operator via Euler's formula:
  \begin{equation*}
    e^{-it\sqrt{\smash[b]{\Delta_0}}} = \cos\!\left( t \sqrt{ \smash[b]{\Delta_0} }
    \right) - i \sin\!\left( t \sqrt{ \smash[b]{\Delta_0} } \right) .
  \end{equation*}
  Differentiating the oscillatory integral expression for $\bmE_0$ in the
  $\xi$-variable produces an expression for the Schwartz kernel of the cosine
  propagator
  \begin{equation*}
    \dot{\vec{W}}_0(t) \defeq \cos\!\left( t \sqrt{ \smash[b]{\Delta_0} }
    \right),
  \end{equation*}
  bringing a factor of $\frac{1}{i} \, \xi$ into the amplitude from
  differentiating the phase.  Thus, the principal term in the amplitude of
  $\dot{\bmE}_0 \defeq \calK \! \left[ \dot{\mathbf{W}}_0(t) \right]$ is
  \begin{equation*}
    (x x')^{-\frac{n-1}{2}} \, \frac{\cutoff(\xi)}{4\pi i} \frac{\xi}{|\xi|} \left[
      H(\xi) \, e^{-i\pi\nu} + H(-\xi) \, e^{i\pi\nu} \right] \omega_g(x,y) \,
    \omega_g(x',y') . 
  \end{equation*}
  To determine the principal amplitude of the operator
  $\sin \! \left( t \sqrt{ \smash[b]{\Lap_0}} \right)$, we
  recall\footnote{Technically, the noncompactness of $X_0$ makes these results
    not directly applicable, but ``closing up'' the large end of the cone yields
    a compact manifold that is locally identical to $X_0$ near the cone points
    on which we may as well (by finite speed of propagation) study the
    propagators.} from \cite{Loya} that $\sqrt{ \smash[b]{\Lap_0}}$ is a
  pseudodifferential operator over $X_0^\circ$ (it lies in the ``big
  b-calculus'' of Melrose \cite{Melrose:APS}), and hence applying it to the
  conormal distribution given in Proposition~\ref{proposition:sinproduct} has
  the effect of multiplying the amplitude by the value of the symbol of the
  pseudodifferential operator $\sqrt{\smash[b]{\Lap_0}}$ along the Lagrangian
  submanifold $N^*\{x + x' = t\}$, to wit, $|\xi|$.  Thus, the leading order
  amplitude of $-i \sin \!  \left( t \sqrt{\smash[b]{\Lap_0}} \right)$ away from
  $\left\{ \dist_{h_0}(y,y') = \pi \right\}$ is
  \begin{equation*}
    (x x')^{-\frac{n-1}{2}} \, \frac{\cutoff(\xi)}{4\pi i} \left[ H(\xi) \,
      e^{-i\pi\nu} + H(-\xi) \, e^{i\pi\nu}
    \right] \omega_g(x,y) \, \omega_g(x',y') .
  \end{equation*}
  Adding these two contributions produces \eqref{symbolofonediffraction}.
\end{proof}

% SECTION 3

% fw-conetrace-3.tex Section 3:  Single diffraction on a non-product cone

\section{Single diffraction on a non-product cone}
\label{section:single}

In this section, we use the information from Corollary
\ref{symbolofonediffraction} on the structure of diffraction on the product cone
to understand the analogous structure on our more general conic manifold $X$.
The finite speed of propagation implies that we only need to understand a single
diffraction, and thus we may work in a small collar neighborhood in $X$ of a
boundary component $Y_\alpha$,
\begin{equation*}
  C_\alpha \defeq \left[ 0, \frac{x_*}{2} \right) \times Y_\alpha ,
\end{equation*}
with $x_*$ as in Section~\ref{sec:conic-geometry}.
We write
$C_\alpha^\circ \defeq \left( 0 , \frac{x_*}{2} \right) \times Y_\alpha$ for the
interior of this collar neighborhood.

We work with two different metrics on the collar
neighborhood\footnote{Technically, we should consider these objects to
  live on the distinct manifolds $X_0$ and $X,$ identified by the
  flowout map $\F$ defined above; however, we will abuse notation to
  the extent of leaving this identification tacit.}  $C_\alpha$:  the
conic metric in designer form $g = dx^2 + x^2 \, h(x,y,dy)$ and the associated
product metric $g_0 = dx^2 + x^2 \, h(0,y,dy)$ coming from the boundary metric
$h(0)$.  Associated to these metrics are their (nonnegative) Laplace-Beltrami
operators, $\Delta_g$ and $\Delta_0 \defeq \Delta_{g_0}$, and their wave
operators, $\Box_g \defeq D_t^2 - \Delta_g$ and
$\Box_0 \defeq D_t^2 - \Delta_0$.  We now show that when half-density factors
are taken into account, the $D_x$-terms in $\Box_g$ and $\Box_0$ agree.  This is
very helpful in proving the conormality of the diffractive part of the
propagator as the remaining first-order $D_y$-terms act harmlessly on
distributions associated to $N^* \left\{ t = x + x' \right\}$, where we now
consider this as a Lagrangian submanifold of
$\Tdot^* \! \left( \bbR \times C_\alpha^\circ \times C_\alpha^\circ \right)$.

\begin{lemma}
  \label{lemma:box-relation-product-nonproduct}
  As operators on half-densities, we have
  \begin{equation*}
    \Box_g - \Box_0 \in x^{-1} \calC^\infty \! \left( \left[ 0,
        \frac{x_*}{2} \right); \Diff^2(Y) \right). 
  \end{equation*}
\end{lemma}

\begin{proof}
  By definition, $\Box_g$ and $\Box_0$ act trivially on two different
  trivializations of the half-density bundle: $\Box_g \omega_g = 0$ and
  $\Box_0 \omega_{g_0} = 0$, where
  \begin{equation*}
    \omega_g \defeq x^{\frac{n-1}{2}} \left[ \det h(x) \right]^{\frac{1}{4}}
    \left| dx \wedge dy
    \right|^\frac{1}{2} \quad \text{and} \quad 
    \omega_{g_0} \defeq x^{\frac{n-1}{2}} \left[ \det h(0) \right]^{\frac{1}{4}}
    \left| dx \wedge dy 
    \right|^\frac{1}{2} . 
  \end{equation*}
  In order to compare the operators we write them both in terms of the new
  half-density
  \begin{equation*}
    \varpi \defeq x^{\frac{n-1}{2}} \left| dx \wedge dy \right|^\frac{1}{2},
  \end{equation*}
  defined (non-invariantly!) in a local coordinate patch in $y$.  Then
  \begin{equation*}
    \Box_g(f \, \varpi) = (Pf) \, \varpi \quad \text{and} \quad \Box_0(f \,
    \varpi) = (P_0 f) \, \varpi 
  \end{equation*}
  with operators on the coefficients of $\varpi$ locally given by
  \begin{equation*}
    P \defeq \left[ \det h(x) \right]^\frac{1}{4} \, \Box_{g,\scal} \,
    \left[ \det h(x) \right]^{-\frac{1}{4}} \quad 
    \text{and} \quad P_0 \defeq \left[ \det h(0)\right]^\frac{1}{4} \,
    \Box_{0,\scal} \, \left[ \det h(0) \right]^{-\frac{1}{4}} . 
  \end{equation*}
  Here, we use the notation $\Box_{\bullet,\scal}$ to indicate the usual scalar
  d'Alembertian operators
  \begin{equation*}
    \Box_{g,\scal} \defeq D_t^2 - \left[ \frac{1}{x^{n-1}} \, \left[ \det h(x)
      \right]^{-\frac{1}{2}} D_x \left[ \det h(x)\right]^{\frac{1}{2}} \,
      x^{n-1} \, D_x + \frac{1}{x^2} \, \Delta_{h(x),\scal} \right]
  \end{equation*}
  and
  \begin{equation*}
    \Box_{0,\scal} \defeq D_t^2 - \left[ \frac{1}{x^{n-1}} \, D_x \, x^{n-1} \,
      D_x + \frac{1}{x^2} \, \Delta_{h(0),\scal} \right]
  \end{equation*}
  (with scalar Laplacians on $Y$ denoted in the analogous way).  Thus,
  \begin{multline*}
    P = D_t^2 - \left[ -\frac{i (n-1)}{x} \, D_x + \left[ \det h(x)
      \right]^{-\frac{1}{4}} D_x \left[ \det h(x) \right]^{\frac{1}{2}} D_x
      \left[ \det h(x) \right]^{-\frac{1}{4}} \right. \\
    \left. \mbox{} + \frac{1}{x^2} \left[ \det h(x) \right]^{\frac{1}{4}} \,
      \Delta_{h(x),\scal} \left[ \det h(x) \right]^{-\frac{1}{4}} \right],
  \end{multline*}
  while since $\del_x \! \left[ \det h(0) \right] = 0$,
  \begin{equation*}
    P_0 = D_t^2 - \left[ -\frac{i (n-1)}{x} \, D_x + D_x^2 + \frac{1}{x^2}
      \left[ \det h(0) \right]^{\frac{1}{4}} \, \Delta_{h(0),\scal} \left[ \det
        h(0) \right]^{-\frac{1}{4}} \right].
  \end{equation*}
  Thus,
  \begin{multline*}
    -(P - P_0) = \left[ \det h(x) \right]^{-\frac{1}{4}} D_x \left[ \det h(x)
    \right]^\frac{1}{2} D_x \left[ \det h(x) \right]^{-\frac{1}{4}} - D_x^2 \\
    \mbox{} + \frac{1}{x^2} \left[ \det h(x) \right]^{\frac{1}{4}} \,
    \Delta_{h(x),\scal} \left[ \det h(x) \right]^{-\frac{1}{4}} \\
    \mbox{} - \frac{1}{x^2} \left[ \det h(0) \right]^{\frac{1}{4}} \,
    \Delta_{h(0),\scal} \left[ \det h(0) \right]^{-\frac{1}{4}} .
  \end{multline*}
  We note that the operator on $\CI(\bbR_x)$ given by
  \begin{equation*}
    \left[ \det h(x)\right]^{-\frac{1}{4}} D_x \left[ \det h(x)\right]
    ^{\frac{1}{2}}  D_x \left[ \det h(x) \right]^{-\frac{1}{4}}
  \end{equation*}
  is formally self-adjoint on $L^2(\bbR_x,dx)$, has real coefficients when
  written in terms of $\del_x$, and has principal symbol $\xi^2$.  Thus, the
  subprincipal term---the coefficient of $\del_x$---vanishes, and the above
  operator equals $D_x^2 + \text{zeroth order term}$ (with the zeroth order term
  smooth, to boot).  Since we also have $h(x)=h(0) + \mathrm{O}(x)$, we now find
  that all the $D_x^2$- and $D_x$-terms cancel, and hence
  \begin{equation*}
    P - P_0 \in x^{-1} \calC^\infty \! \left( \left[0,\frac{x_*}{2} \right);
      \Diff^2(Y) \right) .
  \end{equation*}
  The lemma follows.
\end{proof}

We now state the principal result of this section.
\begin{theorem}
  \label{theorem:samesymbol}
  Uniformly away from the set
  $\left\{ \dist_{h(0)}(y,y') = \pi \right\} \subseteq C_\alpha^\circ \times
  C_\alpha^\circ$,
  the amplitudes\footnote{Note that this requires we use the same phase function
    in both oscillatory integral representations.} of the diffracted waves for
  the half-wave propagators $\calU(t)$ and $\calU_0(t)$ (as operators on
  half-densities) agree in the space
  \begin{equation*}
    \frac{S^0 \! \left( \bbR_t \times C_\alpha^\circ \times C_\alpha^\circ \times
        \bbR_\xi ; \HD(C_\alpha^\circ \times C_\alpha^\circ) \right)}{
      S^{-\gamma + 0} \! \left( \bbR_t \times C_\alpha^\circ \times C_\alpha^\circ
        \times \bbR_\xi ; \HD(C_\alpha^\circ \times C_\alpha^\circ) \right)},  
  \end{equation*}
  where $\gamma = \frac{1}{2}$ unless $h(x) - h(0) = \mathrm{O}(x^2)$, in which
  case $\gamma = 1$.
\end{theorem}

Let us recall what is already known about each of the half-wave propagators
$\calU(t)$ and $\calU_0(t)$.  From Corollary \ref{symbolofonediffraction}, we
know that near strictly diffractively related points in
$\bbR_t \times C_\alpha^\circ \times C_\alpha^\circ$ the half-wave kernel
$\bmU_0 \defeq \calK\!\left[\calU_0(t)\right]$ is locally an element of
\begin{multline*}
  I^{-\frac{1}{4} - \frac{n-1}{2}} \! \left( \bbR \times C_\alpha^\circ \times
    C_\alpha^\circ, N^*\{ t = x + x' \} ; |\Omega|^\frac{1}{2}(C_\alpha^\circ
    \times C_\alpha^\circ) \right) \\
  \mbox{}= \IB^{-\frac{1}{2}}_{2,\infty} \!  \left( \bbR \times C_\alpha^\circ
    \times C_\alpha^\circ, N^* \{ t = x + x' \} ;
    |\Omega|^\frac{1}{2}(C_\alpha^\circ \times C_\alpha^\circ) \right) ,
\end{multline*}
and its principal amplitude is \eqref{symbolofonediffraction} when the phase
function $\phi(t,x,x',\xi) = (x + x' - t) \cdot \xi$ is used in the oscillatory
integral representation (see Appendix~\ref{sec:lagrangian-dists-amplitudes}
for the definitions of the spaces employed here).\footnote{We may instead use the phase function
  $\psi(t,x,x',\eta) = (t - x - x') \cdot \eta$, which would simply correspond
  to changing the sign of the phase variable in the formula
  \eqref{symbolofonediffraction}.}  We also know from the main result of the
work of Melrose and the second author \cite{Melrose-Wunsch1} that in this same
region $\bmU \defeq \calK \!  \left[\calU(t)\right]$ is locally in
\begin{equation*}
  \IH^{-\frac{1}{2}-0} \! \left( \bbR \times C_\alpha^\circ \times C_\alpha^\circ,
    N^*\{t = x + x'\}; \HD(C_\alpha^\circ \times C_\alpha^\circ) \right),  
\end{equation*}
the space of Lagrangian distributional half-densities associated to
$N^*\{t = x + x'\}$ with iterated $H^{-\frac{1}{2}-0}_\loc$-regularity, though
that theorem gives no further information on its amplitude.  Thus, to show the
amplitudes are the same to leading order, we shall need to relate these two
spaces of Lagrangian half-densities and, in particular, relate the amplitudes
associated to each.  This comparision comes in the form of the following lemma,
where we test difference $\calU(t) - \calU_0(t)$ against data which is
Lagrangian with respect $N^*\left\{ x = x_0 \right\}$, the conormal bundle to a
transverse ``slice'' of the cone.

\begin{lemma}
  \label{lemma:reduction}
  Suppose $y_0$ and $y_0'$ are points of $Y_\alpha$ such that
  $\dist_{h(0)}(y_0,y_0') \neq \pi$, i.e., points in the boundary whose flowouts
  into the interior are strictly diffractively related.  Let
  $s \geqslant \frac{1}{2}$, and let
  $f \in H^s \!  \left(C_\alpha^\circ ; \HD(C_\alpha^\circ) \right)$ be a
  distributional half-density which is Lagrangian with respect to
  $N^*\{x = x_0\}$ and supported in a sufficiently small neighborhood of
  $(x_0,y_0')$.  Then for times
  $t \in I \defeq \left(x_0 , x_0 + \frac{x_*}{2} \right)$, we have
  \begin{equation*}
    \left[ \calU(t) - \calU_0(t) \right] f \in H^{s + \gamma - 0} \! \left( I
      \times C_\alpha^\circ ; |\Omega|^\frac{1}{2}(C_\alpha^\circ) \right) 
  \end{equation*}
  locally near $\{x = t - x_0\}$, where $\gamma$ is as in
  Theorem~\ref{theorem:samesymbol}.
\end{lemma}

\begin{remark}
  Note that we do not need to bother specifying whether conormality is with
  respect to $H^s\!\left(C_\alpha^\circ ; \HD(C_\alpha^\circ) \right)$ or some
  other Sobolev space with less regularity.  Indeed, by interpolation we obtain
  conormality with respect to
  $H^{s-0}\!\left( C_\alpha^\circ ; \HD(C_\alpha^\circ) \right)$, and this will
  suffice to prove the lemma.  We may therefore assume without loss of
  generality that $f$ is an element of
  \[\mathit{IH}^{\frac{1}{2}} \! \left( C_\alpha^\circ, N^*\{x = x_0\} ;
    |\Omega|^\frac{1}{2}(C_\alpha^\circ) \right), \]
  the minimal regularity required.
\end{remark}

\begin{proof}[Proof of Lemma~\ref{lemma:reduction}]
  Let $y_0$ and $y_0'$ be as in the statement of the lemma.  Taking into account
  the previous remark, let $f$ be in
  $\IH^{\frac{1}{2}} \! \left( C_\alpha^\circ , N^* \{ x = x_0 \} ;
    |\Omega|^\frac{1}{2}(C_\alpha^\circ) \right)$
  with support in a small neighborhood of $(x_0,y_0')$.  We observe that $f$ is
  then an element of the domain $\calD_{\frac{1}{2}}$ by microlocality of the
  powers of $\sqrt{ \smash[b]{\Delta_g} }$ away from the boundary.  As it is
  sufficient to show $\left[ \calU(t) - \calU_0(t) \right] f$ is locally an
  element of
  $H^{\frac{1}{2} + \gamma - 0} \! \left( I \times C_\alpha^\circ ;
    |\Omega|^\frac{1}{2}(C_\alpha^\circ) \right)$
  away from the geometric rays emanating from the support of $f$, this is how we
  will proceed.  We start with the general case $\gamma = \frac{1}{2}$.

  Define $u(t) \defeq \calU_0(t) f$ to be the associated solution to
  $\Box_0 u(t) = 0$ on $\bbR_t \times C_\alpha^\circ$ with initial half-wave
  data $f$.  By unitarity of $\calU_0(t)$, the solution $u(t)$ is an element of
  $L^\infty \!  \left( \bbR_t ; \calD_\frac{1}{2} \right)$, implying in
  particular that
  \begin{equation}
    \label{eq:reduction-lemma-reg-1}
    u(t) \in L^2 \! \left( I
      \times C_\alpha^\circ ; |\Omega|^\frac{1}{2}(C_\alpha^\circ) \right) .
  \end{equation}
  It is moreover a Lagrangian distributional half-density with respect to this
  regularity and the Lagrangian $N^*\{x = t - x_0\}$.  Writing
  $E \defeq \Box_g - \Box_0$, we compute that
  \begin{equation*}
    \Box_g u(t) = \left[ \Box_0 + E \right] u(t) = E u(t) .
  \end{equation*}
  % Since $E \in x^{-1} \calC^\infty \! \left( \left[ 0, \frac{x_*}{2} \right) ;
  %   \Diff^2(Y) \right)$ by Lemma~\ref{lemma:box-relation-product-nonproduct}
  % and
  % $u(t)$ has the above conormality, we know that $E u(t)$ is also in $L^\infty
  % \! \left( \bbR_t ; \calD_{\frac{1}{2}} \right)$.

  We now compute via Duhamel's formula that
  \begin{equation*}
    u(t) = \calU(t) f + \int_{ 0}^t \calU(t - s) \, E u(s) \, ds ,
  \end{equation*}
  and in particular
  \begin{equation}
    \label{duhamel}
    \left[ \calU_0(t) - \calU(t) \right] f = \int_{ 0}^t \calU(t - s) \, E u(s) \, ds .
  \end{equation}
  Suppose for the moment that the spatial dimension satisfies $n \geqslant 3$.
  We then have the \emph{Morawetz estimate} for the wave equation on the product
  cone $\left((\bbR_+)_x \times Y_\alpha, g_0\right)$ (cf.~\cite{BFM}*{Theorem
    4.1 and Remark 4.2} or the proof of Theorem 2 of \cite{BPST}):
  \begin{equation*}
    u(t) \in x L^2\!\left( \bbR_t \times (\bbR_+)_x \times Y_\alpha ;
      |\Omega|^\frac{1}{2}((\bbR_+)_x \times Y_\alpha) \right) .
  \end{equation*}
  As we are only interested in the submanifold $C_\alpha^\circ$ of
  $(\bbR_+)_x \times Y_\alpha$ and times $t \in I$, we also have
  \begin{equation*}
    u(t) \in x L^2\!\left( I \times C_\alpha^\circ ;
      |\Omega|^\frac{1}{2}(C_\alpha^\circ) \right) 
  \end{equation*}
  since the volume measures arising from $g$ and $g_0$ are comparable on
  $C_\alpha^\circ$.  Interpolating this last estimate with the Lagrangian
  regularity with respect to the space in
  \eqref{eq:reduction-lemma-reg-1}, we find that the solution $u(t)$ is in fact
  Lagrangian with respect to the regularity space
  $x^{1 - 0} L^2 \! \left( I \times C_\alpha^\circ ;
    |\Omega|^\frac{1}{2}(C_\alpha^\circ) \right)$.
  Hence, the result of Lemma~\ref{lemma:box-relation-product-nonproduct} implies
  \begin{equation*}
    Eu(t) \in x^{-0}  L^2\!\left( I \times C_\alpha^\circ ;
      |\Omega|^\frac{1}{2}(C_\alpha^\circ) \right) \subseteq L^2 \! \left( I ;
      \calD_{-0} \right) ,
  \end{equation*}
  where the inclusion follows from Proposition~\ref{thm:domains-b-Sobolev}.
  Thus, the right-hand side of \eqref{duhamel} is contained in
  $L^\infty \!  \left( I ; \calD_{1-0} \right)$ as it solves an inhomogeneous
  equation with inhomogeneity in $L^2 \! \left( I ; \calD_{-0} \right)$,
  implying that this term is more regular than $u(t)$ by $\frac{1}{2} - 0$
  derivatives, as claimed.

  Now, if the spatial dimension is $n = 2$, we must make a slight adjustment to
  the argument above in order to apply the Morawetz estimate.  Let $\Pi_0$ be
  the projection onto the zero mode of $Y_\alpha$ (i.e., the constants):
  \begin{equation*}
    \left[\Pi_0 f\right]\!(x) = \frac{1}{\mathrm{vol}(Y_\alpha)} \int_{Y_\alpha}
    f(x,y) \, \omega_{h(0)} . 
  \end{equation*}
  We decompose $f$ into its projection onto the zero mode and the positive
  modes:
  \begin{equation*}
    f = \left( f - \Pi_0 f \right) + \Pi_0 f .
  \end{equation*}
  The Morawetz estimate argument above applies verbatim to $f - \Pi_0 f$.  For
  the final piece, $\Pi_0 f$, we note that
  $E \left( \calU_0(t) \circ \Pi_0 f \right)$ is smooth on
  $I \times C_\alpha^\circ$ since $E$ acts as a multiplication operator on the
  range of $\Pi_0$.  Hence,
  $E \left( \calU_0(t) \circ \Pi_0 f \right) \in L^\infty \! \left( I ;
    \calD_{\frac{1}{2}} \right)$, implying that
  \begin{equation*}
    [ \calU(t) - \calU_0(t) ] \circ \Pi_0 f \in L^\infty \! \left( I ;
      \calD_\frac{3}{2} \right) ,
  \end{equation*}
  an even stronger estimate than is required.  Altogether, this establishes the
  result for $n = 2$.

  When $\gamma = 1$, i.e., when we have the stronger product cone structure
  given by $h(x) - h(0) = \mathrm{O}(x^2),$ then the operator $E$ is simply a
  family of differential operators on $Y$ with smooth coefficients:  there is no
  $x^{-1}$ singularity.  This improves the Duhamel term above so that we gain a
  full derivative in solving the inhomogeneous equation for
  $\calU(t)f-\calU_0(t) f$.
\end{proof}

Finally, we return to the proof of Theorem~\ref{theorem:samesymbol}.
\begin{proof}[Proof of Theorem~\ref{theorem:samesymbol}]
  From the discussion following the statement of this theorem, we know that near
  strictly diffractively related points of
  $\bbR \times C_\alpha^\circ \times C_\alpha^\circ$ our two half-wave
  propagators $\bmU$ and $\bmU_0$ are each Lagrangian distributional
  half-densities of class $\IH^{-\frac{1}{2} - 0}$ associated to
  $N^*\{t = x + x'\}$.  Therefore, over sufficiently small open subsets
  $I \times U \times U'$ of $\bbR \times C_\alpha^\circ \times C_\alpha^\circ$
  they each have local oscillatory integral representations of the form
  \begin{equation}
    \label{eq:kernel-oscil-int-form-1}
    \int_{\bbR_\xi} e^{i (x + x' - t) \cdot \xi} \, a(t,x,y;x',y';\xi) \, d\xi ,
  \end{equation}
  where $a$ is a symbolic half-density of class
  \[S^{\frac{1}{2} + 0}_\upc L^2 \!  \left( I \times U \times U' \times \bbR_\xi
    ; \HD(U \times U') \right).\]

  Now, by Lemma~\ref{lemma:reduction} we also know that for all choices of
 initial data
  $f \in \IH^\frac{1}{2} \! \left( C_\alpha^\circ, N^*\{x = x_0\} ;
    |\Omega|^\frac{1}{2}(C_\alpha^\circ) \right)$ we have
  \begin{equation}
    \label{eq:regularity-after-testing}
    \left[ \calU(t) - \calU_0(t) \right] f \in H^{\frac{1}{2} + \gamma - 0} \! \left(
      \bbR_t \times C_\alpha^\circ ; |\Omega|^\frac{1}{2}(C_\alpha^\circ) \right) .
  \end{equation}
  These initial data have oscillatory integral representations
  \begin{equation*}
    f(x,y) \equiv \int_{\bbR_\eta} e^{- i (x - x_0) \cdot \eta} \, b(y,\eta) \,
    d\eta \left( \text{mod $\calC^\infty$} \right) 
  \end{equation*}
  with $b$ in $S^{-\frac{1}{2}} L^2 (V \times \bbR_\eta ; \HD(V))$ for some open
  $V \subseteq Y_\alpha$; this is valid locally in a neighborhood of
  $\{x_0\} \times V$ in $C_\alpha^\circ$.  Taking $a$ and $a_0$ to be the
  amplitudes in \eqref{eq:kernel-oscil-int-form-1} of $\bmU$ and $\bmU_0$
  respectively, we then compute
  \begin{multline*}
    \left( \left[ \calU(t) - \calU_0(t) \right] f \right) \! (x,y) \\
    \mbox{} = \int_{C_\alpha} \int_{\bbR_\xi} e^{i (x + x' - t) \cdot \xi} \, [ a - a_0
    ](t,x,y;x',y';\xi) \int_{\bbR_\eta} e^{-i (x' - x_0) \cdot \eta} \,
    b(y',\eta) \, d\eta d\xi dx' dy' ,
  \end{multline*}
  again up to a smooth error.  Applying stationary phase in the
  $(x',\eta)$-variables, this becomes
  \begin{equation*}
    \left( \left[ \calU(t) - \calU_0(t) \right] f \right) \! (x,y) \equiv
    \int_{\bbR_\xi} 
    e^{i (x + x_0 - t) \cdot \xi} \, e(t,x,y;\xi) \, d\xi \left( \text{mod
        $\calC^\infty$} \right) ,
  \end{equation*}
  where
  \begin{equation*}
    e(t,x,y;\xi) \equiv C \int [a - a_0](t,x,y;x_0,y';\xi) \, b(y',\xi) \, dy'
    \left( \text{mod $S^{-\frac{3}{2} + 0}_\upc L^2$} \right)
  \end{equation*}
  for a constant $C$.  In particular, $e$ is a priori an element of
  $S^{-\frac{1}{2} + 0}_\upc L^2 (W \times \bbR_\xi; \HD(W))$ for some open
  $W \subseteq \bbR_t \times C_\alpha^\circ$ by
  \eqref{eq:L2-symbol-multiplication}.  However, by
  \eqref{eq:regularity-after-testing} and an interpolation argument we know that
  $[\calU(t) - \calU_0(t)]f$ is locally an element of
  \begin{equation*}
    \IH^{\frac{1}{2} + \gamma - 0} \! \left( \bbR_t \times C_\alpha^\circ , N^*
      \{ x = t - x_0 \} ; \HD(C_\alpha^\circ) \right),
  \end{equation*}
  so $e$ is actually an element of the lower-order symbol space
  $S^{-\frac{1}{2} - \gamma + 0}_\upc L^2 (W \times \bbR_\xi; \HD(W))$.  Hence,
  Lemma~\ref{lemma:averagedsymbol} implies that
  \begin{equation*}
    a - a_0 \in S^{- \gamma + 0}_\upc (U \times \bbR_\xi; \HD(U)) ,
  \end{equation*}
  and this proves the theorem.
\end{proof}

Now we define
\begin{equation}
  \label{eq:scattering-matrix-shorthand}
  \bmD_\alpha(y,y') \defeq \calK\!\left[ e^{- i \pi \nu_\alpha} \right]\!(y,y').
\end{equation}
We employ the comparison of product and non-product metrics in 
\eqref{nonproductmetric} to 
express the principal amplitude of $\bmU$ given in
\eqref{symbolofonediffraction} in terms of the ambient (nonproduct) metric
half-density $\omega_g$. 

\begin{theorem}
  \label{thm:nonproduct-hw-amp}
  Let $p = (x,y)$ and $p' = (x',y')$ be strictly diffractively related points in
  $C_\alpha^\circ$.  Then near
  $(t,p,p') \in \bbR \times C_\alpha^\circ \times C_\alpha^\circ$, the Schwartz
  kernel $\bmU$ of the half-wave propagator $\calU(t)$ has an oscillatory
  integral representation
  \begin{equation}
    \label{eq:nonproduct-hw-o-int}
    \bmU(t,x,y;x',y') = \int_{\bbR_\xi} e^{i(x + x' - t) \cdot \xi} \,
    d(t,x,y;x',y';\xi) \, d\xi
  \end{equation}
  whose amplitude $d \in S^0$ is
  \begin{equation}
    \label{eq:nonproduct-hw-amp}
    \frac{(xx')^{-\frac{n-1}{2}}}{2\pi i} \, \cutoff(\xi) \cdot
    \bmD_\alpha(y,y') \cdot
    \Theta^{-\frac{1}{2}}(Y_\alpha \to y) \, \Theta^{-\frac{1}{2}}(y' \to
    Y_\alpha) \, \omega_g(x,y) \, \omega_g(x',y') 
  \end{equation}
  modulo elements of $S^{-\frac{1}{2} + 0}$.  Here,
  $\cutoff \in \calC^\infty(\bbR_\xi)$ is a smooth function satisfying
  $\cutoff \equiv 1$ for $ \xi > 2$ and $\cutoff \equiv 0$ for $ \xi < 1$.
\end{theorem}

\section{The amplitude of a multiply-diffracted wave}
\label{sec:amplitude-mult-diff}

We now return to the setting of a general conic manifold $(X,g)$.  Before
calculating the trace, we calculate the amplitude of $\calU(t)$ microlocally
along a geodesic $\gamma$ undergoing multiple (strictly) diffractive
interactions with the cone points of $X$.  Of central importance is a
calculation of the amplitude along geodesics in the interior of $X$ with a
particularly convenient choice of phase function, which we
may treat as a calculation in the smooth manifold setting.  To our knowledge, no
version of this calculation currently exists in the literature, although
many similar analyses of the propagator have been made close to the
diagonal (within the injectivity radius).

\subsection{The $\gamma$-microlocalization of the half-wave group}
\label{sec:gamma-microlocalization}

To define what we mean by the amplitude of $\calU(t)$ microlocally along a
geodesic, let us fix a (broken) geodesic segment $\gamma : [0,T] \To \Sbstar X$
whose endpoints $\gamma(0)$ and $\gamma(T)$ lie over the interior $X^\circ$ and
which undergoes $k$ strictly diffractive interactions with the cone points of
$X$.  Thus, $\gamma$ is a piecewise smooth curve in the b-cosphere bundle with
jump discontinuities at each of the boundary components
$Y_{\alpha_1},\ldots,Y_{\alpha_k}$ through which it passes (note boundary
components may repeat in this sequence).  Writing
\begin{equation*}
  \gamma^\flat \defeq \pr \mbox{} \circ \mbox{} \gamma : [0,T] \To X
\end{equation*}
for the projection of this geodesic to the base manifold, we label the endpoints
as $p_0^\flat \defeq \gamma^\flat(0)$ and $p_1^\flat \defeq \gamma^\flat(T)$.
By shortening the geodesic slightly, we may arrange that
\begin{equation}
  \label{eq:nonconjugacy-assumption-multiple-diffractions}
  \text{$p_0^\flat$ is not conjugate to $Y_{\alpha_1}$ and $p_1^\flat$ is not
    conjugate to $Y_{\alpha_k}$ along $\gamma^\flat$},
\end{equation}
the analogue of the nonconjugacy assumption
\eqref{eq:nonconjugacy-assumption-cone-points} at this stage.  We label the
segments and endpoints of $\gamma^\flat$ as follows:
\begin{enumerate}[(i)]
\item $\gamma_0^\flat$ is the segment connecting $p_0^\flat$ to
  $q_1' \in Y_{\alpha_1}$;
\item $\gamma_j^\flat$ is the segment connecting $q_j \in Y_{\alpha_j}$ to
  $q_j' \in Y_{\alpha_{j+1}}$ for $j = 1, \ldots, k-1$; and
\item $\gamma_k^\flat$ is the segment connecting $q_k \in Y_{\alpha_k}$ to
  $p_1^\flat$.
\end{enumerate}

We partition the domain of $\gamma$ as
\begin{equation*}
  0 = T_0 < T_1 < \cdots < T_{2k} < T_{2k+1} = T
\end{equation*}
so that for $m = 1, \ldots, k$ the projections of
$\gamma \!  \left(T_{2m-1} \right)$ and $\gamma \! \left(T_{2m} \right)$ to the
base each lie in the collar neighborhoods $C_{\alpha_m}^\circ$ of the boundary
component $Y_{\alpha_m}$ and so that the $m$-th diffraction of $\gamma$ occurs
between $t = T_{2m-1}$ and $t = T_{2m}$.  By perturbing such a partition
slightly, we may also arrange that none of the points $\gamma^\flat(T_m)$ in
$X^\circ$ are conjugate to one another or to the cone points along $\gamma$,
refining the assumption
(\ref{eq:nonconjugacy-assumption-multiple-diffractions}).  See
Figure~\ref{fig:broken-geodesic-and-partition} for an illustration of such a
partition.
\begin{figure}
  \centering
  \includegraphics{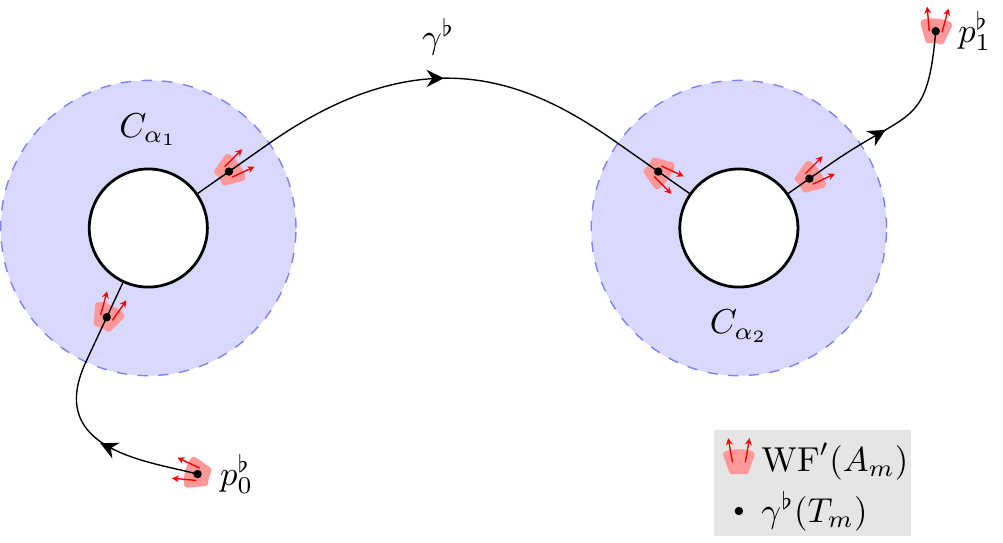}
  \caption{A projected geodesic segment $\gamma^\flat$ and its partitioning}
  \label{fig:broken-geodesic-and-partition}
\end{figure}

We define $t_m \defeq T_m - T_{m-1}$ to be the interim time between the
designated points along $\gamma$, and we choose microlocalizers
$A_m \in \Psi^0_\upc(X^\circ)$ whose microsupports are contained in sufficiently
small neighborhoods of the points $\gamma(t_m) \in S^* X^\circ$, e.g., within
$\varepsilon$-balls with respect to the metric $g$ restricted to the unit sphere
bundle for an $\varepsilon > 0$ as small as we need.

\begin{definition}
  Let $\gamma : [0,T] \To \Sbstar X$ be a broken geodesic segment partitioned as
  above.  We define the \emph{$\gamma$-microlocalization} of the half-wave
  propagator to be
  \begin{equation}
    \label{eq:gamma-microlocalization-hw}
    \calU^\gamma(t) \defeq A_{2k+1} \, \calU(t - T_{2k}) \, A_{2k} \,
    \calU(t_{2k}) \, A_{2k-1} \cdots A_1 \,\calU(t_1) \, A_0 .
  \end{equation}
\end{definition}

The rest of our efforts in this section go towards calculating the principal
amplitude of $\calU^\gamma(t)$.  We note that the factors in this operator are
of two basic types.  The first, the factors
$A_{2m+2} \circ \calU(t_{2m+2}) \circ A_{2m+1}$, microlocalize to within the
collar neighborhoods $C_{\alpha_{m+1}}$ capturing the propagation through the
cone point $Y_{\alpha_{m+1}}$.  The computation of their amplitudes is an
application of Theorem~\ref{thm:nonproduct-hw-amp} above.  The other factors, of
the form $A_{2m+1} \circ \calU(t_{2m+1}) \circ A_{2m}$, microlocalize the
propagator to within the interior $X^\circ$ and thus capture the propagation
along $\gamma$ away from the cone points.  We now calculate their amplitudes.

\subsection{The amplitude in the interior}
\label{sec:amplitude-interior}

Since the factors $A_{2m+1} \circ \calU(t_{2m+1}) \circ A_{2m}$ are only
microlocally nontrivial in a compact subset of the interior of $X$, it is
equivalent to consider the same framework in a closed, smooth manifold $Z$.
Therefore, suppose $\gamma : [0,T] \To \Tdot^*Z$ is a geodesic segment of length
$T$ with endpoints
\begin{equation*}
  \gamma(0) = (z_0,\zeta_0) \quad \text{and} \quad \gamma(T) = (z_1,\zeta_1) .
\end{equation*}
We make the analogous assumption to
\eqref{eq:nonconjugacy-assumption-multiple-diffractions} in this case:
\begin{equation}
  \label{eq:nonconjugacy-assumption-interior}
  \text{ $z_0$ and $z_1$ are not conjugate along $\gamma^\flat$ in $Z$.}
\end{equation}
This implies the existence of fiber-homogeneous neighborhoods
$U_0 \ni (z_0,\zeta_0)$ and $U_1 \ni (z_1,\zeta_1)$ in $\Tdot^*Z$ between which
the exponential map is a diffeomorphism.

Using these neighborhoods, we choose microlocalizers $A_0$ and $A_1$ in
$\Psi^0(Z)$ such that
\begin{equation*}
  (z_0,\zeta_0) \in \WF'(A_0) \subseteq U_0 \quad \text{and} \quad (z_1,\zeta_1)
  \in \WF'(A_1) \subseteq U_1 .
\end{equation*}
By choosing the microsupports of these operators sufficiently small, we may
write $A_0$ as the right quantization
% \begin{equation}
%   \label{eq:A_0-oscill-pres}
%   \calK\!\left[ A_0 \right](z,z') \equiv (2\pi)^{-n} \int_{\bbR^n_\zeta} e^{i
%   (z - 
%     z') \cdot \zeta} \, a_0(z',\zeta) \, d\zeta \mod{\calC^\infty}
% \end{equation}
% for
of a compactly-supported symbol $a_0 \in S^0_\upc(U_0)$, and we may write $A_1$
as the left quantization
% \begin{equation}
%   \label{eq:A_1-oscill-pres}
%   \calK\!\left[ A_1 \right](z,z') \equiv (2\pi)^{-n} \int_{\bbR^n_\zeta} e^{i (z
%     - z') \cdot \zeta} \, a_1(z,\zeta) \, d\zeta \mod{\calC^\infty} .
% \end{equation}
of $a_1 \in S^0_\upc(U_1)$. Composing these, we form the associated
$\gamma$-microlocalization of the half-wave group
$\calU^\gamma(t) = A_1 \, \calU(t) \, A_0$.  From the calculus of Fourier
integral operators \cite{Hormander:FIO1} and H\"ormander's result on the
structure of $\calU(t)$ \cite{Hormander:SpectralFunction}, we conclude
\begin{equation*}
  \bmU^\gamma \defeq \calK\!\left[\calU^\gamma(t)\right] \in
  I^{-\frac{1}{4}}_\upc \! \left( \bbR \times Z \times Z 
    , G^t[A_0,A_1]' ; |\Omega|^\frac{1}{2} (Z \times Z) \right) ,
\end{equation*}
where
\begin{equation*}
  G^t[A_0,A_1] \defeq \left\{(t,\tau;z,\zeta;z',\zeta') : 
    \begin{aligned}
      &\text{$\tau = |\zeta'|_g$, $(z',\zeta') \in \WF'(A_0)$,} \\
      &\text{and $(z,\zeta) = G^t(z',\zeta') \in \WF'(A_1)$}
    \end{aligned}
  \right\}
\end{equation*}
is the graph of geodesic flow from $\WF'(A_0)$ to $\WF'(A_1)$ and $(\cdot)'$
denotes the fiber-twist
$(t,\tau;z,\zeta;z',\zeta')' = (t,\tau;z,\zeta;z',-\zeta')$ making
$G^t[A_0,A_1]'$ into a Lagrangian submanifold of
$\Tdot^*(\bbR \times Z \times Z)$.
% If a point $(t,\tau;z,\zeta;z',\zeta')$ is part of this relation
% $G^t[A_0,A_1]$, then we shall write $c : [0,t]_s \To \Tdot^*Z$ for the unique
% geodesic with with $c(0) = (z',\zeta')$ and $c(t) = (z,\zeta)$, and we again
% let $c^\flat \defeq \pr \mbox{} \circ \mbox{} c$ be its projection to $Z$.
% Furthermore,

For points $(z,z')$ in $U_1^\flat \times U_0^\flat$, where $U_j^\flat$ is the
projection of $U_j$ to $Z$, we may find a variation $c^\flat : [0,t]_s \To Z$ of
$\gamma^\flat$ by a Jacobi field such that $c^\flat(0) = z'$ and
$c^\flat(t) = z$.  We then define
\begin{equation}
  \label{eq:gamma-distance}
  \dist^\gamma_g(z,z') \defeq \length(c^\flat)
\end{equation}
to be the function measuring the distance between $z'$ and $z$ ``along
$\gamma$''.  Our geometric assumptions imply that this distance function is
smooth for $(z,z') \in U_1^\flat \times U_0^\flat$, and moreover we have
\begin{equation*}
  G^t[A_0,A_1]' \subseteq N^* \! \left\{ \dist_g^\gamma(z,z') = t \right\} .
\end{equation*}

\begin{theorem}
  \label{thm:interior-ampl}
  Assume the nonconjugacy condition \eqref{eq:nonconjugacy-assumption-interior},
  and suppose $t > \varepsilon$ for some $0 < \varepsilon \ll 1$.  Provided the
  microsupports of $A_0$ and $A_1$ are chosen sufficiently small, the Schwartz
  kernel of $\calU^\gamma(t)$ on $\bbR \times Z \times Z$ has the representation
  \begin{equation}
    \label{eq:interior-gamma-microlocalization-representation}
    \bmU^\gamma(t,z,z') = \int_{\bbR_\theta} e^{i \left[ \dist^\gamma_g(z,z') -
        t \right] \theta} \, b(t,z,z',\theta) \, d\theta
  \end{equation}
  whose amplitude
  $b \in S^{\frac{n-1}{2}}_\upc \! \left( \bbR \times U_1^\flat \times U_0^\flat
    \times \bbR_\theta ; |\Omega|^\frac{1}{2} (U_1^\flat \times U_0^\flat)
  \right)$ is congruent to
  \begin{multline}
    \label{eq:interior-gamma-microlocalization-principal-amplitude}
    a_1\!\left(z,\del_z \dist_g^\gamma(z,z') \cdot \theta\right) a_0\!\left(z',
      - \del_{z'} \dist_g^\gamma(z,z') \cdot \theta\right) \\
    \mbox{} \times \frac{e^{- \frac{i\pi(n-1)}{4}} \,
      i^{-m_\gamma}}{(2\pi)^{\frac{n+1}{2}}} \cdot \frac{\chi(\theta) \,
      \theta^{\frac{n-1}{2}}}{\dist_g^\gamma(z,z')^{\frac{n-1}{2}}} \cdot
    \Theta^{-\frac{1}{2}}(z' \to z) \cdot \omega_g(z) \, \omega_g(z')
  \end{multline}
  modulo elements of $S^{\frac{n-1}{2}-1}$.  In the above,
  \begin{enumerate}[\hspace*{.25em}$\bullet$]
  \item $\chi \in \calC^\infty(\bbR_\theta)$ satisfies $\chi \equiv 0$ for
    $\theta < 1$ and $\chi \equiv 1$ for $\theta > 2$;
  \item $m_\gamma$ is the Morse index of $\gamma$; and
  \item
    $\Theta(z' \to z) \defeq \left| \det_g \left[ D \exp_{z'}(-) \right]
      \big\vert_{\exp^{-1}_{z'}(z)} \right|$
    is the determinant of the matrix representing
    $\left[ D \exp_{z'}(-) \right] \big\vert_{\exp^{-1}_{z'}(z)}$ in
    $g$-orthonormal bases of $T_{z'}Z$ and $T_zZ$.
  \end{enumerate}
\end{theorem}
Note that the use of $\Theta$ here to denote the determinant of the derivative of the
exponential map at a point in $X^\circ$ is consistent with the definitions we have made of the
analogous quantity at cone points above.

We prove this theorem in Section~\ref{sec:proof-theorem-int-ampl}.

\subsection{Assembling the pieces}
\label{sec:assembling-pieces}

We assemble the calculations from Sections~\ref{section:single} and
\ref{sec:amplitude-interior} to compute the principal amplitude of
$\calU^\gamma(t)$.  In the following, we let
$\Pi^\gamma_\alpha : X^\circ \To Y_\alpha$ be the map taking a point in the
interior of $X$ to the point in $Y_\alpha$ to which it is linked by a geodesic
near $\gamma^\flat$ (in the sense of small variations by cone Jacobi fields as
above), and we set the geodesic segments $\tilde\gamma_0^\flat$ and
$\tilde\gamma_k^\flat$ in $X$ to be
\begin{equation*}
  \tilde\gamma_0^\flat \defeq \gamma^\flat \big|_{[0,T_1]} \quad \text{and}
  \quad \tilde\gamma_k^\flat = \gamma^\flat \big|_{[T_{2k},T]} .
\end{equation*}

\begin{theorem}
  \label{thm:amplitude-mult-diff}
  Let
  $\calU^\gamma(t) = A_{2k+1} \, \calU(t - T_{2k}) \, A_{2k} \, \calU(t_{2k}) \,
  A_{2k-1} \cdots A_1 \,\calU(t_1) \, A_0$
  be a $\gamma$-microlocalization of the half-wave propagator undergoing $k$
  diffractions through the cone points $Y_{\alpha_1},\ldots,Y_{\alpha_k}$, no
  pair of which are conjugate.  Then the Schwartz kernel of $\calU^\gamma(t)$
  has an oscillatory integral representation
  \begin{equation}
    \label{eq:oscil-int-mult-diff}
    \bmU^\gamma \equiv \int e^{i \phi} \, b(t,z,z',\xi) \, d\xi
    \mod{\calC^\infty} 
  \end{equation}
  with phase function
  \begin{equation}
    \label{eq:phase-mult-diff}
    \phi \defeq \left[ \dist^{\gamma_k}_g(z, Y_{\alpha_{k-1}}) + \sum_{j
        = 2}^{k-1} \dist_g^{\gamma_{j-1}}(Y_{\alpha_j},Y_{\alpha_{j-1}}) +
      \dist_g^{\gamma_0}(Y_{\alpha_1},z') - t \right] \xi
  \end{equation}
  and amplitude
  $b \in S^{- \frac{k(n-1)}{2}}_\upc \! \left( \bbR_t \times U_1^\flat \times
    U_0^\flat \times \bbR_\xi ; \HD \! \left( U_1^\flat \times U_0^\flat \right)
  \right)$ given by
  \begin{multline}
    \label{eq:amplitude-mult-diff}
    \bm{a}(z,z',\xi) \cdot \frac{e^{\frac{i\pi(n-1)(k-1)}{4}} \,
      (2\pi)^{\frac{(n+1)(k-1)}{2}}}{(2\pi i)^k} \cdot \chi(\xi) \, \xi^{-
      \frac{(k-1)(n-1)}{2}} \cdot \left[ \prod_{j = 0}^k i^{-m_{\gamma_j}}
    \right] \\
    \mbox{} \times \bmD_{\alpha_k}(\Pi^{\gamma_k}_{\alpha_k}(z), q_k') \cdot
    \left[ \prod_{j = 2}^{k-1} \bmD_{\alpha_j}(q_j,q_j') \right] \cdot
    \bmD_{\alpha_1}(q_1,\Pi^{\gamma_1}_{\alpha_1}(z')) \\
    \mbox{} \times \dist_g^{\gamma_k}(z, Y_{\alpha_k})^{-\frac{n-1}{2}} \cdot
    \left[ \prod_{j = 1}^{k-1}
      \dist_g^{\gamma_j}(Y_{\alpha_{j+1}},Y_{\alpha_j})^{-\frac{n-1}{2}} \right]
    \cdot
    \dist_g^{\gamma_0}(Y_{\alpha_1},z')^{-\frac{n-1}{2}} \\
    \mbox{} \times \Theta^{-\frac{1}{2}}(z' \to Y_{\alpha_1}) \cdot \left[
      \prod_{j = 1}^{k-1} \Theta^{-\frac{1}{2}}(Y_{\alpha_j} \to
      Y_{\alpha_{j+1}}) \right] \cdot \Theta^{-\frac{1}{2}}(Y_{\alpha_k} \to z)
    \cdot \omega_g(z) \, \omega_g(z')
  \end{multline}
  modulo elements of $S^{-\frac{(k-1)(n-1)}{2} - \frac{1}{2} + 0}$.  In the
  above,
  \begin{enumerate}[\hspace*{.25em}$\bullet$]
  \item the symbol
    $\bm{a} \in S^0_\upc \!  \left(U_1^\flat \times U_0^\flat \times \bbR_\xi ;
      \HD(U_1^\flat \times U_0^\flat) \right)$
    is the combined amplitude of the microlocalizers
    $A_j \in \Psi^0_\upc(X^\circ)$,
    \begin{multline}
      \label{eq:amplitude-mult-diff-microlocalizers}
      \bm{a}(z,z',\xi) \defeq a_{2k+1} \! \left( z, \del_z
        \dist_g^{\tilde\gamma_k}(z,w_{k}) \cdot \xi \right) \\
      \mbox{} \times \left[ \prod_{j = 1}^k a_{2j} \! \left( w_{j}, - \del_{w_j}
          \dist_g(w_{j},Y_{\alpha_j}) \cdot \xi \right) \cdot a_{2j-1} \! \left(
          w_{j}', \del_{w_{j}'}
          \dist_g(Y_{\alpha_j}, w_{j}') \cdot \xi \right) \right] \\
      \mbox{} \times a_0 \! \left(z', - \del_{z'}
        \dist_g^{\tilde\gamma_0}(w_{1}',z') \cdot \xi \right)
    \end{multline}
    evaluated at the critical points
    \begin{equation*}
      \begin{aligned}
        w_1' &\defeq \exp_{z'} \! \left(
          \frac{T_1}{\dist_g^{\gamma_0}(Y_{\alpha_1},z')}
          \exp_{z'}^{-1}(Y_{\alpha_1}) \right),\\
        w_j' &\defeq \gamma_{j-1}^\flat \! \left( T_{2j-1} - \sum_{\ell =
            1}^{j-2} \dist_g^{\gamma_\ell}(Y_{\alpha_{\ell+1}} Y_{\alpha_\ell})
          - \dist_g^{\gamma_0}(Y_{\alpha_1},z') \right), \text{ $j = 2,
          \ldots, k$},\\
        w_j &\defeq \gamma_j^\flat \! \left( \defeq T_{2j} - \sum_{\ell =
            1}^{j-1} \dist_g^{\gamma_\ell}(Y_{\alpha_{\ell+1}}, Y_{\alpha_\ell})
          - \dist_g^{\gamma_0}(Y_{\alpha_1},z') \right), \text{
          $j = 1, \ldots, k -
          1$, and}\\
        w_k &\defeq \exp_z \! \left( \frac{t -
            T_{2k}}{\dist_g^{\gamma_k}(z,Y_{\alpha_k})}
          \exp_{z}^{-1}(Y_{\alpha_k}) \right) ;
      \end{aligned}
    \end{equation*}
  \item $\chi \in \calC^\infty(\bbR_\xi)$ satisfies $\chi \equiv 0$ for
    $\xi < 1$ and $\chi \equiv 1$ for $\xi > 2$;
  \item $m_{\gamma_j}$ is the Morse index of $\gamma^\flat_j$;
  \item $\bmD_\alpha(y,y')$ is the Schwartz kernel of the half-Klein-Gordon
    propagator $e^{- i \pi \nu_\alpha}$ on $Y_\alpha$; and
  \item
    $\Theta(z' \to z) \defeq \left| \left. \det_g \! \left[ D \exp_{z'}(-)
        \right] \right|_{\exp_{z'}^{-1}(z)} \right|$
    is the determinant of the matrix representing
    $\left. \left[ D \exp_{z'}(-) \right] \right|_{\exp_{z'}^{-1}(z)}$ in
    $g$-orthonormal bases of $\coneT_{z'} X$ and $\coneT_z X$.
  \end{enumerate}
\end{theorem}

We prove this theorem in Section~\ref{sec:proof-theorem-mult-diff}.

% SECTION 5

% fw-conetrace-5.tex Section 5:  Proof of Theorem \ref{thm:interior-ampl}

\section{Proof of Theorem \ref{thm:interior-ampl}}
\label{sec:proof-theorem-int-ampl}

We begin by fixing an element
$(t,z,z') \in \bbR_+ \times U_1^\flat \times U_0^\flat$ at which we want to
compute the representation
\eqref{eq:interior-gamma-microlocalization-representation} of $\calU^\gamma(t)$.
If $(t,z,z')$ is not the image of a point
$(t,\tau;z,\zeta; z',\zeta') \in G^t[A_0,A_1]$ under the projection map, then
the amplitude of \eqref{eq:interior-gamma-microlocalization-representation} is
residual (i.e., in $S^{-\infty}$) there.  Therefore, we restrict to those
$(t,z,z')$ which are in this projection.  To each such point corresponds a
geodesic $c : [0,t]_s \To \Tdot^*Z$ with $c(0) = (z',\zeta')$ and
$c(t) = (z,\zeta)$ which we use to compute the amplitude.  We partition the
domain of $c$ as
\begin{equation*}
  0 = t_0 < t_1 < \cdots < t_{M-1} < t_M = t,
\end{equation*}
with these times chosen so that
\begin{enumerate}[(i)]
\item the interim times $s_m \defeq t_m - t_{m-1}$ are each less than the
  injectivity radius $\inj(Z,g)$;
\item for $m = 1,\ldots,M-1$, none of the points $c^\flat(t_m)$ is conjugate to
  $c^\flat(0)$;
\item and for $m = 1,\ldots,M-1$ there is at most one point between
  $c^\flat(t_{m})$ and $c^\flat(t_{m+1})$ along $c^\flat$ which is conjugate to
  $c^\flat(0)$.
\end{enumerate}
Using this partition and the group property, we decompose the operator as
\begin{equation}
  \label{eq:int-ml-decomp}
  \calU^\gamma(s) = A_1 \, \calU(s - t_{M-1}) \, \calU(s_{M-1}) \cdots
  \calU(s_2) \, \calU(s_1) \, A_0 ,
\end{equation}
a microlocalized composite of short-time propagators analogous to
\eqref{eq:gamma-microlocalization-hw}.

We prove Theorem~\ref{thm:interior-ampl} by induction on $M$, the number of
short-time propagators required in this decomposition of $\calU^\gamma(t)$
corresponding to the point $(t,z,z')$.  The base case is $M = 1$, corresponding
to the $\gamma$-microlocalizations of a single short-time propagator.  The core
of this is the following.

\begin{lemma}
  \label{thm:CdV-Berard-presentation-short-time}
  Let $I_0$ and $W_0$ be the open sets
  \begin{equation*}
    I_0 \defeq \left\{ 0 < t < \inj(Z) \right\} \subseteq \bbR_t \quad
    \text{and} \quad W_0 \defeq \left\{ \dist_g(z,z') < \inj(Z) \right\}
    \subseteq Z \times Z .
  \end{equation*}
  Then for $(t,z,z') \in I_0 \times W_0$, the Schwartz kernel of $\calU(t)$ has
  the representation
  \begin{equation}
    \label{eq:short-time-Berard-presentation}
    \bmU(t,z,z') \equiv \int_{\bbR_\theta} e^{i
      \left[ \dist_g(z,z') - t \right] \theta} \, b^\natural(t,z,z',\theta) \,
    d\theta \mod{\calC^\infty}
  \end{equation}
  whose amplitude
  $b^\natural \in S^{\frac{n-1}{2}} \! \left( I_0 \times W_0 \times \bbR_\theta
    ; |\Omega|^\frac{1}{2} \! (W_0) \right)$ is to leading order
  \begin{equation}
    \label{eq:short-time-Berard-principal-amplitude}
    \frac{t \, e^{-
        \frac{i\pi(n-1)}{4}}}{\pi^{\frac{n + 
          1}{2}} \left[ \dist_g(z,z') + t \right]^{\frac{n+1}{2}}} \cdot
    \chi(\theta) \, \theta^{\frac{n-1}{2}} \cdot \Theta^{-\frac{1}{2}}(z \to z')
    \cdot \omega_g(z) \, \omega_g(z')
  \end{equation}
  modulo elements of $S^{\frac{n-1}{2}-1}$.  The quantities here are the same as
  in Theorem~\ref{thm:interior-ampl}.
\end{lemma}

\begin{proof}
  The oscillatory integral representation
  \eqref{eq:short-time-Berard-presentation} is a modification of B\'erard's
  Hadamard parametrix construction for the sine propagator $\mathbf{W}(t)$ on
  functions from \cite{Ber}.  Rephrasing B\'erard's calculation in terms of the
  antidifferentiated half-wave propagator
  $\underdot{\bmU} \defeq \calK\!\left[ \frac{e^{i t \sqrt\Delta}}{\sqrt\Delta}
  \right]$, we have the \emph{exact} equality
  \begin{equation}
    \label{eq:Berard-antidiff-half-wave}
    \underdot{\bmU}(t,z,z') = \int_{\bbR_\eta} e^{i \left[ \dist^2_g(z,z') - t^2
      \right] \eta} 
    \cdot C_0 \, e^{- \frac{i \pi (n-1)}{2}} \left\{ \sum_{k = 0}^\infty
      U_k(z,z') \, \eta_+^{\frac{n-3}{2}-k} \right\} \omega_g^2(z') \, d\eta
  \end{equation}
  for $(t,z,z') \in \left\{ |t| < \inj(Z) \right\} \times W_0$, where $C_0$ is a
  constant depending only on the dimension $n$ and
  $U_0(z,z') = \Theta^{-\frac{1}{2}}(z \to z')$.  (The lower-order $U_k$'s are
  all explicit in terms of the geometry of $Z$; see (11) and (13) in
  \cite{Ber}.)

  We transform this into a representation for $\bmU$ by differentiating with
  respect to $t$, dividing by $i$, and replacing $t$ by $-t$.  This yields
  \begin{equation*}
    \bmU(t,z,z') = \int_{\bbR_\eta} e^{i \left[ \dist^2_g(z,z') - t^2 \right] \eta}
    \cdot 2C_0 \, t \, e^{- \frac{i \pi (n-1)}{4}} \left\{ \sum_{k = 0}^\infty
      U_k(z,z') \, \eta_+^{\frac{n-1}{2} - k} \right\} \omega_g^2(z') \, d\eta .
  \end{equation*}
  We now restrict to $t \in I_0$ so that $t > 0$.  After introducing the new
  phase variable $\theta = \left[ \dist_g(z,z') + t \right] \eta$, which we note
  is a positive multiple of the original phase variable $\eta$, this expression
  becomes
  \begin{multline*}
    \bmU(t,z,z') = \int_{\bbR_\theta} e^{i \left[ \dist_g(z,z') - t \right]
      \theta} \cdot 2 C_0 \, t \, e^{-
      \frac{i \pi (n-1)}{4}} \\
    \times \left\{ \sum_{k = 0}^\infty U_k(z,z') \cdot
      \frac{\theta_+^{\frac{n-1}{2}}}{ \left[ \dist_g(z,z') + t
        \right]^{\frac{n+1}{2}} } \cdot \frac{\theta_+^{-k}}{ \left[
          \dist_g(z,z') + t \right]^{-k}} \right\} \omega_g^2(z') \, d\theta .
  \end{multline*}
  To convert this to the Schwartz kernel of the operator acting between
  half-densities, we multiply by the factor $\omega_g(z) \, \omega_g^{-1}(z')$.
  Modulo the calculation of the constant $C_0$, this yields the desired
  representation \eqref{eq:short-time-Berard-presentation} once we insert the
  cutoff $\chi$ localizing in $\theta > 1$ (producing an overall smooth, and
  thus microlocally negligible, error).

  To finish, we briefly indicate how to calculate the constant $C_0$.  Starting
  with the classical expression for the antidifferentiated half-wave kernel on
  $\bbR^n$,
  \begin{equation*}
    \underdot{\bmU} (t,z,z') =
    \frac{\Gamma\!\left( \frac{n-1}{2} \right)}{2 \pi^{\frac{n+1}{2}}}
    \lim_{\varepsilon \downarrow 0} \left[ |z - z'|^2 - (t - i \varepsilon)^2
    \right]^{- \frac{n-1}{2}} |dz'| ,
  \end{equation*}
  we obtain the oscillatory integral representation
  \begin{equation*}
    \underdot{\bmU} (t,z,z') =
    \frac{e^{- \frac{i\pi(n-1)}{4}}}{2 \cdot \pi^{\frac{n+1}{2}}} \int_{\theta =
      0}^\infty e^{i \left[ |z - z'|^2 - (t - i 0)^2 \right] \theta } \,
    \theta^{\frac{n-3}{2}} \, |dz'| \, d\theta
  \end{equation*}
  using the distributional identity
  \begin{equation*}
    \int_{\theta = 0}^\infty e^{i (u + i 0) \theta} \, \theta^\alpha \, d\theta
    = \Gamma(\alpha + 1) \, e^{\frac{- i \pi (\alpha + 1)}{4}} \, (u + i
    0)^{-(\alpha + 1)} , \qquad u \in \bbR. 
  \end{equation*}
  Comparing this with \eqref{eq:Berard-antidiff-half-wave}, we see that
  $C_0 = \frac{1}{2} \pi^{-\frac{n+1}{2}}$, concluding the proof.
\end{proof}

\begin{corollary}
  \label{thm:short-time-reduced-Berard-principal-amplitude}
  Choose $0 < \varepsilon \ll 1$, and suppose
  $(t,z,z') \in \left(\varepsilon, \inj(Z) \right) \times W_0$.  Then there is a
  representation
  \begin{equation*}
    \bmU(t,z,z') \equiv \int_{\bbR_\theta} e^{i \left[ \dist_g(z,z') - t \right]
      \theta} \, b(z,z',\theta) \, d\theta \mod{\calC^\infty} 
  \end{equation*}
  whose amplitude
  $b \in S^{\frac{n-1}{2}} \! \left( W_0 \times \bbR_\theta ;
    |\Omega|^{\frac{1}{2}}(W_0) \right)$ is to leading order
  \begin{equation}
    \label{eq:short-time-reduced-Berard-principal-amplitude}
    \frac{e^{- \frac{i\pi(n-1)}{4}}}{(2\pi)^{\frac{n+1}{2}}} 
    \cdot \frac{\chi(\theta) \, \theta^{\frac{n-1}{2}}}{\dist_g(z,z')^{\frac{n-1}{2}}}
    \cdot \Theta^{-\frac{1}{2}}(z \to z') \cdot \omega_g(z) \, \omega_g(z')
    \mod{S^{\frac{n-1}{2}-1}} . 
  \end{equation}
\end{corollary}

\begin{proof}
  This follows from Lemma~\ref{thm:CdV-Berard-presentation-short-time} by
  introducing $s = \dist_g(z,z') - t$ into
  \eqref{eq:short-time-Berard-principal-amplitude} to replace the $t$-variable
  and applying Lemma 18.2.1 of \cite{HorIII}.
\end{proof}

To finish the base case we must introduce the microlocalizers $A_0$ and $A_1$
and compute the amplitude of $\calU^\gamma(t) = A_1 \, \calU(t) \, A_0$.  By the
above, we may write $\bmU^\gamma$ as
\begin{equation*}
  \int e^{i \psi} \, a_1(z,\zeta) \, b(z',z'',\theta) \, a_0(z''',\zeta''') \,
  d\theta d\zeta d\zeta''' dz' dz''
\end{equation*}
with
$\psi = (z - z') \cdot \zeta + \left[ \dist_g(z',z'') - t \right] \theta + (z''
- z''') \cdot \zeta'''$.
Applying the method of stationary phase in the variables
$(\zeta,\zeta''',z',z'')$ then yields this case's version of the formula
\eqref{eq:interior-gamma-microlocalization-principal-amplitude} for the
principal part of the amplitude.

We now move on to the induction step in the proof.  We assume that the
representation
\eqref{eq:interior-gamma-microlocalization-representation}--\eqref{eq:interior-gamma-microlocalization-principal-amplitude}
holds for all $\gamma$-microlocalizations of $\calU(t)$ at all points $(t,z,z')$
whose associated decomposition is into $M$ short-time propagators.  (In
particular, we assume this for all microlocalizers $A_0$ and $A_1$.)  We shall
show this representation also holds at those points requiring a decomposition into
$M+1$ short-time propagators.  Thus, let $\calU^\gamma(t)$ be a
$\gamma$-microlocalization of $\calU(t)$, and let $(t,z,z')$ be a point in the
kernel spacetime with associated decomposition
\begin{equation*}
  \calU^\gamma(t) = A_1 \, \calU(t - t_M) \, \calU(s_M) \, \cdots \calU(s_1)
  \, A_0 .
\end{equation*}
Recalling that $t_M = s_1 + \cdots + s_M$, choose $A \in \Psi^0_\upc(Z)$ to be a
microlocalizer whose symbol $a$ is identically $1$ on the set
\begin{equation*}
  G^{t_M} \! \left[ \WF'(A_0) \right] \cap G^{-(t - t_{M})} \! \left[ \WF'(A_1)
  \right] 
\end{equation*}
and which is microsupported in a small neighborhood of this set.  We may then write
\begin{equation*}
  \calU^\gamma(t) \equiv \left[ A_1 \, \calU(t - t_M) A \right] \circ \left[ A
    \, \calU(s_M) \cdots \calU(s_1) \, A_0 \right] \mod{\calC^\infty} .
\end{equation*}

If we define $\gamma_0$ and $\gamma_1$ to be the geodesics
\begin{equation*}
  \gamma_0 \defeq \gamma \big\vert_{[0,t_M]} \qquad \text{and} \qquad
  \gamma_1 \defeq \gamma \big\vert_{[t_M,T]}
\end{equation*}
and shrink the microsupports of $A_0$ and $A_1$ if necessary, then we may
arrange for $A \, \calU(s_M) \cdots \calU(s_1) A_0$ to be the decomposition of
$\calU^{\gamma_0}(t_M)$ associated to the point $(t_M,c^\flat(t_M),z')$ and for
$A_1 \, \calU(t - t_M) \, A$ to be the decomposition of
$\calU^{\gamma_1}(t-t_M)$ associated to $(t-t_M,z,c^\flat(t_M))$.  Since each
consists of fewer than $M+1$ short-time propagators, both of these operators
satisfy the induction hypothesis, and thus the expressions
\eqref{eq:interior-gamma-microlocalization-representation}-\eqref{eq:interior-gamma-microlocalization-principal-amplitude}
hold for each.  This implies that
\begin{equation}
  \label{eq:int-composite-SK}
  \bmU^\gamma (t,z,z') \equiv \int_{Z_w}
  \int_{\bbR_\theta} \int_{\bbR_\eta} e^{i \Phi} \,
  b_1(z,w,\eta) \, b_0(w,z',\theta) \, d\eta d\theta \mod{\calC^\infty} ,
\end{equation}
where the phase function is
\begin{equation*}
  \Phi = \left[ \dist_g(z,w) - (t - t_M) \right] \eta + \left[
    \dist_g^{\gamma_0} (w,z') - t_M \right] \theta
\end{equation*}
and the amplitudes satisfy
\begin{multline*}
  b_0(w,z',\theta) \equiv a_{\mathrm{L}}\!\left(w,\del_w \dist_g(w,z') \cdot
    \eta\right) a_0\!\left(z', -
    \del_{z'} \dist_g(w,z') \cdot \eta\right) \\
  \mbox{} \times \frac{e^{- \frac{i\pi(n-1)}{4}}}{(2\pi)^{\frac{n+1}{2}}} \cdot
  \frac{\chi(\eta) \, \eta^{\frac{n-1}{2}}}{\dist_g(w,z')^{\frac{n-1}{2}}} \cdot
  \Theta^{-\frac{1}{2}}(w \to z') \cdot \omega_g(w) \, \omega_g(z')
  \mod{S^{\frac{n-1}{2}-1}}
\end{multline*}
and
\begin{multline*}
  b_1(z,w,\theta) \equiv a_1\!\left(z,\del_z \dist_g^{\gamma_0}(z,w) \cdot
    \theta\right) a_\mathrm{R}\!\left(w, -
    \del_{w} \dist_g^{\gamma_0}(z,w) \cdot \theta\right) \\
  \mbox{} \times \frac{e^{- \frac{i\pi(n-1)}{4}} \,
    i^{-m_{\gamma_0}}}{(2\pi)^{\frac{n+1}{2}}} \cdot \frac{\chi(\theta) \,
    \theta^{\frac{n-1}{2}}}{\dist_g^{\gamma_0}(z,w)^{\frac{n-1}{2}}} \cdot
  \Theta^{-\frac{1}{2}}(z \to w) \cdot \omega_g(z) \, \omega_g(w)
  \mod{S^{\frac{n-1}{2}-1}} .
\end{multline*}
Here, $a_{\mathrm{L}}$ and $a_{\mathrm{R}}$ are the left and right symbols of
$A$ respectively, and we note there is a density factor in the $w$-variable so
that the integration makes sense.

We now apply the method of stationary phase to the integral
\eqref{eq:int-composite-SK} in the $(\eta,w)$-variables.  $\Phi$ is critical in
these variables precisely when
\begin{equation}
  \label{eq:int-critical-set}
  \begin{aligned}
    \del_\eta \Phi &= \dist_g(z,w) - (t - t_M) = 0 \\
    \del_w \Phi &= \eta \cdot \del_w \dist_g(z,w) + \theta \cdot \del_w
    \dist_g^{\gamma_0}(w,z') = 0 ;
  \end{aligned}
\end{equation}
on the support of the amplitude, this is where
\begin{itemize}
\item $\dist_g(z,w) = t - t_M$;
\item $\eta = \theta$; and
\item $\del_w \dist_g(z,w) = - \del_w \dist_g^{\gamma_0}(w,z')$, implying that
  $w$ lies on the geodesic $c^\flat$ between $z'$ and $z$.
\end{itemize}
Therefore, the critical set consists of the single point
$w_* \defeq c^\flat(t_M)$.  Hence, the expression \eqref{eq:int-composite-SK} is
equivalent to an oscillatory integral of the form
\begin{equation}
  \label{eq:int-composite-SK-statphase}
  \int_{\bbR_\theta} e^{i \left[ \dist^\gamma(z,z') - t \right] \theta}
  b(z,z',\theta) \, d\theta
\end{equation}
modulo smooth kernels.  The amplitude $b \in S^{\frac{n-1}{2}}$ satisfies
\begin{multline}
  \label{eq:int-composite-SK-statphase-princ-amp}
  b(z,z',\theta) \\
  \mbox{} \equiv \left. (2\pi)^{\frac{n+1}{2}} \, e^{\frac{i \pi
        \sgn(\bsfQ)}{4}} \, b_1(z,w,\eta) \, b_0(w,z',\theta) \left| \det(\bsfQ)
    \right|^{-\frac{1}{2}} \right|_* \mod{S^{\frac{n-1}{2}-1}} ,
\end{multline}
where $\bsfQ$ is the matrix
\begin{equation}
  \label{eq:int-composite-SK-statphase-quadform}
  \bsfQ \defeq 
  \begin{bmatrix}
    0 & \del_w \dist_g(z,w) \\
    \del_w \dist_g(z,w) & \eta \cdot \nabla_w^2 \dist_g(z,w) + \theta \cdot
    \nabla_w^2 \dist_g^{\gamma_0}(w,z')
  \end{bmatrix}
\end{equation}
and $(-)\big\vert_*$ denotes restriction to the critical set described above.

Observe that
$\left. \left| \det(\bsfQ) \right|^{-\frac{1}{2}} \right|_* =
\theta^{-\frac{n-1}{2}} \left| \det(\bsfQ_0) \right|^{-\frac{1}{2}}$,
where $\bsfQ_0$ is a matrix no longer depending on phase variables:
\begin{equation}
  \label{eq:int-composite-SK-statphase-quadform2}
  \bsfQ_0 \defeq 
  \begin{bmatrix}
    0 & \del_w \dist_g(z,w_*) \\
    \del_w \dist_g(z,w_*) & \nabla_w^2 \dist_g(z,w_*) + \nabla_w^2
    \dist_g^{\gamma_0}(w_*,z')
  \end{bmatrix}.
\end{equation}
Since $\theta > 0$ on the support of the amplitude, the signatures of
$\bsfQ \big\vert_*$ and $\bsfQ_0$ also agree.  Substituting in this together
with the principal parts of $b_0$ and $b_1$ and noting that $a$ is identically
$1$ on the critical set, \eqref{eq:int-composite-SK-statphase-princ-amp} becomes
\begin{multline}
  \label{eq:int-composite-SK-statphase-princ-amp-2}
  b(z,z',\theta) \equiv a_1(z, \del_z \dist^\gamma_g(z,z') \cdot \theta) \,
  a_0(z', - \del_z' \dist_g^\gamma(z,z') \cdot \theta) \\
  \mbox{} \times \frac{e^{\frac{i \pi \sgn(\bsfQ_0)}{4}} \, e^{-\frac{i \pi
        (n-1)}{2}} \, i^{- m_{\gamma_0}}}{(2\pi)^{\frac{n+1}{2}}} \cdot
  \chi(\theta)^2 \,\theta^{\frac{n-1}{2}} \cdot \omega_g(z) \,
  \omega_g(z') \\
  \times \frac{\omega_g^2(w_*) \cdot \Theta^{-\frac{1}{2}} (z \to w_*) \,
    \Theta^{-\frac{1}{2}} (w_* \to z')}{|dw| \cdot
    \dist_g^{\gamma_0}(z,w_*)^{\frac{n-1}{2}} \cdot
    \dist_g(w_*,z')^{\frac{n-1}{2}}} \left| \det(\bsfQ_0) \right|^{-\frac{1}{2}}
  \mod{S^{\frac{n-1}{2}-1}} .
\end{multline}

We compute the signature of $\bsfQ_0$ via the following lemma.

% To compute the signature of $\bsfQ_0$, we introduce a Fermi normal coordinate
% system along our geodesic $c$ as follows.  We let
% $\del_\ell \defeq (c^\flat)_* \del_s$ be the unit vector in $T_{w_*} c^\flat$
% pointing along $c^\flat$ in the increasing $s$-direction, and we choose
% $\left(\del_{\nu_1},\ldots,\del_{\nu_{n-1}}\right)$ to be an orthonormal basis
% of the orthocomplement of $\bbR \cdot \del_\ell$ in $T_{w_*} Z$, i.e., an
% orthonormal basis of the normal space
% $N_{w_*} c^\flat \defeq \left( T_{w_*} c^\flat \right)^\perp$.  Parallel
% transport yields a frame $\left( \del_\ell , \del_\nu \right)$ all along
% $c^\flat$, which we exponentiate to a coordinate system $(\nu,\ell)$ in a
% tubular neighborhood of $c^\flat$.  This system is orthogonal along $c^\flat$,
% and its origin is $w_*$.  In these coordinates, $\bsfQ_0$ becomes
% \begin{equation}
%   \label{eq:Hessian-Fermi-coords}
%   \bsfQ_0 =
%   \begin{bmatrix}
%     0 & 0 & 1 \\
%     0 & \bsfA(0) & 0 \\
%     1 & 0 & 0
%   \end{bmatrix} ,
% \end{equation}
% where $\bsfA(\nu)$ is the Hessian in the $\nu$-variables of a broken distance
% function:
% \begin{equation}
%   \label{eq:nu-Hessian-2-composite}
%   \bsfA(\nu) \defeq \del_\nu^2 \left[ \dist_g(z;\nu,0) +
%     \dist_g^{\gamma_0}(\nu,0;z') \right] .
% \end{equation}
% A simple calculation shows the signature of $\bsfQ_0$ is the same as that of
% $\bsfA(0)$, the latter of which we now compute.

\begin{lemma}
  \label{thm:int-signature}
  Let $c^\flat : [0,T] \To Z$ be a geodesic in $Z$ with endpoints
  $z_1 \defeq c^\flat(0)$ and $z_2 \defeq c^\flat(T)$.  For $0 < S < T$ with
  $T - S < \inj(Z)$, let \begin{equation*} c_1^\flat \defeq c^\flat
    \Big\vert_{[0,S]} \qquad \text{and} \qquad c_2^\flat \defeq c^\flat
    \Big\vert_{[S,T]}
  \end{equation*}
  be a decomposition of $c^\flat$ with $w_* \defeq c^\flat(S)$ the common
  endpoint.  Assume $w_*$ is not conjugate to either $z_1$ or $z_2$ along
  $c^\flat$, and assume also that there is at most one conjugate point to $z_1$
  along $c_2^\flat$.  Writing $m_c$ and $m_{c_1}$ for the Morse
  indices\footnote{Note that the Morse index of $c_2^\flat$ vanishes since it is
    a distance-minimizing geodesic, i.e., $m_{c_2} = 0$.} of $c^\flat$ and
  $c^\flat_1$, respectively, we have
  \begin{equation}
    \label{eq:morse-index-additivity}
    m_c = m_{c_1} + \ind \! \left( \left. \Hess \! \left[
          \dist_g^{c_1}(z_1,w) + \dist_g^{c_2}(w,z_2) \right] \right|_{w=w_*}
    \right) ,
  \end{equation}
  where $\ind(-)$ is the index of a quadratic form, i.e., the sum of the
  dimensions of the eigenspaces associated to negative eigenvectors.
\end{lemma}

\begin{proof}
  Define $m_* \defeq m_c - m_{c_1}$, and set
  \begin{equation*}
    i_* \defeq \ind \! \left( \left. \Hess \! \left[
          \dist_g^{c_1}(z_1,w) + \dist_g^{c_2}(w,z_2) \right] \right|_{w=w_*}
    \right) .
  \end{equation*}
  We first claim that $m_* \leqslant i_*$.  If $m_* = 0$, this is trivially
  true, so we assume $m_* \geqslant 1$, i.e., that there exists a conjugate
  point $z_\upc = c^\flat_2(s_\upc)$ to $z_1$ along $c_2^\flat$.  Therefore, we
  may choose independent normal Jacobi fields
  $\left( \vec{J}_1, \ldots, \vec{J}_{m_*} \right)$ along $c^\flat$ such that
  $\vec{J}_j(0) = \vec{J}_j(s_\upc) = \vec{0}$.  We arrange that their values at
  $w_*$ (i.e., at $s = S$) are $g$-orthonormal, and we extend these to a maximal
  collection $\left( \vec{J}_1, \ldots, \vec{J}_{n-1} \right)$ whose values at
  $w_*$ form an orthonormal basis of $N_{w_*} c^\flat$.  These values then
  define a Fermi normal coordinate system $(\nu,\ell)$ along $c^\flat$; we use
  these coordinates to calculate $i_*$.

  For each $j = 1,\ldots,n-1$ let $\vec{V}_j(s)$ be the unique broken (normal)
  Jacobi field along $c^\flat$ satisfying
  \begin{equation*}
    \vec{V}_j(0) = \vec{V}_j(T) = \vec{0} \qquad \text{and} \qquad \vec{V}_j(S)
    = \vec{J}_j(S) .
  \end{equation*}
  We construct the variation of $c^\flat$ with respect to these broken Jacobi
  fields $\bsfV = \left( \vec{V}_1, \ldots, \vec{V}_{n-1} \right)$:
  \begin{equation*}
    c^\flat_{\bsfV}(s;\nu) \defeq \exp_{c^\flat(s)} \! \left[ \nu_1 \,
      \vec{V}_1(s) + \cdots + \nu_{n-1} \, \vec{V}_{n-1}(s) \right].
  \end{equation*}
  By definition, $c^\flat_{\bsfV}(s;\nu)$ agrees to first order with the path realizing the broken
  distance between $z_1$ and $z_2$ with intermediate point $(\nu, \ell = 0)$ in
  the local manifold $\{\ell = 0\}$ transverse to $c^\flat.$  Since the
  second variation formula involves only the first derivative of the
  variation, this implies that
  \begin{equation}
    \label{eq:broken-distance-variation}
    \length \! \left( c^\flat_{\bsfV}(\cdot;\nu) \right) = \dist_g^{c_1}(z_1 ;
    \nu,0) + \dist_g^{c_2}(\nu,0; z_2) +O(\nu^3).
  \end{equation}
  By putting the Hessian in \eqref{eq:morse-index-additivity} into our Fermi
  normal coordinate system, we see that
  \begin{equation*}
    i_* = \ind \! \left( \left. \del_\nu^2 \left[ \dist_g^{c_1}(z_1;\nu,0) +
          \dist_g^{c_2}(\nu,0;z_2) \right] \right|_{\nu = 0} \right) ,
  \end{equation*}
  that is, the index in question is the same as the index of the Hessian in the
  $\nu$-variables only.

  We show the existence of a negative eigenvalue of the $\nu$-Hessian of
  \eqref{eq:broken-distance-variation} by using a standard piece of Riemannian
  geometry.  For any $\varepsilon > 0$ small, we may construct a smooth vector
  field $\vec{X}_1(s)$ along $c^\flat$ satisfying the following:
  \begin{enumerate}[(i)]
  \item $\vec{X}_1(s)$ agrees with $\vec{J}_1(s)$ for
    $0 \leqslant s < s_\upc - \varepsilon$;
  \item $\vec{X}_1(s) \equiv \vec{0}$ for
    $s_\upc + \varepsilon < s \leqslant T$; and
  \item the variation $c^\flat_{\vec{X}_1}(s;\nu_1)$ is shorter than $c^\flat$
    for small $\nu_1 > 0$.
  \end{enumerate}
  (For the details, see, e.g., \cite{Jos}*{Theorem 5.3.1}.)  By choosing
  $\varepsilon < s_\upc - S$, we produce a variation of $c^\flat$ which agrees
  with the broken variation $c^\flat_{\bsfV}(s;\nu_1,0,\ldots,0)$ obtained
  previously for $0 \leqslant s \leqslant S$.  On the other hand, as long as
  $\nu_1$ is small, we are guaranteed that $c^\flat_{\bsfV}(s;\nu_1,0,\ldots,0)$
  is shorter than the injectivity radius of $Z$ for $S \leqslant s \leqslant T$
  since $c^\flat_2$ has this property.  This implies that it is distance
  minimizing for these values of $s$, which in turn implies
  \begin{equation*}
    \length \! \left( c^\flat_{\bsfV}( \cdot ; \nu_1, 0, \ldots, 0) \right)
    \leqslant \length \! \left( c^\flat_{\vec{X}_1}(\cdot;\nu_1) \right) <
    \length \! \left( c^\flat \right) .
  \end{equation*}
  Hence, the length of $c^\flat_{\bsfV}(\cdot;\nu)$ is decreasing in the
  $\nu_1$-direction at $\nu = 0$, and therefore the Hessian
  $\left. \del_\nu^2 \left[ \dist_g^{c_1}(z_1;\nu,0) + \dist_g^{c_2}(\nu,0;z_2)
    \right] \right|_{\nu = 0}$
  is negative in that direction.  This implies the existence of a negative
  eigenvalue $\mu_1 < 0$ and an associated $\mu_1$-eigenvector $\vec{v}_1$ of
  that Hessian, as desired.

  To generate the remaining eigenvectors of the Hessian, we apply this argument
  inductively.  If $\left( \vec{v}_1, \ldots, \vec{v}_\ell \right)$ are distinct
  eigenvectors of the Hessian associated to negative eigenvalues, then we may
  find independent Jacobi fields
  $\left( \vec{J}_1^\sharp, \ldots, \vec{J}_{n-\ell-1}^\sharp \right)$ in the
  span of our original collection
  $\left( \vec{J}_1, \ldots, \vec{J}_{n-1} \right)$ such that
  $\vec{J}_j^\sharp(s_\upc) = \vec{0}$ for at least $j = 1, \ldots, m_* - \ell$
  and whose values at $s = S$ are an orthonormal basis of the
  $g$-orthocomplement of the span of
  $\left( \vec{v}_1, \ldots, \vec{v}_\ell \right)$ in $N_{w_*} c^\flat$.  Thus,
  as long as $m_* - \ell > 0$, our previous argument shows there is another
  negative eigenvalue.  Altogether, this shows that $m_* \leqslant i_*$,
  finishing the first claim.

  We now show the remaining inequality, $i_* \leqslant m_*$.  Suppose that
  $\left( \vec{v}_1,\ldots,\vec{v}_{n-1} \right)$ are the distinct eigenvectors
  of our Hessian, the first $i_*$ of which are associated to negative
  eigenvalues $\mu_j$.  We may then use these vectors as the coordinate vector
  fields $\del_{\nu_j}$ at $w_*$, extending along $c^\flat$ as before.  For each
  $j$, we let $\vec{V}_j(s)$ be the broken Jacobi field along $c^\flat$
  satisfying
  \begin{equation*}
    \vec{V}_j(0) = \vec{V}_j(T) = \vec{0} \quad \text{and} \quad \vec{V}_j(S) =
    \vec{v}_j ,
  \end{equation*}
  and we again construct the (approximate) joint broken variation $c^\flat_{\bsfV}(s;\nu)$
  of $c^\flat$ coming from these broken Jacobi fields.  Thus
  \begin{equation*}
    \left. \del_{\nu_j}^2 \length\! \left( c^\flat_{\vec{V}_j} \! \left(\cdot; \nu_j
        \right) \right) \right|_{\nu_j} = \mu_j ,
  \end{equation*}
  which is negative when $j = 1,\ldots,i_*$.

  Finally, let $z_{\upc(\kappa)} = c^\flat(s_{\upc(\kappa)})$ be the points
  which are conjugate to $z_1$ along $c^\flat$ for times
  $0 < s_{\upc(\kappa)} < S$, and let $\vec{M}_{\kappa,j}(s)$ be independent
  broken Jacobi fields along $c^\flat$ such that
  $\vec{M}_{\kappa,j}(0) = \vec{M}_{\kappa,j}(s_{\upc(\kappa)}) = \vec{0}$ and
  $\vec{M}_{\kappa,j} (s) \equiv \vec{0}$ for
  $s_{\upc(\kappa)} \leqslant s \leqslant T$.  There are exactly $m_{c_1}$ of
  these broken Jacobi fields.  Since our vector fields $\vec{V}_j$ are nonzero
  at $s = S$, they do not lie in the span of the $\vec{M}_{\kappa,j}$'s.
  Moreover, since the index form is negative on $\vec{V}_j$ for $j = 1,\ldots,m$,
  there must be a conjugate point $z_\upc = c^\flat(s_\upc)$ for some
  $S < s_\upc < T$.  This conjugate point is unique by our assumptions on the
  decomposition of $[0,T]_s$, so it must have multiplicity $m_*$.  This concludes
  the proof.
\end{proof}

As the signature of a matrix is invariant under small pertubations, we
may adjust the microsupports of $A_0$, $A_1$, and $A$ to ensure that
$m_{\gamma_0} = m_{c_0}$, where we recall that $m_{\gamma_0}$ is the Morse index
of the geodesic segment $\gamma_0^\flat$.  Similarly, we may arrange that
$m_\gamma = m_{c}$, and hence
\begin{equation}
  \label{eq:int-signature-relation}
  e^{\frac{i \pi \sgn(\bsfQ_0)}{4}} \, e^{- \frac{i \pi (n-1)}{2}} \,
  i^{-m_{\gamma_0}} = e^{- \frac{i\pi(n-1)}{4}} \, i^{-m_\gamma} .
\end{equation}

The remainder of the proof of Theorem~\ref{thm:interior-ampl} deals with the
third line of \eqref{eq:int-composite-SK-statphase-princ-amp-2}.

\begin{lemma}
  \label{thm:int-determinant}
  We have the following identifications:
  \begin{equation}
    \label{eq:int-amplitude-det-factor}
    \frac{\omega_g^2(w_*) \cdot \Theta^{-\frac{1}{2}} (z \to w_*) \,
      \Theta^{-\frac{1}{2}} (w_* \to z')}{|dw| \cdot
      \dist_g^{\gamma_0}(z,w_*)^{\frac{n-1}{2}} 
      \cdot \dist_g(w_*,z')^{\frac{n-1}{2}}} \cdot \left| \det(\bsfQ_0)
    \right|^{-\frac{1}{2}} = \frac{\Theta^{-\frac{1}{2}}(z \to
      z')}{\dist^\gamma(z,z')^{\frac{n-1}{2}}}
  \end{equation}
  and
  \begin{equation}
    \label{eq:Theta-equivalence}
    \Theta^{-\frac{1}{2}}(z \to z') = \Theta^{-\frac{1}{2}}(z' \to z) .
  \end{equation}
\end{lemma}

\begin{proof}
  We shall express the quantities in \eqref{eq:int-amplitude-det-factor} and
  \eqref{eq:Theta-equivalence} in terms of Jacobi endomorphisms along
  $TZ \big|_{c^\flat}$.  To set these up, let $(\nu',\ell')$ be a Fermi normal
  coordinate system along $c^\flat$ based at $z' = c^\flat(0)$, and let
  $\bsfJ(s)$ be the Jacobi endomorphism satisfying
  \begin{equation*}
    \bsfJ(0) = \vec{0} \quad \text{and} \quad \skew{3}{\dot}{\bsfJ}(0)
    = \Id .
  \end{equation*}
  Now, let $\bsfL(s)$ be the analogous Jacobi endomorphism arising from a Fermi
  normal coordinate system $(\nu_*,\ell_*)$ based at $w_* = c^\flat(t_M)$, and
  let and $\bsfK(s)$ be that coming from a Fermi normal coordinate system
  $(\nu,\ell)$ based at $z = c^\flat(t)$.  We may then identify
  $\Theta(z \to w_*)$ as
  \begin{equation*}
    \Theta(z \to w_*) \defeq \left| \left. \det_g \! \left[ D \exp_z (-) \right]
      \right|_{-(t - t_M) \, \del_{\ell}} \right| = (t - t_M)^{-(n-1)} \left|
      \det \bsfK^\perp(t_M) \right|
  \end{equation*}
  using \eqref{eq:Jacobi-endo-exp-2} and the fact that
  $\bsfK^\parallel(t_M) \cdot \del_\ell = (t - t_M) \cdot \del_\ell$.
  Similarly,
  $\Theta(w_* \to z') = t_M^{-(n-1)} \left| \det \bsfL^\perp(0) \right|$, and
  from the Wronskian identity
  \begin{equation*}
    \left|\det \bsfL^\perp(0) \right| = \left| \det
      \calW(\bsfJ^\perp,\bsfL^\perp) \big\vert_{s = 0} \right| = 
    \left| \det \calW(\bsfL^\perp,\bsfJ^\perp) \big\vert_{s = t_M} \right| =
    \left| \det \bsfJ^\perp(t_M) \right| 
  \end{equation*}
  we may rewrite this as
  $\Theta(w_* \to z') = t_M^{-(n-1)} \left| \det \bsfJ^\perp(t_M) \right|$.
  Hence, upon identifying the distance factors with values of $s \in [0,t]$, we
  have
  \begin{equation}
    \label{eq:det-lemma-1}
    \frac{\Theta^{-\frac{1}{2}} (z \to w_*) \,
      \Theta^{-\frac{1}{2}} (w_* \to z')}{\dist_g^{\gamma_0}(z,w_*)^{\frac{n-1}{2}} 
      \cdot \dist_g(w_*,z')^{\frac{n-1}{2}}} = \left| \det \! \left[
        \bsfJ^\perp(t_M) \cdot \bsfK^\perp(t_M) \right] \right|^{-\frac{1}{2}} .
  \end{equation}

  We now turn to the factor
  $\frac{\omega_g^2(w_*)}{|dw|} \cdot \left| \det \bsfQ_0 \right|^{-\frac{1}{2}}
  = \left| \det \! \left[ g^{-1}(w_*) \cdot \bsfQ_0 \right]
  \right|^{-\frac{1}{2}}$.
  Observe that the lower right block of
  \eqref{eq:int-composite-SK-statphase-quadform2} may be viewed
  invariantly\footnote{This makes sense as a tensor since the gradient
    of the sum of distances vanishes along $c^\flat$.} as
  \begin{equation*}
    \left. \nabla^2_w d_w \left[ \dist_g(z,w) + \dist_g^{\gamma_0}(w,z') \right]
    \right|_{N_{w_*} c^\flat \otimes N_{w_*} c^\flat} ,
  \end{equation*}
  the Hessian of $\dist_g(z,w) + \dist_g^{\gamma_0}(w,z')$ acting on normal
  vectors to our geodesic $c^\flat$ at the critical point $w_*$, and therefore
  \begin{equation*}
    \left| \det \! \left[ g^{-1}(w_*) \cdot \bsfQ_0 \right] \right| = \left| \det
      \! \left[ \left. \left[ \Hess \! \left[ \dist_g(z,\cdot) \right]^\sharp +
            \Hess \! \left[ \dist_g^{\gamma_0}(\cdot,z') \right]^\sharp \right]
        \right|_{N_{w_*} c^\flat} \right] \right| ,
  \end{equation*}
  where the raising operator $(-)^\sharp$ converts these Hessians into
  endomorphisms of $N_{w_*}c^\flat$ via $\Hess(-)^\sharp = \nabla \grad (-)$.
  These gradients are
  \begin{equation*}
    \grad_w\!\left( \dist_g(z,w) \right) = \del_\ell \quad \text{and} \quad
    \grad_w\!\left( \dist_g^{\gamma_0}(w,z') \right) = \del_{\ell'} ,
  \end{equation*}
  which are the radial vector fields for the geodesic spheres with centers $z$
  and $z'$ respectively when restricted to $c^\flat$.  Applying the connection
  then returns the shape operators for these geodesic spheres, and as described
  in \cite{KowVan}, we may express these in terms of our Jacobi endomorphisms
  via
  \begin{equation}
    \label{eq:shape-operator-Jacobi-field}
    \begin{aligned}
      \left. \Hess \! \left[ \dist_g^{\gamma_0}(\cdot,z') \right]^\sharp
      \right|_{N_{w_*}c^\flat} &= \bsfJdot^\perp(t_M) \cdot
      (\bsfJ^\perp)^{-1}(t_M) \\
      \left. \Hess \! \left[ \dist_g(z,\cdot) \right]^\sharp \right|_{N_{w_*}
        c^\flat} &= - \dot{\bsfK}^\perp(t_M) \cdot (\bsfK^\perp)^{-1}(t_M) ,
    \end{aligned}
  \end{equation}
  where the minus sign arises from the different orientations of the two
  coordinate systems along $c^\flat$.

  We conclude by putting these calculations together.  Firstly, note that
  $\dot{\bsfK}^\perp(t_M) \cdot (\bsfK^\perp)^{-1}(t_M)$ is a symmetric
  endomorphism since $\calW\!\left(\bsfK^\perp,\bsfK^\perp\right) = \vec{0}$.
  Therefore, by \eqref{eq:det-lemma-1} and
  \eqref{eq:shape-operator-Jacobi-field} we compute that the left-hand side of
  \eqref{eq:int-amplitude-det-factor} is equal to
  \begin{multline*}
    \left| \det \! \left[ \bsfJdot^\perp(t_M) \cdot (\bsfJ^\perp)^{-1}(t_M) -
        \dot{\bsfK}^\perp(t_M) \cdot (\bsfK^\perp)^{-1}(t_M) \right] \cdot \det
      \! \left[ \bsfJ^\perp(t_M) \cdot (\bsfK^\perp)^\upt(t_M) \right]
    \right| \\
    \mbox{} = \left| \det \! \left[ \bsfJdot^\perp(t_M) - \dot{\bsfK}^\perp(t_M)
        \cdot (\bsfK^\perp)^{-1}(t_M) \cdot \bsfJ^\perp(t_M) \right] \cdot \det
      \! \left[ (\bsfK^\perp)^\upt(t_M) \right] \right| \\
    \mbox{} = \left| \det \! \left[ \bsfJdot^\perp(t_M) -
        \left((\bsfK^\perp)^\upt\right)^{-1}\!(t_M) \cdot
        (\dot{\bsfK}^\perp)^\upt(t_M) \cdot \bsfJ^\perp(t_M) \right] \cdot \det
      \! \left[ (\bsfK^\perp)^\upt(t_M) \right] \right| \\
    \mbox{} = \left| \det \! \left[ (\bsfK^\perp)^\upt(t_M) \cdot
        \bsfJdot^\perp(t_M) - (\dot{\bsfK}^\perp)^\upt(t_M) \cdot
        \bsfJ^\perp(t_M) \right] \right| ,
  \end{multline*}
  and this is exactly
  $\left| \det \! \left[ \left. - \calW\!\left(\bsfK^\perp, \bsfJ^\perp\right)
      \right|_{s = t_M} \right] \right|$.
  By constancy of the Wronskian, this is equal to
  $\left| \det \calW\!\left(\bsfK^\perp,\bsfJ^\perp\right) \big\vert_{s = 0}
  \right| = \left| \det \bsfK^\perp(0) \right|$
  and
  $\left| \det \calW\!\left(\bsfK^\perp,\bsfJ^\perp\right) \big\vert_{s = t}
  \right| = \left| \det \bsfJ^\perp(t) \right|$,
  which is what we wanted to show.
\end{proof}

% SECTION 6

% fw-conetrace-6.tex Section 6:  Proof of Theorem \ref{thm:amplitude-mult-diff}

\section{Proof of Theorem \ref{thm:amplitude-mult-diff}}
\label{sec:proof-theorem-mult-diff}

% Induction on $k$, the number of diffractions.  Base case:  compose the cone
% point propagator with an interior propagator on either side.

Let $\gamma : [0,T] \To \Sbstar X$ be a fixed broken geodesic with partition as
described in Section~\ref{sec:gamma-microlocalization}, and let
$\calU^\gamma(t)$ be a corresponding $\gamma$-microlocalization of the half-wave
group on $X$ as shown in \eqref{eq:gamma-microlocalization-hw}.  We shall prove
Theorem~\ref{thm:amplitude-mult-diff} by induction on $k$, the number of
diffractions the geodesic $\gamma$ undergoes.  As in the smooth case, we begin
by fixing an element $(t,z,z') \in \bbR_+ \times U_1^\flat \times U_0^\flat$.
To each point for which the Schwartz kernel
$\bmU^\gamma \defeq \calK\!\left[ \calU^\gamma(t) \right]$ is singular we
associate a broken geodesic $c : [0,t]_s \To \Sbstar Z$ with
$c(0) = (z',\zeta')$ and $c(t) = (z,\zeta)$ which we use in the calculation.
Note that we need only consider points $(t,z,z')$ whose associated geodesics $c$
satisfy
\begin{equation*}
  c(T_m) \in \WF'(A_m)
\end{equation*}
for all $m = 1,\ldots,2k+1$:  if this fails, then $\bmU^\gamma$ will be smooth
at $(t,z,z')$.  By adjusting these microlocalizers if necessary, we may ensure
that none of the points $c^\flat(t_m)$ are conjugate to one another since the
points $\gamma^\flat(T_m)$ of the reference geodesic are not conjugate to one
another.  Lastly, we define the broken geodesic segments
\begin{equation*}
  c_m \defeq c \big\vert_{[T_m,T_{m+1}]} \text{ for $m = 0, \ldots, 2k-1$ and } 
  c_{2k} \defeq c \big\vert_{[T_{2k},t]} .
\end{equation*}

We start with the case of a single diffraction, $k = 1$.  Recall that the
interim times are $t_m \defeq T_m - T_{m-1}$.

\begin{lemma}
  Let
  $\calU^\gamma(t) = A_3 \, \calU(t - T_2) \, A_2 \, \calU(t_2) \, A_1 \,
  \calU(t_1) \, A_0$
  be a $\gamma$-microlocalization of the half-wave group undergoing a single
  diffraction through the cone point $Y_\alpha$.  Then $\bmU^\gamma$ has a
  representation
  \begin{equation}
    \label{eq:1-diff-gammamicrolocalization-oscil-int}
    \bmU^\gamma(t,z,z') \equiv \int_{\bbR_\xi} e^{i \left[
        \dist_g^{\gamma_1}(z,Y_\alpha) + \dist_g^{\gamma_0}(Y_\alpha,z') - t
      \right] \xi} \, b(t,z,z',\xi) \, d\xi \mod{\calC^\infty}
  \end{equation}
  with amplitude
  $b \in S^0_\upc \! \left(\bbR_t \times U_1^\flat \times U_0^\flat \times
    \bbR_\xi ; |\Omega|^{\frac{1}{2}} (U_1^\flat \times U_0^\flat) \right)$
  given by
  \begin{multline}
    \label{eq:1-diff-gamma-microlocalization-principal-amplitude}
    \bm{a}(z,z',\xi) \cdot \frac{i^{-m_{\gamma_0}} \, i^{-m_{\gamma_1}}}{2\pi i}
    \cdot \frac{\chi(\xi) \cdot \bmD_\alpha(\Pi_\alpha(z), \Pi_\alpha(z'))}{
      \dist_g^{\gamma_1}(z,Y_\alpha)^{\frac{n-1}{2}} \,
      \dist_g^{\gamma_0}(Y_\alpha,z')^{\frac{n-1}{2}}} \\
    \mbox{} \times \Theta^{-\frac{1}{2}}(z' \to Y_\alpha) \,
    \Theta^{-\frac{1}{2}}(Y_\alpha \to z) \cdot \omega_g(z) \, \omega_g(z')
  \end{multline}
  modulo elements of $S^{-\frac{1}{2} + 0}$, where
  $\bm{a} \in S^0_\upc \!  \left( U_1^\flat \times U_0^\flat \times \bbR_\xi ;
    \HD(U_1^\flat \times U_0^\flat) \right)$
  is the combined amplitude of the microlocalizers:
  \begin{multline*}
    \bm{a}(z,z',\xi) \defeq \bigg[ a_3\!\left(z,\del_z \dist_g^{c_2}(z,w)
      \cdot \xi\right) \, a_2(w, - \del_{w} \dist_g(w,Y_\alpha) \cdot \xi) \\
    \mbox{} \times a_1(w',\del_{w'} \dist_g(Y_\alpha,w') \cdot \xi) \,
    a_0\!\left(z', - \del_{z'} \dist_g^{c_0}(w',z') \cdot \xi\right) \bigg]
    \bigg\vert_{ \substack{ w = c_1^\flat(t_2) \\ w' = c_0^\flat(t_1) } }.
  \end{multline*}
\end{lemma}

\begin{proof}
  We shall apply the method of stationary phase to the oscillatory integral
  \begin{equation*}
    \bmU^\gamma \equiv \int e^{i \psi} \, b_1(z,w,\theta) \, d(t_2,w,w',\xi) \,
    b_0(w',z',\theta') \, d\theta d\xi d\theta' \mod{\calC^\infty}
  \end{equation*}
  in the $(\theta,w;\theta',w')$-variables, where $\psi$ is the phase function
  \begin{multline*}
    \psi = \left[ \dist_g^{c_2}(z,w) - (t - T_2) \right] \theta + \left[
      \dist_g(w,Y_\alpha) + \dist_g(Y_\alpha,w') - t_2 \right] \xi \\
    \mbox{} + \left[ \dist_g^{c_0}(w',z') - t_1 \right] \theta' ,
  \end{multline*}
  $b_1$ is the amplitude of $A_3 \, \calU(t - T_2) \, A_2$, $d$ is the amplitude
  of the diffracting propagator $\calU(t_2)$, and $b_0$ is the amplitude of
  $A_1 \, \calU(t_1) \, A_0$.  Here, we identify $x(z) = \dist_g(z,Y_\alpha)$
  for $z \in C_\alpha^\circ$, and we use the standardized oscillatory integral
  and amplitude formulations in Theorem~\ref{thm:nonproduct-hw-amp} and
  Theorem~\ref{thm:interior-ampl}.

  The stationary phase calculations proceed, \emph{mutatis mutandis}, as those
  of the interior case in Section~\ref{sec:proof-theorem-int-ampl}.  As before,
  the critical set in the $(\theta',w')$-variables is
  $\left\{ w'_* \defeq c_0^\flat(t_1) \text{ and } \theta' = \xi \right\}$.  The
  signature and determinant of the Hessian of $\psi$ in these variables are
  calculated following Lemmas~\ref{thm:int-signature} and
  \ref{thm:int-determinant} with the only change being notational---the smooth
  Jacobi fields become cone Jacobi fields.  (Note in particular that
  the analogous expression to \eqref{eq:det-lemma-1} yields the
  interesting relationship between the $\Theta(Y_\alpha\to Y_\beta)$ factors and
  determinants of differences of shape operators alluded to earlier.)

 Following these calculations with
  those arising from the $(\theta,w)$-variables then concludes the proof.
\end{proof}

To conclude the section, we go through the induction step of composing a
propagator undergoing $k - 1$ diffractions with a propagator undergoing a single
diffraction as in the previous lemma:
\begin{equation*}
  \calU^\gamma(t) = \left[ A_{2k+1} \, \calU(t - T_{2k}) \, A_{2k} \right] \circ
  \calU(t_{2k}) \circ \left[ A_{2k-1} \, \calU(t_{2k-1}) \, A_{2k - 2} \cdots A_1
    \, \calU(t_1) \, A_0 \right] .
\end{equation*}
The Schwartz kernel of this $\gamma$-microlocalization may be represented as
\begin{equation*}
  \bmU^\gamma \equiv \int e^{i \psi} \, b_1(z,w,\theta) \, d(t_{2k},w,w',\xi) \,
  b_0(T_{2k-1}, w', z', \theta') \, d\theta d\xi d\theta' \mod{\calC^\infty} 
\end{equation*}
with the phase function being
\begin{multline*}
  \psi = \left[ \dist_g^{c_{2k}}(z,w) - (t - T_{2k}) \right] \theta + \left[
    \dist_g(w,Y_{\alpha_k}) + \dist_g(Y_{\alpha_k},w') - t_{2k} \right] \xi \\
  \mbox{} + \left[ \dist^{\tilde\gamma_{k-1}}_g(w', Y_{\alpha_{k-1}}) + \sum_{j
      = 2}^{k-1} \dist_g^{\gamma_{j-1}}(Y_{\alpha_j},Y_{\alpha_{j-1}}) +
    \dist_g^{\gamma_0}(Y_{\alpha_1},z') - T_{2k-1} \right] \theta'
\end{multline*}
and with $b_1$ the amplitude of $A_{2k+1} \, \calU(t - T_{2k}) \, A_{2k}$, $d$
the amplitude of the diffracting propagator $\calU(t_{2k})$, and $b_0$ the
amplitude of
$A_{2k-1} \, \calU(t_{2k-1}) \, A_{2k-2} \cdots A_1 \, \calU(t_1) \, A_0$ as in
\eqref{thm:amplitude-mult-diff}.  In this phase function, we set
$\tilde\gamma_{k-1}^\flat$ to be the segment of $\gamma^\flat$ stretching from
$Y_{\alpha_{k-1}}$ to $\gamma^\flat(T_{2k-1})$.  We apply the method of
stationary phase in the $(\theta,w;\theta',w')$-variables.  The calculations for
the $(\theta,w)$-variables are exactly as before, and those coming from the
$(\theta',w')$-variables are different only in that the critical set forces the
path taken from $Y_{\alpha_{k-1}}$ to $Y_{\alpha_k}$ to be the geodesic segment
$\gamma_{k-1}$ (which is the \emph{unique} geodesic connecting the cone points)
and the introduction of an overall factor of
$e^{\frac{i \pi (n-1)}{4}} \cdot (2\pi)^{\frac{n+1}{2}}$.  This ends the proof
of Theorem~\ref{thm:amplitude-mult-diff}.

% SECTION 7

% fw-conetrace-7.tex Section 7:  A microlocal partition of unity

\section{A microlocal partition of unity}
\label{section:partition}

In this section, we develop a microlocal partition of unity on our conic
manifold $X$ which is adapted to the diffractive part of the geodesic
flow; such methods have been previously used by Hillairet
\cite{Hillairet:Contribution} and by the second author \cite{Wunsch:Poisson}.
This is a preparatory step in the analysis of the wave trace at the length of
strictly diffractive closed geodesic $\gamma$, ultimately allowing us to
microlocalize $\calU(t)$ in such a way that we may apply the calculation of
Theorem~\ref{thm:amplitude-mult-diff}.  In particular, we will be able to
microlocalize $\calU(t)$ to the extent where we need only consider a single
diffraction through one of the cone points $Y_\alpha$ at a time.

First, fix $\ell$ to be the minimal distance between the cone points of $X$:
\begin{equation*}
  \ell \defeq \min \left\{ \dist_g(Y_\alpha, Y_\beta) : \alpha,\beta =
    1,\ldots,N \right\} .
\end{equation*}
Let $\delta_{\mathrm{cone}} > 0$ be a positive constant satisfying
$\delta_{\mathrm{cone}} < \min \! \left( \frac{x_*}{9}, \frac{1}{10} \, \ell
\right)$.\footnote{Thus
  the set $\{ x < \delta_{\mathrm{cone}}\}$ is well-defined.}  For each
$\alpha = 1,\ldots,N$, we choose $\psi_\alpha \in \calC^\infty_\upc(X)$ to be a
nonnegative bump function such that $\psi_\alpha \equiv 1$ near the cone point
$Y_\alpha$ and whose support is contained in the neighborhood
$\left\{ x_\alpha < \delta_{\mathrm{cone}} \right\}$ of $Y_\alpha$.  Multiplying
by $\psi_\alpha$ then localizes within this neighborhood of $Y_\alpha$.  We call
these multiplication operators $\left\{\psi_\alpha\right\}_{\alpha = 1}^N$ the
\emph{cone localizers}.

Next, let $\left\{ A_j \right\}_{j = 1}^J \subseteq \Psi^0(X^\circ)$ be a finite
collection of pseudodifferential operators on the interior of $X$ possessing the
following four properties:
\begin{enumerate}[(i)]
\item Each $A_j$ has compactly supported Schwartz kernel, i.e.,
  $A_j \in \Psi^0_\upc(X^\circ)$;
\item For a fixed constant $\delta_\mathrm{int} > 0$, the microsupport
  $\WF'(A_j) \subseteq S^* X^\circ$ of each $A_j$ is contained in a ball of
  radius $\delta_\mathrm{int}$ with respect to an overall fixed Finsler metric
  on the interior cosphere bundle; and
\item The $A_j$'s complete the cone localizers to a microlocal partition of
  unity in the sense\footnote{In particular, this error here is smoothing and
    compactly supported \emph{in the interior}.} that
  \begin{equation*}
    \Id - \left[ \sum_{\alpha = 1}^N \psi_\alpha + \sum_{j = 1}^J A_j \right]
    \in \Psi^{-\infty}_\upc(X^\circ) .
  \end{equation*}
\item The $A_j$ have square roots modulo smoothing errors: there
  exists pseudodifferential operators, denoted $\sqrt{A_j} \in \Psi_c^0(X^\circ),$ such
  that $(\sqrt{A_j})^2-A_j \in \Psi_c^{-\infty}(X^\circ).$  The
 cutoffs $\psi_\alpha$ also have smooth square roots.
\end{enumerate}
We call these operators $\left\{ A_j \right\}_{j=1}^J$ the \emph{interior
  localizers}.  Note that we may microlocalize in the interior more finely by
adjusting the parameter $\delta_{\mathrm{int}}$ or by choosing a finite number
of the interior localizers as desired (while, of course, keeping in mind that
they must form a microlocal partition of unity).  In what follows, we shall
write $B_\bullet$ for an operator which is either a cone localizer or an
interior localizer.

Now, let us fix a strictly diffractive closed geodesic $\gamma$ of period $T$;
thus, $\gamma$ is a piecewise smooth curve on $X$ having jumps within some
boundary component $Y_\alpha$ at each time of discontinuity.  We subdivide its
principal domain $[0,T]$ as
\begin{equation*}
  0 \defeq T_0 < T_1 < \cdots < T_M \defeq T,
\end{equation*}
requiring the lengths of these subintervals be sufficiently short:
\begin{equation*}
  t_m \defeq T_{m} - T_{m-1} < \frac{\ell}{10} \ \text{for each $m = 0,
    \ldots, M$.}
\end{equation*}
As $\sum_{k = 0}^m t_k = T_{m}$, for all times $t \in \bbR$ we have
\begin{equation*}
  \calU(t) = \calU(t - T_{M-1}) \, \calU(t_{M-1}) \, \calU(t_{M-2}) \cdots
  \calU(t_1) .  
\end{equation*}

To interweave our operators $B_\bullet$ and the above subdivision of the
geodesic $\gamma$, let $\bmw = (\bmw_0,\ldots,\bmw_{M-1})$ be a word in the
indices $\alpha$ and $j$ for the cone and interior localizers respectively,
i.e., either $\bmw_m = \alpha \in \{1,\ldots,N\}$ or
$\bmw_m = j \in \{1,\ldots,J\}$ for each $m = 0,\ldots,M-1$.  We write $\bm{W}$
for the collection of all such words.  For each time $t \in \bbR$, the operator
\begin{equation*}
  \calU(t) - \sum_{\bmw \in \bm{W}} \calU(t - T_{M-1}) \,
  B_{\bmw_{M-1}} \, \calU(t_{M-1}) \, B_{\bmw_{M-2}} \, \cdots B_{\bmw_1} \,
  \calU(t_1)\, B_{\bmw_{0}} 
\end{equation*}
then maps $\calD_{-\infty}$ to $\calD_\infty$ since the operators $B_\bullet$
make up the microlocal partition of unity above.  Hence, taking the trace of
this operator yields a smooth function on $\bbR$, implying that the
singularities of $\Tr \calU(t)$ are the same as the sum over $\bmw \in \bm{W}$
of the singularities of the microlocalized terms
\begin{equation*}
  \Tr \! \left[ \sqrt{B_{\bmw_{0}}} \, \calU(t - T_{M-1}) \, B_{\bmw_{M-1}} \,
    \calU(t_{M-1}) \, B_{\bmw_{M-2}} \cdots B_{\bmw_1} \, \calU(t_1) \,
     \sqrt{B_{\bmw_{0}}}  \right] ;
\end{equation*}
here we have used cyclicity of the trace to move $ \sqrt{B_{\bmw_{0}}}
$ to the left.
Let us now suppose $\gamma$ is the only such purely diffractive closed geodesic
of period $T$.\footnote{If there are multiple such geodesics, then the same
  argument will show that each contributes its own term of the same form to the
  wave trace.} Given a word $\bmw \in \bm{W}$, for $t$ close to $T$ we have
\begin{equation*}
\sqrt{  B_{\bmw_{M}} }\, \calU(t - T_{M-1}) \, B_{\bmw_{M-1}} \, \calU(t_{M-1}) \,
  B_{\bmw_{M-2}} \cdots B_{\bmw_1} \, \calU(t_1) \, \sqrt{  B_{\bmw_{M}} }  : \calD_{-\infty} \To
  \calD_\infty 
\end{equation*}
unless there is a parametrization of $\gamma$ satisfying
\begin{equation}\label{shadowproperty}
  \gamma(T_m) \in \WF'(B_{\bmw_m}) \ \text{for all $m=0,\ldots,M$},
\end{equation}
where we set $B_{\bmw_{M}}=B_{\bmw_{0}}.$
Therefore, we need only consider terms satisfying \eqref{shadowproperty} in calculating the trace.

For later use, we further refine our microlocal partition of unity to have the
following two convenient properties, which we call Partition
Properties~\ref{assumption:partition1} and~\ref{assumption:partition2}.

\begin{property}
  \label{assumption:partition1}
  Suppose $\psi_\alpha$ and $A_j$ are a cone localizer and an interior localizer
  respectively in the microlocal partition of unity, and let
  $p \in \WF'_\upb(\psi_\alpha)$ and $q \in \WF'(A_j)$ be elements of their
  respective microsupports.\footnote{Here,
    $\WF'_\upb(\psi_\alpha) = \Sbstar_{Y_\alpha} X \cup \WF'\!\left( \psi_\alpha
      \big\vert_{X^\circ} \right)$;
    the use of the b-microsupport is only to be precise.}  If there is a time
  $t \in \left[0, \frac{1}{10} \, \ell \right]$ such that $p \dtilde q$ (or
  equivalently $q \dtilde p$), then either
  \begin{enumerate}[(i)]
  \item $p^\flat \in \left\{ \psi_\alpha \equiv 1 \right\}$, or
  \item $x(p^\flat) > \frac{1}{100} \, \delta_\mathrm{cone}$.
  \end{enumerate}
\end{property}

Thus, the permissible alternatives are as follows:  either the diffractive
(forward or backward) geodesic flowout from $\WF'(A_j)$ all lies close to the
cone point $Y_\alpha$ (i.e., within $\{\psi_\alpha \equiv 1\}$) or it stays
slightly away from the cone point (so $x$ is bounded below).  This may be
ensured by leaving the cone localizers $\psi_\alpha$ fixed and shrinking the
support of the Schwartz kernels of the interior localizers $A_j$'s---controlled
by the constant $\delta_\mathrm{int}$---as needed.

We also make the following further requirement on the partition of unity:  when
flowing out from the microsupport of an interior localizer to that of a conic
localizer and thence to the microsupport of another interior localizer, only one
of these flowouts may involve interaction with a cone point.  To describe this,
we use the notation
\begin{equation}
  \label{eq:regtilde}
  p \regtilde q \Longleftrightarrow \mbox{} {
    \begin{aligned}
      &\text{$p$ and $q$ are connected via a limit\footnotemark\ of the ordinary
        geodesic} \\
      &\text{flow for time $t$ in the cosphere bundle \emph{of the interior}
        $S^* X^\circ$}.
    \end{aligned}
  }
\end{equation}
\footnotetext{In particular, if $p \regtilde q$, then either $p$ or $q$ may be a
  point of the boundary.}

\begin{property}
  \label{assumption:partition2}
  Let $A_{j_1}$ and $A_{j_2}$ be interior localizers and $\psi_\alpha$ a cone
  localizer in our microlocal partition of unity, and let
  $p_1 \in \WF'(A_{j_1})$, $p_2 \in \WF'(A_{j_2})$, and
  $q \in \WF'_\upb(\psi_\alpha)$ be points in their respective microsupports.
  Suppose that there are $t$ and $t'$ in
  $\left[ 0, \frac{1}{10} \, \ell \right]$ such that
  $p_1 \dtilde q \dtilde[t'] p_2$.  Then either
  \begin{equation*}
    p_1 \regtilde q \quad \text{or} \quad q \regtilde[t'] p_2 
  \end{equation*}
(with the same alternative holding for every choice of $q$).
\end{property}
Once again this property can be ensured by shrinking the microsupports of the
$A_j$'s and leaving the $\psi_\alpha$'s fixed; we simply rely on the fact that a
diffractive geodesic cannot undergo two diffractions in time less than $\ell$.

% SECTION 8

% fw-conetrace-8.tex Section 8:  The wave trace along diffractive orbits

\section{The wave trace along diffractive orbits}
\label{sec:diff-wave-trace}

We now prove the \hyperref[theorem:main]{Main Theorem}.  Recall that $\gamma$ is
a strictly diffractive closed geodesic in $\Sbstar X$ undergoing $k$
diffractions through cone points $Y_{\alpha_1},\ldots,Y_{\alpha_k}$, and these
cone points are pairwise nonconjugate to one another.  We denote the length of
$\gamma$ by $L$, and we decompose it as a concatenation of geodesic segments
$\gamma_1, \ldots, \gamma_k$.

\begin{proof}[Proof of the Main Theorem]
  Using the microlocal partition of unity adapted to $\gamma$ developed in
  Section~\ref{section:partition}, we may write $\Tr \calU(t)$ as the sum
  \begin{multline}
    \label{eq:propagator-decomp}
    \sum_{\bmw \in \bm{W}} \Tr \! \left[ \sqrt{B_{\bmw_{0}}} \, \calU(t - T_{M-1}) \,
      B_{\bmw_{M-1}} \, \calU(t_{M-1}) \, B_{\bmw_{M-2}} \cdots B_{\bmw_1} \,
      \calU(t_1) \, \sqrt{B_{\bmw_0}} \right]
  \end{multline}
  modulo a smooth error.  As discussed in the previous section, the terms in the
  sum which are microlocally nontrivial are those for which there is a partition
  of the domain of $\gamma$ such that $\gamma(T_m) \in \WF'(B_{\bmw_m})$ for all
  $m = 0, \ldots, M$.  If each of the microlocalizers $B_{\bmw_m}$ is an
  interior localizer, then we may compute the trace of the resulting
  term using
  the information from Theorem~\ref{thm:amplitude-mult-diff}.  Therefore, we
  need to massage the remaining terms in \eqref{eq:propagator-decomp} into such
  a form, eliminating the cone localizers $\psi_\alpha$ from the expression.
  (Such a technique was previously employed in \cite{Wunsch:Poisson} and in
  \cite{Hillairet:Contribution}.)

  To start, we use cyclicity of the trace to ensure that $B_{\bmw_0}$
  is not a cone localizer.  Since
  $\delta_\mathrm{cone} < \frac{1}{10} \, \ell$ and each
  $t_m < \frac{1}{10} \, \ell$, any summand in \eqref{eq:propagator-decomp} has
  at least two interior localizers appearing between any pair of cone
  localizers.  Thus, a cyclic shift always suffices to ensure
  the outermost terms are $\sqrt{A_j}$'s rather than $\sqrt{\psi_\alpha}$'s, and by
  reparametrizing $\gamma$ we may assume the propagators are evaluated as in
  \eqref{eq:propagator-decomp}.

  We now deal with the internal copies of the cone localizers $\psi_\alpha$,
  which appear as factors
  \begin{equation*}
    A_{j_2} \, \calU(t_{m+1}) \, \psi_\alpha \, \calU(t_m) \, A_{j_1}
  \end{equation*}
  corresponding to an internal diffraction.  (There are also terms
  where one of the $A_j$ is replaced by $\sqrt{A_j}$ and they are
  treated in exactly the same way.)  To simplify these terms and
  eliminate the cone localizers $\psi_\alpha$, we appeal to the properties built
  into our microlocal partition of unity.  If Partition
  Property~\ref{assumption:partition1}(i) holds for
  $(q,t) \in \WF'(A_{j_2}) \times \{t_{m+1}\}$ or
  $(q,t) \in \WF'(A_{j_1}) \times \{t_{m}\}$, then we may replace $\psi_\alpha$
  by the identity operator in the factor:
  \begin{equation*}
    A_{j_2} \, \calU(t_{m+1}) \, \psi_\alpha \, \calU(t_m) \, A_{j_1} \equiv
    A_{j_2} \, \calU(t_{m+1}) \, \Id \, \calU(t_m) \, A_{j_1} \mod{\calC^\infty} .
  \end{equation*}
  On the other hand, if both of these compositions fall under Partition
  Property~\ref{assumption:partition1}(ii), then the points in the support of
  $\psi_\alpha$ which are diffractively related to the projections of
  $\WF'(A_{j_1})$ and $\WF'(A_{j_2})$ to the base are at distance greater than
  $\frac{1}{100} \, \delta_{\mathrm{cone}}$ from the boundary. Then by Partition
  Property~\ref{assumption:partition2}, at most \emph{one} of these compositions
  can involve a diffractive interaction with the boundary.  Suppose the
  $\calU(t_m)$ factor only propagates within the interior.  Then there
  exists $Q \in \Psi_c^0(X^\circ)$ where
  $Q$ is microlocally the identity on the time-$t_m$ geodesic flowout
  of $\WF' A_{j_1}.$  Thus
$$
 A_{j_2} \, \calU(t_{m+1}) \, \psi_\alpha \, \calU(t_m) \, A_{j_1}
\equiv  A_{j_2} \, \calU(t_{m+1}) \, \psi_\alpha \, Q\, \calU(t_m) \, A_{j_1}\mod{\calC^\infty}.
$$
 We write
  \begin{equation*}
    A_{j_2} \, \calU(t_{m+1}) \, \psi_\alpha \, Q\, \calU(t_m) \, A_{j_1}= A_{j_2}
    \, \calU(t_{m+1} + t_m) \left[ \calU(-t_m) \, \psi_\alpha \, Q \, \calU(t_m)
    \right] A_{j_1}   
  \end{equation*}
  and by applying the Egorov Theorem over $X^\circ$, we conclude
  $\left[ \calU(-t_m) \, \psi_\alpha\, Q \, \calU(t_m) \right] A_{j_1}$ is a
  pseudodifferential operator $\widetilde{A}_{j_1}$ in $\Psi^0_\upc(X^\circ)$.
  Moreover, its principal symbol is
  $\sigma\big(\widetilde{A}_{j_1}\big) = \sigma(A_{j_1}) \cdot (G^{t_m})^* \! \left[
    \sigma(\psi_\alpha) \right]$.
  Hence, the microlocalized propagator collapses into an expression of the same
  form:
  \begin{equation*}
    A_{j_2} \, \calU(t_{m+1}) \, \psi_\alpha \, \calU(t_m) \, A_{j_1} = A_{j_2}
    \, \calU(t_{m+1}) \, \Id \, \calU(t_m) \, \widetilde{A}_{j_2} . 
  \end{equation*}
  In particular, the product of the principal symbols of these expressions
  remains unchanged.  An analogous argument holds if $\calU(t_{m+1})$ is the
  factor propagating within the interior.

  Proceeding in this fashion, we may replace \eqref{eq:propagator-decomp} with
  an analogous sum
  \begin{equation}
    \label{eq:propagator-decomp-new}
    \sum_{\bmw \in \bm{W}} \Tr \! \left[ \sqrt{\widetilde{B}_{\bmw_{0}}} \, \calU(t -
      T_{M-1}) \, \widetilde{B}_{\bmw_{M-1}} \, \calU(t_{M-1}) \,
      \widetilde{B}_{\bmw_{M-2}} 
      \cdots \widetilde{B}_{\bmw_1} \, \calU(t_1) \,
     \sqrt{ \widetilde{B}_{\bmw_0}} \right] 
  \end{equation}
  in which no cone localizers appear.  Thus, each of the pseudodifferential
  operators $\widetilde{B}_{\bmw_m}$'s are either interior localizers $A_j$,
  modified interior localizers $\widetilde{A}_j$'s having essentially the same
  properties, or copies of the identity operator $\Id$.  In particular, the
  principal symbols of the pseudodifferential operators in each term still sum
  to $1$ when evaluated along the geodesic, i.e.,
  \begin{equation*}
    \sum_{\bmw \in \bm{W}} \prod_{m = 0}^{M-1} \sigma\!\left(
      \widetilde{B}_{\bmw_m} \right) \! \left( \gamma(T_m) \right) = 1 .
  \end{equation*}
  As each term in \eqref{eq:propagator-decomp-new} falls under the description
  of Theorem~\ref{thm:amplitude-mult-diff}, all that remains is to compute the
  trace of each term.

  We now perform the method of stationary phase in the base variables of the
  expression for the propagator in Theorem~\ref{thm:amplitude-mult-diff},
  restricted to the diagonal and integrated over the base manifold $X$---we do
  not apply stationary phase in the phase variable $\xi$.  In order to compute
  the trace of each term, we once again use cyclicity to permute so that the
  outermost factors of $\widetilde{B}_{w_m}$ are not identity terms but rather
  are interior microlocalizers $\sqrt{A_j}$ or $\sqrt{\widetilde{A}_j}$.  (Note that this
  requires switching the roles of $t-T_M$ and of one of the $t_j$'s, which
  amounts to a simple relabeling.)  The resulting phase function is then
  \begin{equation*}
    \phi \defeq \left[ \dist_g^{\gamma_k}(z,Y_{\alpha_{k-1}}) + \sum_{j =
        2}^{k-1} \dist_g^{\gamma_{j-1}}(Y_{\alpha_j},Y_{\alpha_{j-1}}) +
      \dist_g^{\gamma_0}(Y_{\alpha_1},z) - t \right] \xi ,
  \end{equation*}
  where the variable $z$ is supported in a compact subset of the interior of $X$
  owing to the support of the amplitude.  The phase $\phi$ is critical in the
  $z$-variable precisely when
  \begin{equation*}
    \del_z \! \left[ \dist_g^{\gamma_k}(z,Y_{\alpha_{k-1}}) +
      \dist_g^{\gamma_0}(Y_{\alpha_1},z) \right] = 0 ,
  \end{equation*}
  and this forces $z$ to lie along the unique geodesic $\gamma_k^\flat$
  connecting the cone points $Y_{\alpha_1}$ and $Y_{\alpha_k}$, as before.
  (Note that $\gamma_0 = \gamma_k$ in this calculation.)  Unlike the previous
  calculations, though, there is no integration in the phase variables, and as a
  result there is no particular point along the segment which is fixed by
  stationarity.  Thus, the entire segment of $\gamma_k^\flat$ within the support
  of the amplitude is part of the critical set, i.e., it is a Morse-Bott
  stationary manifold for the phase.  Therefore, we applying this version of the
  method of stationary phase in Fermi normal coordinates $(\nu,\ell)$ along
  $\gamma_k^\flat$, and we compute that
  \begin{equation*}
    \Tr \! \left[ \sqrt{\widetilde{B}_{\bmw_{M}}} \, \calU(t - T_{M-1}) \,
      \widetilde{B}_{\bmw_{M-1}} \, \calU(t_{M-1}) \, \widetilde{B}_{\bmw_{M-2}}
      \cdots \widetilde{B}_{\bmw_1} \, \calU(t_1) \, \sqrt{\widetilde{B}_{\bmw_0}}
    \right]
  \end{equation*}
  has the oscillatory integral representation
  \begin{equation}
    \label{eq:monomial-trace-oscil-int-rep-0}
    \int_{\bbR_\xi} e^{ - i (t - L) \xi} \, e^{\frac{i\pi(n-1)}{4}} \,
    (2\pi)^{\frac{n+1}{2}} \, i^{-m_{\gamma_k}} \left\{ \int_{\ell =
        0}^{\dist_g^{\gamma_k}(Y_{\alpha_1},Y_{\alpha_k})}
      b(t;\nu_*,\ell;\nu_*,\ell;\xi) \right\} d\xi ,
  \end{equation}
  where $b$ is the amplitude \eqref{eq:amplitude-mult-diff} from
  Theorem~\ref{thm:amplitude-mult-diff} and $\nu_*$ is the critical point in the
  $\nu$-variables.  In particular, the signature and Hessian factors coming from
  criticality in the $\nu$-variables are the same as before.  We may thus write
  \eqref{eq:monomial-trace-oscil-int-rep-0} as
  \begin{equation}
    \label{eq:monomial-trace-oscil-int-rep}
    \int_{\bbR_\xi} e^{- i (t - L) \xi} \, a(t,\xi) \, d\xi ,
  \end{equation}
  where the amplitude $a \in S^{- \frac{k(n-1)}{2}}_\upc(\bbR_t)$ is
  \begin{multline}
    \label{eq:monomial-trace-amplitude}
    (2\pi)^{\frac{kn}{2}} \, e^{\frac{i k \pi (n-3)}{4}} \cdot \chi(\xi) \,
    \xi^{-\frac{k(n-1)}{2}} \cdot \left\{ \int_{\ell =
        0}^{\dist_g^{\gamma_k}(Y_{\alpha_1},Y_{\alpha_k})}
      \bm{a}(\nu_*,\ell;\nu_*,\ell;\xi) \right\} \\
    \mbox{} \times \left[ \prod_{j = 1}^k i^{-m_{\gamma_j}} \cdot
      \bmD_{\alpha_j}(q_j,q_j') \cdot
      \dist_g^{\gamma_{j}}(Y_{\alpha_{j+1}},Y_{\alpha_{j}})^{-\frac{n-1}{2}}
      \cdot \Theta^{-\frac{1}{2}}(Y_{\alpha_{j}} \to Y_{\alpha_{j+1}}) \right]
  \end{multline}
  modulo elements of $S^{-\frac{k(n-1)}{2} - \frac{1}{2} + 0}$ and we treat the
  labels for $\gamma_j$ and $Y_{\alpha_j}$ as being cyclic in $\{1,\ldots,k\}$.
  Adding up these contributions yields the asserted expression for the trace
  once we note that the combined amplitudes $\bm{a}$ sum to the amplitude of the
  identity operator.  (Note that in this fashion we integrate over the geodesic just
  once, if it happens to be an iterate of a shorter one, thus
  accounting for the factor of $L_0$ rather than $L.$)  This concludes the proof.
  %
  % The stationary points again occur where $\nabla (x_1+x_k)=0.$ This time,
  % however, we have no fiber integral and so this geodesic is a whole
  % (Morse-Bott) stationary manifold; note also that the support of the
  % amplitude
  % lies over $X^\circ$.  We employ geodesic normal coordinates about
  % $\gamma_{k},$ the final segment of the closed geodesic, connecting $Y_{i_k}$
  % to $Y_{i_1},$ to see that we again obtain the same Hessian factor in the
  % normal directions as we did in Theorem~\ref{thm:amplitude-mult-diff}, as
  % well
  % as the same phase factors.  The resulting phase is of course just
  % \begin{equation*}
  %   -t+x_1(p)+x_{k}(p)-\sum_{j=1}^{k-1} L_j=-t+L,
  % \end{equation*}
  % since for $p$ on the geodesic segment $\gamma_k$ we have
  % $x_1(p)+x_k(p)=L_k.$
  % Thus we end up with an oscillatory integral with the desired phase, and
  % where
  % the amplitude is, to leading order, the integral along the (lifted) geodesic
  % \begin{equation}\label{microlocalizedtrace}
  %   \begin{aligned}
  %     & (2\pi)^{kn/2} e^{ik(n-3)\pi/4}\chi(\xi)\xi^{-k(n-1)/2}\prod_{j=1}^k
  %     i^{-m_{\gamma_j}} \D_j \W_j \\ &\times \int
  %     \sigma(C_{\word_0})[\gamma(\ell)]
  %     \sigma(C_{\word_1})[\gamma(\ell+s_1)]\dots
  %     \sigma(C_{\word_M})[\gamma(\ell+s_M)] \sigma(C_{\word_{M+1}})
  %     [\gamma(\ell))]\, d\ell\\
  %     &= (2\pi)^{kn/2} e^{ik (n-3)\pi /4}\chi(\xi) \xi^{-k(n-1)/2} L
  %     \prod_{j=1}^k i^{-m_{\gamma_j}} \D_j \W_j
  %   \end{aligned}
  % \end{equation}
  % Adding up these contributions yields the asserted expression for the trace.
\end{proof}

\appendix

% APPENDIX A

% fw-conetrace-8.tex Appendix A:  Lagrangian distributions and their amplitudes

\section{Lagrangian distributions and their amplitudes}
\label{sec:lagrangian-dists-amplitudes}

In this appendix, we briefly collect the various facts we need in the body of
the work about Lagrangian distributions and their differing kinds of amplitudes.
The bulk of these facts are due to H\"ormander---see, e.g. \cite{Hormander:FIO1}, but we approach Lagrangian
distributions through the iterated regularity perspective of Melrose
\cite{Melrose:Transformation} (see \cite{HorIV} for a unified
exposition). 

Let $Z$ be a smooth $n$-dimensional manifold, and suppose
$\Lambda \subseteq T^*Z \setminus \vec{0}$ is a fiber-homogeneous Lagrangian
submanifold of its cotangent bundle with the zero section removed.  Writing
$\mathfrak{X}^s_\loc$ for either the local Sobolev space $H^s_\loc(X)$ or the
local Besov space $B^s_{2,\infty,\loc}(X)$, we say $u$ is a Lagrangian
distribution with respect to $\Lambda$ of class $I \mathfrak{X}^s(Z,\Lambda)$ if
and only if
\begin{equation*}
  A_N \cdots A_1 u \in \mathfrak{X}^s_\loc \ \text{for all finite
    compositions of} \ A_j
  \in \Psi^1(Z) \ \text{with} \ \sigma_1(A_j) \big\vert_\Lambda \equiv 0 .
\end{equation*}
In particular, the property of a distribution $u \in \calC^{-\infty}(Z)$ being
Lagrangian with respect to $\Lambda$ and $\mathfrak{X}^s$ is entirely local.

When $u$ is an element of $\IB_{2,\infty}^{-m - \frac{n}{4}}(Z,\Lambda)$, we
also write $u \in I^m(Z,\Lambda)$; these are H\"ormander's classes of Lagrangian
distributions.  As shown in \cite{HorIV}, one of
their characterizing properties is the existence of local oscillatory integral
representations:  given $u \in I^m(Z,\Lambda)$ there is a covering of $Z$ by open
sets $U$ on which
\begin{equation}
  \label{eq:Lagn-oscillatory-integral-presentation}
  u(x) = \int_{\bbR^N_\theta} e^{i\phi(z,\theta)} \, a(z,\theta) \, d\theta .
\end{equation}
Here, $\phi$ is a phase function locally parametrizing $\Lambda$ over
$U \subseteq Z$, and the amplitude $a$ is a Kohn-Nirenberg symbol in the class
$S^{m-\frac{n+2N}{4}}(U \times \bbR^N_\theta)$, i.e., a smooth function on
$U \times \bbR^N_\theta$ satisfying the estimates
\begin{equation*}
  \left\| \left< \theta \right>^{-m + |\beta|} D_z^\alpha D_\theta^\beta
    a(z,\theta) \right\|_{L^\infty(U \times \bbR^N)} \leqslant C_{\alpha,\beta}
\end{equation*}
for all multi-indices $\alpha$ and $\beta$, where
$\left< \theta \right> \defeq \left( 1 + |\theta|^2 \right)^\frac{1}{2}$.

When $u$ is an element of a Sobolev-based iterated regularity space
$\IH^s(Z,\Lambda)$, it also has an oscillatory integral representation like
\eqref{eq:Lagn-oscillatory-integral-presentation}, but in this case the
amplitude $a$ is an \emph{$L^2$-based symbol} of class
$S^{-s} L^2 \! \left( U \times \bbR^n_\theta \right)$.  In the general case, we
define this class of symbols as follows.

\begin{definition}
  \label{def:L2-symbols}
  The space $S^m L^2 \! \left(\bbR^n_z \times \bbR^N_\theta \right)$ of
  \emph{$L^2$-based symbols of order $m$} on $\bbR^n_z \times \bbR^N_\theta$ are
  those smooth functions $a$ satisfying the estimates
  \begin{equation*}
    \label{eq:L2-symbols}
    \left\| \left< \theta \right>^{-m + |\beta|} D_z^\alpha D_\theta^\beta
      a(z,\theta) \right\|_{L^2(\bbR^n \times \bbR^N)} \leqslant C_{\alpha,\beta}
  \end{equation*}
  for all multi-indices $\alpha$ and $\beta$.
\end{definition}

The Sobolev embedding theorem shows that these $L^2$-based symbols give an
alternative filtration of the Kohn-Nirenberg symbol spaces:
\begin{equation}
  \label{KNL2}
  S^m \! \left(\bbR^n_z \times \bbR^N_\theta \right) \subsetneq S^{m +
    \frac{N}{2} + 0} L^2 \! \left(\bbR^n_z \times \bbR^N_\theta
  \right) \subsetneq S^{m + 0} \! \left( \bbR^n_z \times \bbR^N_\theta \right)
  .
\end{equation}
At the level of iterated regularity, this chain of inclusions just reflects the
similar inclusions of regularity spaces
$B^s_{2,\infty}(\bbR^n) \supsetneq H^{s+0}(\bbR^n) \supsetneq
B^{s+0}_{2,\infty}(\bbR^n)$, in turn yielding the chain of inclusions
\begin{equation*}
  \IB^s_{2,\infty}(\bbR^n,\Lambda) \supsetneq \IH^{s+0}(\bbR^n,\Lambda)
  \supsetneq \IB^{s+0}_{2,\infty}(\bbR^n,\Lambda) 
\end{equation*}
Furthermore, we may use \eqref{KNL2} to conclude the following rule for
multiplication of $L^2$-based symbols:
\begin{equation}
  \label{eq:L2-symbol-multiplication}
  \text{$a \in S^m L^2 \! \left( \bbR^n_z \times \bbR^N_\theta \right)$ and $b
    \in S^{m'} L^2 \! \left( \bbR^n_z \times \bbR^N_\theta \right)$ 
    $\Longrightarrow$ $ab \in S^{m + m' - \frac{N}{2} + 0} \! \left( \bbR^n_z \times
      \bbR^N_\theta \right)$} .
\end{equation}

We now state and prove our main technical lemma relating the two types of
amplitudes, used in the proof of Lemma \ref{lemma:reduction} above.

\begin{lemma}
  \label{lemma:averagedsymbol}
  Let
  $a \in S^m_\upc L^2 \! \left( \bbR^{p+q}_{(w,z)} \times \bbR^N_\theta \right)$
  be an $L^2$-symbol which is compactly supported in the base variables $(w,z)$.
  Suppose that there exist $\ell \in \bbR$ and $\delta>0$ such that for all compactly supported
  $b \in S^\ell_\upc L^2 \! \left( \bbR^q_z \times \bbR^N_\theta \right),$
  \begin{equation*}
    \int_{\bbR^q_z} a(w,z,\theta) \, b(z,\theta) \, dz  \in S^{m + \ell - \frac{N}{2} - \delta} L^2 \!  \left( \bbR^p_w \times
    \bbR^N_\theta \right)
  \end{equation*}
 Then
  \begin{equation*}
a\in    S^{m - \delta + 0}_\upc L^2 \! \left(\bbR^{p+q}_{(w,z)} \times \bbR^N_\theta
    \right) \subsetneq S^{m-\frac{N}{2}-\delta+0}_\upc \! \left(
      \bbR^{p+q}_{(w,z)} \times \bbR^N_\theta \right).
  \end{equation*}
  % Let $a \in S_{L^2}^k(\RR^p_w\times \RR^q_z \times \RR^m_\xi)$ (with $w,z$
  % being base and $\xi$ fiber variables) be compactly supported.  Suppose that
  % there exists $s$ such that for all compactly supported $b \in
  % S_{L^2}^s(\RR^q_z \times \RR^m_\xi)$ we have $\int ab \, dz \in
  % S_{L^2}^{k+s-m/2-\alpha}(\RR^p_w \times \RR^k_\xi).$ Then $a \in
  % S^{k-\alpha+0}(\RR^p_w\times \RR^q_z \times \RR^m_\xi).$
\end{lemma}

\begin{proof}
  For simplicity we fix $\ell=0$; the general proof is similar.

  We begin by showing the undifferentiated estimate
  $a \in \smallang{\theta}^{m-\delta+0} L^2(\bbR^{n+N})$.  Let $\varphi$ be an
  element of $\calC^\infty_\upc(\bbR^q_z)$, and choose
  $b(z,\theta) = \varphi(z) \left< \theta \right>^{-\frac{N}{2} + \delta -
    \varepsilon}$.
  Noting that $\left< \theta \right>^{-(m - \delta) - \varepsilon}$ is an
  element of $S^{m-\delta-\frac{N}{2}} L^2(\bbR^N_\theta)$ for any
  $\varepsilon > 0$, we observe that the map
  \begin{equation*}
    \varphi \longmapsto c(w,\theta) \defeq \left< \theta \right>^{-(m-\delta) -
      \varepsilon} \int_{\bbR^q_z} a(w,z,\theta) \, \varphi(z) \, dz
  \end{equation*}
  is continuous from $H^\infty_\upc(\bbR^q_z)$ to
  $S^0 L^2 \! \left( \bbR^p_w \times \bbR^N_\theta \right)$.  Thus, by the
  Closed Graph Theorem there is an $A \in \bbZ_+$ such that
  \begin{equation}
    \label{eq:bound1}
    \left\| c \right\|_{L^2(\bbR^p_w \times \bbR^N_\theta)} \lesssim \left(
      \sum_{|\alpha| \leqslant A} \left\| D^\alpha_z \varphi
      \right\|^2_{L^2(\bbR^q_z)} \right)^\frac{1}{2} .
  \end{equation}
  This implies that
  $\left< \theta \right>^{-(m-\delta)-\varepsilon} a(w,z,\theta)$ must be an
  element $L^2(\RR^p_w\times \RR^N_\theta)$ with values in the dual to the Hilbert space whose norm is defined by the
  right-hand side of \eqref{eq:bound1}, i.e.,
  \begin{equation}
    \label{eq:bound2}
    a \in \left< \theta \right>^{m - \delta + \varepsilon} L^2 \! \left(
      \bbR^p_w \times \bbR^N_\theta ; H^{-A}(\bbR^q_z) \right) .
  \end{equation}
  On the other hand, we know a priori that $a$ satisfies symbol estimates and
  may, in particular, be differentiated without loss:
  \begin{equation}
    \label{eq:bound3}
    a \in \left< \theta \right>^{m} L^2 \! \left( \bbR^p_w
      \times \bbR^N_\theta ; H^\infty(\bbR^q_z) \right) .
  \end{equation}
  Interpolating \eqref{eq:bound2} and \eqref{eq:bound3} yields that $a$ is an
  element of
  \begin{equation*}
    \left< \theta \right>^{m - \delta + 2 \varepsilon} L^2 \! \left(
      \bbR^p_w \times \bbR^N_\theta ; H^M(\bbR^q_z) \right)
  \end{equation*}
  for every $M \in \bbZ_+$ and every $\varepsilon > 0$, and this is precisely
  the first estimate required to show that $a$ is an element of
  $S^{m - \delta + 0} L^2 \! \left( \bbR^{p+q}_{(w,z)} \times \bbR^N_\theta
  \right)$.
  The higher-order symbol estimates are proved analogously, estimating the
  higher-order symbol norms in \eqref{eq:bound1} instead of just the $L^2$-norm.
  The lemma then follows once we use the inclusion from \eqref{KNL2}.
\end{proof}

% REFERENCES

\section*{References}
\label{sec:references}

\end{document}